%% file: survey.tex
\documentclass[12pt, oneside]{amsart}

\usepackage[utf8]{inputenc}
\usepackage[T1]{fontenc}

\usepackage{amssymb}
\usepackage{amsthm}
\usepackage{amsmath}
\usepackage[only,Yup]{stmaryrd} 
\usepackage{mathrsfs}    
\usepackage[new]{old-arrows}   

\usepackage[a4paper,left=2cm,top=3cm,right=2cm,bottom=2cm]{geometry}

\usepackage[all,cmtip]{xy}

\usepackage{enumerate}

\usepackage{indentfirst}

\usepackage{hyperref}
\hypersetup{colorlinks=true,allcolors=black}  

\usepackage{pstricks}
\usepackage{graphicx}

\usepackage{multicol}

\DeclareMathOperator{\codim}{codim}

\DeclareMathOperator{\rank}{rank}
\DeclareMathOperator{\symrank}{symrank}

\DeclareMathOperator{\etale}{Et}
\DeclareMathOperator{\vol}{vol}
\DeclareMathOperator{\Ad}{Ad}

\DeclareMathOperator{\id}{Id}
\DeclareMathOperator{\Zero}{Zero}
\DeclareMathOperator{\ric}{Ric}

\DeclareMathOperator*{\spannn}{span}
\DeclareMathOperator{\sym}{S}
\DeclareMathOperator{\poin}{P}
\DeclareMathOperator{\tub}{Tub}
\DeclareMathOperator{\hol}{Hol}
\DeclareMathOperator{\res}{res}

\mathchardef\ordinarycolon\mathcode`\:
\mathcode`\:=\string"8000
\begingroup \catcode`\:=\active
  \gdef:{\mathrel{\mathop\ordinarycolon}}
\endgroup

\begin{document}

\title{Leaf closures of Riemannian foliations: a survey on topological and geometric aspects of Killing foliations}

\author{Marcos M.~Alexandrino}
\address{Instituto de Matemática e Estatística, Universidade de São Paulo, Rua do Matão 1010, 05508-090, São Paulo - SP, Brazil}
\email{malex@ime.usp.br}
\thanks{The first author was supported by grant \#2016/23746-0, São Paulo Research Foundation (FAPESP)}

\author{Francisco C.~Caramello Jr.}
\address{Departamento de Matemática, Universidade Federal de Santa Catarina, R. Eng. Agr. Andrei Cristian Ferreira, 88040-900, Florianópolis - SC, Brazil}
\email{francisco.caramello@ufsc.br}
\thanks{The second author was supported by grant \#2018/14980-0, São Paulo Research Foundation (FAPESP)}


\newenvironment{proofoutline}{\proof[Proof outline]}{\endproof}
\theoremstyle{definition}
\newtheorem{example}{Example}[section]
\newtheorem{definition}[example]{Definition}
\newtheorem{remark}[example]{Remark}
\theoremstyle{plain}
\newtheorem{proposition}[example]{Proposition}
\newtheorem{theorem}[example]{Theorem}
\newtheorem{lemma}[example]{Lemma}
\newtheorem{corollary}[example]{Corollary}
\newtheorem{claim}[example]{Claim}
\newtheorem{conjecture}[example]{Conjecture}
\newtheorem{thmx}{Theorem}
\renewcommand{\thethmx}{\Alph{thmx}} 
\newtheorem{corx}{Corollary}
\renewcommand{\thecorx}{\Alph{corx}} 

\newcommand{\dif}[0]{\mathrm{d}}
\newcommand{\od}[2]{\frac{\dif #1}{\dif #2}}
\newcommand{\pd}[2]{\frac{\partial #1}{\partial #2}}
\newcommand{\dcov}[2]{\frac{\nabla #1}{\dif #2}}
\newcommand{\proin}[2]{\left\langle #1, #2 \right\rangle}
\newcommand{\f}[0]{\mathcal{F}}
\newcommand{\g}[0]{\mathcal{G}}
\newcommand{\metric}{\ensuremath{\mathrm{g}}}

\begin{abstract}
A smooth foliation is Riemannian when its leaves are locally equidistant. The closures of the leaves of a Riemannian foliation on a simply connected manifold, or more generally of a Killing foliation, are described by flows of transverse Killing vector fields. This offers significant technical advantages in the study of this class of foliations, which nonetheless includes other important classes, such as those given by the orbits of isometric Lie group actions. Aiming at a broad audience, in this survey we introduce Killing foliations from the very basics, starting with a brief revision of the main objects appearing in this theory, such as pseudogroups, sheaves, holonomy and basic cohomology. We then review Molino's structural theory for Riemannian foliations and present its transverse counterpart in the theory of complete pseudogroups of isometries, emphasizing the connections between these topics. We also survey some classical results and recent developments in the theory of Killing foliations. Finally, we review some topics in the theory of singular Riemannian foliations and discuss singular Killing foliations.
\end{abstract}

\maketitle
\setcounter{tocdepth}{1}
\tableofcontents

\section{Introduction}

A foliation on a Riemannian manifold is called Riemannian if its leaves are locally equidistant. Alternatively, the leaves of a Riemannian foliation are locally defined by fibers of a Riemannian submersion. These objects, first presented by B.~Reinhart in \cite{reinhart}, form a very relevant class of foliations, whose research has been quite active since their introduction \cite[Appendix D]{tondeur}. As noted by G.~Thorbergsson in his survey \cite{SurveyThorbergsson2010}, in the last two decades the theory of singular Riemannian foliations started to play an important role in the theory of submanifolds and isometric actions. In addition, singular Riemannian foliations appear naturally in all non-compact spaces of non-negative curvature, having played a fundamental role in the proof of the smoothness of the metric projection onto the soul, as noted in the work of B.~Wilking \cite{Wilking}.

There is a rich structural theory for Riemannian foliations, due mainly to P.~Molino, that asserts, among other results, that a complete Riemannian foliation $\f$ admits a locally constant sheaf of Lie algebras of germs of local transverse Killing vector fields $\mathscr{C}_{\f}$ whose action describes the dynamics of $\f$, in the sense that for each leaf $L_x\in\f$ one has
$$T_x\overline{L}_x=\{X_x\ |\ X\in(\mathscr{C}_{\f})_x\}\oplus T_xL_x,$$
where $\overline{L}_x$ denotes the closure of $L_x$. Using this one verifies that the partition $\overline{\f}:=\{\overline{L}\ |\ L\in \f\}$ of $M$ is a singular foliation, meaning that it is a smooth partition into embedded submanifolds of varying dimension.

In this work we are primarily interested in the so-called Killing foliations, that is, those Riemannian foliations which are complete an whose Molino sheaf $\mathscr{C}$ is \emph{globally} contant. In other words, for a Killing foliation $\f$ there exists transverse Killing vector fields $X_1,\dots,X_d$ such that $T\overline{\f}=T\f\oplus\langle X_1,\dots, X_d \rangle$. This class of foliations includes Riemannian foliations on simply-connected manifolds and foliations given by orbits of isometric Lie group actions. This motivates the study of the class of Killing foliations, since it contains important subclasses of Riemannian foliations whilst presents relevant technical advantages, in comparison to general Riemannian foliations.

The main goal of this article is to survey the classical theory of Riemannian and Killing foliations, including Molino's structural theory and the pseudogroup approach to the transverse geometry of these foliations due mostly to A.~Haefliger, and present some recent developments on Killing foliations via a deformation technique. In addition, we present the basics on singular Riemannian foliations and introduce the concept of \emph{singular} Killing foliations.

This article is organized as follows. In Section \ref{section: foliations} we introduce the basics of foliation theory and transverse geometry, including the language of pseudogroups to treat holonomy and the notion of basic cohomology. In Section \ref{section: riemann foliations} we define Riemannian foliations and see some examples and classical results, including the structural theory for pseudogroups of local isometries. For this end, we also briefly review the basics of sheaf theory in this section. Next, we survey Molino's structural theory for Riemannian foliations in Section \ref{subsection: Molino Theory}, establishing some relations of it with the structural theory for pseudogroups of isometries. Section \ref{section: killing foliations} introduces Killing foliations and presents its main examples. This section also brings a deformation technique for Killing foliations that allows one to deform such a foliation into a Riemannian foliation with all leaves closed, whilst some topological and geometric transverse properties are mantained. Sections \ref{section: transverse topology of killing} and \ref{section: transverse geometry of Killing} survey recent applications of these techniques, which allows one to reduce the study of the transverse geometry of these foliations to classical geometry and topology of orbifolds. We then move to the second part of this paper, consisting of singular foliations. In Section \ref{section: singular riemannian foliations} we revisit the concept of singular Riemannian foliation as a natural generalization of the regular case, and introduce some of the technical machinery from this area. After that, Section \ref{section: molino conjecture and its proof} is dedicated to survey the recent results concerning the proof of Molino's conjecture and introduce the analog notion of the Molino sheaf in the singular setting. Finally, in Section \ref{section: singular Killing} we propose the concept of singular Killing foliations and see that this class contains relevant subclasses (such as that of homogeneous Riemannian foliations), which motivates its study.

\section{Foliations}\label{section: foliations}

Let $M$ be a smooth $n$-dimensional connected manifold. A \textit{regular foliation} of $M$ is a partition $\f$ of $M$ into $p$-dimensional, connected, immersed submanifolds, called \textit{leaves}, such that the module $\mathfrak{X}(\f)$ of smooth vector fields that are tangent to the leaves is transitive on each leaf. This means, more precisely, that for each $L\in\f$ and each $x\in L$ one can find smooth vector fields $X_i$ whose values at $x$ form a basis for $T_xL$. We denote the distribution defined by the tangent spaces of the leaves by $T\f$ and the leaf containing $x$ by $L_x$. The number $q=n-p$ is the \textit{codimension} of $\f$. In Section \ref{section: singular riemannian foliations} we will introduce \textit{singular} foliations, which drop the requirement that all leaves have the same dimension. Until then we will often omit the word ``regular'' when referring to regular foliations.

There are several equivalent definitions for regular foliations (see for instance \cite[Section 1.2]{mrcun}). Here we recall the following one, which will be specially useful. A regular foliation $\f$ is equivalently defined by an open cover $\{U_i\}_{i\in I}$ of $M$, submersions $\pi_i:U_i\to S_i$, with $S_i\subset\mathbb{R}^q$ open, and diffeomorphisms $\gamma_{ij}:\pi_j(U_i\cap U_j)\to\pi_i(U_i\cap U_j)$ satisfying
$$\gamma_{ij}\circ\pi_j|_{U_i\cap U_j}=\pi_i|_{U_i\cap U_j}$$
for all $i,j\in I$. The collection $(U_i,\pi_i,\gamma_{ij})$ is a \textit{Haefliger cocycle} representing $\f$ and each $U_i$ is a \textit{simple open set} for $\f$ (see Figure \ref{haefligercocycle}). We will assume without loss of generality that the fibers $\pi_i^{-1}(\overline{x})$ are connected, in which case they are called \textit{plaques}. Plaques glue together to form immersed submanifolds, the leaves of $\f$.

\begin{figure}
\centering{
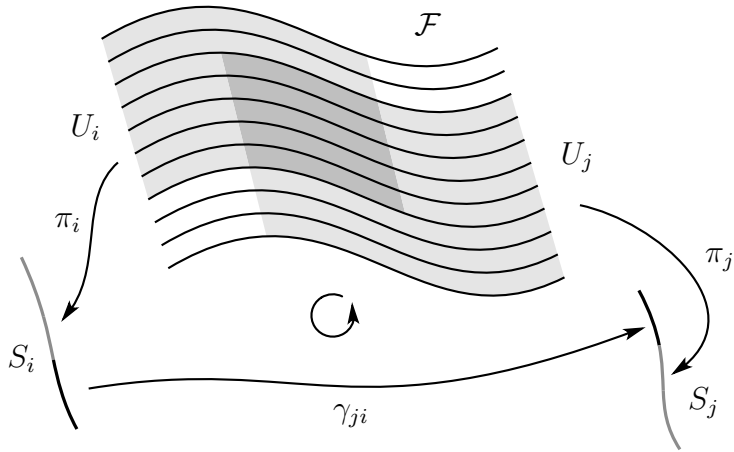}
\caption{A foliation is locally defined by submersions}
\label{haefligercocycle}
\end{figure}

\begin{example}[Pullbacks]\label{exe: pullback foliation}
Let $\f$ be a foliation of $M$ and $f:N\to M$ a smooth map that is transverse to each leaf. Then $f$ defines a foliation $f^*(\f)$ on $N$ as follows. If $(U_i,\pi_i,\gamma_{ij})$ is a cocycle representing $\f$, then $f^*(\f)$ is given by the cocycle $(V_i,\pi_i', \gamma_{ij})$, where $V_i=f^{-1}(U_i)$ and $\pi_i'=\pi_i\circ f|_{V_i}$. Observe that $Tf^*(\f)=\dif f^{-1}(T\f)$ and that $\codim(f^*(\f))=\codim(\f)$.
\end{example}

\begin{example}[Homogeneous foliations]\label{exe: foliated actions}
Lie group actions constitute a main source of foliations. Precisely, recall that when $\mu:G\times M\to M$ is a smooth action, each orbit $Gx$ is the image of an injective immersion $G/G_x\to M$ (see, for instance, \cite[Proposition 3.14]{alex}). Thus, if we suppose that $\dim(G_x)$ is a constant function of $x$, it follows that the connected components of orbits of $G$ decompose $M$ into immersed submanifolds of constant dimension. This decomposition $\f$ is easily seen to be a foliation, because $T_x(Gx)=\dif (\mu_x)_{e}(\mathfrak{g})$, so the fundamental vector fields $V^\#\in\mathfrak{X}(M)$, for $V\in\mathfrak{g}$, induced by the action generate $T\f$, showing that this is an involutive distribution.

A specific example is the following. Consider the flat torus $\mathbb{T}^2=\mathbb{R}^2/\mathbb{Z}^2$. For each $\lambda\in(0,+\infty)$, we have a smooth $\mathbb{R}$-action
$$
\begin{array}{rcl}
\mathbb{R}\times\mathbb{T}^2& \longrightarrow &\mathbb{T}^2\\
(t,[x,y]) & \longmapsto & [x+t,y+\lambda t]
\end{array}$$
with $\dim(\mathbb{R}_{[x,y]})\equiv 0$. The resulting foliation is the \textit{$\lambda$-Kronecker foliation} of the torus, $\f(\lambda)$. Observe that when $\lambda$ is irrational each leaf is dense in $\mathbb{T}^2$, while a rational $\lambda$ yields closed leaves.
\end{example}

When a foliation $\f$ is given by the action of a Lie group we say that $\f$ is \textit{homogeneous}.

\begin{example}[Suspensions]\label{example: suspensions}
Another class of examples of foliations comes from suspensions of homomorphisms, a useful construction originally due to A.~Haefliger \cite{haefliger4}. Let $B$ and $S$ be smooth manifolds, let $h:\pi_1(B,x_0)\to\mathrm{Diff}(S)$ be a group homomorphism and denote by $\rho:\widehat{B}\to B$ the projection of the universal covering space of $B$. On $\widetilde{M}:=\widehat{B}\times S$, the fibers of the second projection $\widetilde{M}\to S$ determine a foliation $\widetilde{\f}$. Define a right action of $\pi_1(B,x_0)$ on $\widetilde{M}$ by setting, for $[\gamma]\in\pi_1(B,x_0)$,
$$(\hat{b},s)\!\cdot\![\gamma]=\left(\hat{b}\!\cdot\![\gamma],h\left([\gamma]^{-1}\right)(s)\right),$$
where $\hat{b}\!\cdot\![\gamma]$ denotes the image of $\hat{b}$ by the deck transformation associated to $[\gamma]$. There is a manifold structure on $M=\widetilde{M}/\pi_1(B,x_0)$ \cite[p. 28]{molino} such that the orbit projection $\pi:\widetilde{M}\to M$ is a covering map and, if $\tau:M\to B$ is given by $\tau(\pi(\tilde{b},t))=\rho(\hat{b})$, then it is the projection of a fiber bundle with total space $M$, base $B$, fiber $S$ and structural group $h(\pi_1(B,x_0))$. The action of $\pi_1(B,x_0)$ preserves the leaves of $\widetilde{\f}$, so projecting through $\pi$ we obtain a foliation $\f$ on $M$ with $\codim(\f)=\dim(S)$, constructed by \textit{suspension of the homomorphism} $h$.

For example, the Kronecker foliation $\f(\lambda)$ (see Example \ref{exe: foliated actions}) can be obtained by suspension of the homomorphism $\pi_1(\mathbb{S}^1,1)\cong\mathbb{Z}\to\mathrm{Diff}(\mathbb{S}^1)$ given by $k\mapsto e^{-2\pi\mathrm{i}\lambda k}$.
\end{example}

As the Kronecker foliation shows, a leaf $L$ of a foliation $\f$ need not to be closed as a subspace of the ambient manifold $M$. We denote the set of leaf closures by $\overline{\f}:=\{\overline{L}\ |\ L\in\f\}$. Understanding $\overline{\f}$ is part of the study of the dynamics of the foliation. In the simple case when $\overline{\f}=\f$, that is, when all the leaves of $\f$ are closed, we say that $\f$ is a \textit{closed} foliation. A submanifold $N\subset M$ is \textit{saturated} if it is a union of leaves or, equivalently, if $N=\pi^{-1}(\pi(N))$, where $\pi:M\to M/\f$ is the projection to the leaf space. We say that $\f$ is \textit{transversely compact} when $M/\overline{\f}$ is compact.

A foliation $(M,\f)$ is \textit{tangentially orientable} if $T\f$ is orientable, and \textit{transversely orientable} if its \textit{normal bundle} $\nu\f:=TM/T\f$ is orientable. In this case, choices of orientations for $T\f$ and $\nu\f$ give, respectively, a \textit{tangential orientation} and a \textit{transverse orientation} for $\f$. It is always possible to choose an orientable finite covering space $\widehat{M}$ of $M$ such that the lifted foliation $\widehat{\f}$ is transversely (and hence also tangentially) orientable \cite[Proposition 3.5.1]{candel}. In terms of a Haefliger cocycle, $\f$ is transversely oriented if and only if there is a cocycle $\{(U_i,\pi_i,\gamma_{ij})\}$ representing $\f$ that satisfies $\det(\dif\gamma_{ij})>0$ as a function on $\pi_j(U_i\cap U_j)$, for all $i,j\in I$.

Let $(M,\f)$ and $(N,\g)$ be foliations. A \textit{foliate morphism} between $(M,\f)$ and $(N,\g)$ is a map $f:M\to N$ that sends leaves of $\f$ into leaves of $\g$. When there is a foliate diffeomorphism $f:M\to N$ (that is, $\f$ is foliate and admits a foliate inverse), the foliations $\f$ and $\g$ are often said to be \textit{congruent}. In particular, we may consider $\f$-foliate diffeomorphisms $f:M\to M$. The infinitesimal counterparts of this notion are the \textit{foliate vector fields} of $\f$, that is, vector fields in the subalgebra
$$\mathfrak{L}(\f)=\{X\in\mathfrak{X}(M)\ |\ [X,\mathfrak{X}(\f)]\subset\mathfrak{X}(\f)\}.$$
These are precisely the fields whose local flows send leaves to leaves. Another characterization is that $X\in\mathfrak{L}(\f)$ if, and only if, for each submersion $\pi:U\to S$ locally defining $\f$ we have that $X|_U$ is $\pi$-related to some vector field $\overline{X}_S\in\mathfrak{X}(S)$ \cite[Section 2.2]{molino}.

The Lie algebra $\mathfrak{L}(\f)$ also has the structure of a module, whose coefficient ring consists of the \textit{basic functions} of $\f$, that is, functions $f\in C^{\infty}(M)$ such that $Xf=0$ for every $X\in\mathfrak{X}(\f)$. We denote this ring by $\Omega^0(\f)$. A smooth function is basic if and only if it is constant on each leaf and also if and only if it factors through each submersion $\pi:U\to S$ locally defining $\f$ to a smooth function on the quotient $S$ \cite[Section 2.1]{molino}.

The quotient of $\mathfrak{L}(\f)$ by the ideal $\mathfrak{X}(\f)$ yields the Lie algebra $\mathfrak{l}(\f)$ of \textit{transverse vector fields}. For $X\in\mathfrak{L}(\f)$ we denote its induced transverse field by $\overline{X}\in\mathfrak{l}(\f)$. Notice that each $\overline{X}$ defines a unique section of $\nu\f$ and that $\mathfrak{l}(\f)$ is also a $\Omega^0(\f)$-module.

\subsection{Holonomy}\label{section: holonomy}

We start this section by recalling the language of pseudogroups. Let $S$ be a smooth manifold. Recall that a \textit{pseudogroup} $\mathscr{H}$ of local diffeomorphisms of $S$ consists of a set of diffeomorphisms $h:U\to V$, where $U$ and $V$ are open sets of $S$, such that

\begin{enumerate}[(i)]
\item $\id_U\in\mathscr{H}$ for any open set $U\subset S$,
\item $h\in\mathscr{H}$ implies $h^{-1}\in\mathscr{H}$,
\item if $h_1:U_1\to V_1$ and $h_2:U_2\to V_2$ are in $\mathscr{H}$, then their composition
$$h_2\circ h_1:h_1^{-1}(V_1\cap U_2)\longrightarrow h_2(V_1\cap U_2)$$
also belongs to $\mathscr{H}$, and
\item if $U\subset S$ is open and $k:U\to V$ is a diffeomorphism such that $U$ admits an open cover $\{U_i\}$ with $k|_{U_i}\in\mathscr{H}$ for all $i$, then $k\in\mathscr{H}$.
\end{enumerate}

The \textit{$\mathscr{H}$-orbit} of $x\in S$ consists of the points $y\in S$ for which there is some $h\in\mathscr{H}$ satisfying $h(x)=y$. The quotient by the corresponding equivalence relation, endowed with the quotient topology, is the \textit{space of orbits} of $\mathscr{H}$, that we denote $S/\mathscr{H}$.

If we have two pseudogroups of local diffeomorphisms $\mathscr{H}$ and $\mathscr{K}$ of $S$ and $T$, respectively, a \textit{smooth equivalence} between $\mathscr{H}$ and $\mathscr{K}$ is a maximal collection $\Phi$ of diffeomorphisms from open sets of $S$ to open sets of $T$ such that $\{\mathrm{Dom}(\varphi)\ |\ \varphi\in\Phi\}$ covers $S$, $\{\mathrm{Im}(\varphi)\ |\ \varphi\in\Phi\}$ covers $T$ and, for all $\varphi,\psi\in\Phi$, $h\in\mathscr{H}$ and $k\in\mathscr{K}$, we have $\psi^{-1}\circ k\circ\varphi\in\mathscr{H}$, $\psi\circ h\circ\varphi^{-1}\in\mathscr{K}$ and $k\circ\varphi\circ h\in\Phi$, whenever these compositions make sense.

The collection of all changes of charts of an atlas $\mathcal{A}$ of a smooth manifold defines a pseudogroup $\mathscr{H}_{\mathcal{A}}$ on the disjoint union of the images of the charts. If $\mathcal{B}$ is a compatible atlas then one has a smooth equivalence $\mathscr{H}_{\mathcal{A}}\cong\mathscr{H}_{\mathcal{B}}$. More generally:

\begin{example}[Orbifolds]\label{exe orbiflds}
An \textit{$n$-dimensional smooth orbifold} is an equivalence class $\mathcal{O}=[(S,\mathscr{H})]$ of pseudogroups of local diffeomorphisms, with $S$ an $n$-dimensional manifold, satisfying that $|\mathcal{O}|:=S/\mathscr{H}$ is Hausdorff and paracompact and each $x\in|\mathcal{O}|$ has a neighborhood $U$ such that $\mathscr{H}|_U$ is generated by a finite collection of diffeomorphisms of $U$. Orbifolds are generalizations of manifolds that appear naturally in many areas of mathematics, for instance as quotients of manifolds by properly discontinuous actions, the so-called \textit{good orbifolds}. We refer to \cite{caramello3}, \cite[Chapter 1]{adem}, \cite[Section 2.4]{mrcun} and \cite{kleiner} to detailed introductions.

Equivalently, an orbifold $\mathcal{O}$ is usually defined, in analogy with the classical definition of manifolds, as a Hausdorff paracompact space $|\mathcal{O}|$ admitting an orbifold atlas. Each chart of this atlas consists of an open subset $\widetilde{U}\subset\mathbb{R}^n$, a finite subgroup $H$ of $\mathrm{Diff}(\widetilde{U})$ and an $H$-invariant map $\phi:\widetilde{U}\to |\mathcal{O}|$ that induces a homeomorphism between $\widetilde{U}/H$ and some open subset $U\subset|\mathcal{O}|$. That is, orbifolds are locally modeled in finite quotients of Euclidean spaces, thus generalizing manifolds by allowing this type of singularity. If we consider $U_\mathcal{A}:=\bigsqcup_{i\in I}\widetilde{U}_i$ and $\phi:=\bigsqcup_{i\in I} \phi_i:U_\mathcal{A}\to |\mathcal{O}|$, a \textit{change of charts} of $\mathcal{A}$ is a diffeomorphism $h:V\to W$, with $V,W\subset U_\mathcal{A}$ open sets, such that $\phi\circ h=\phi|_V$. The collection of all changes of charts of $\mathcal{A}$ generates a pseudogroup $\mathscr{H}_{\mathcal{A}}$ representing $[(S,\mathscr{H})]$.
\end{example}

Now let $(M,\f)$ be a foliation represented by the cocycle $\{(U_i,\pi_i,\gamma_{ij})\}$. The pseudogroup of local diffeomorphisms generated by $\gamma=\{\gamma_{ij}\}$ acting on
$$S_\gamma:=\bigsqcup_i S_i$$
is the \textit{holonomy pseudogroup} of $\f$ associated to $\gamma$, that we denote by $\mathscr{H}_\gamma$. If $\delta$ is another Haefliger cocycle defining $\f$ then $\mathscr{H}_\delta$ is equivalent to $\mathscr{H}_\gamma$, so we can define, up to equivalence, the holonomy pseudogroup of $\f$. We will write $(S_\f,\mathscr{H}_\f)$ to denote both this equivalence class and a specific representative in it, for it seldom leads to confusion. It is clear that $S_\f/\mathscr{H}_\f$ is precisely the $M/\f$ of $\f$ endowed with the quotient topology. Notice also that there is an isomorphism $\mathfrak{l}(\f)\to\mathfrak{X}(S_\f)^{\mathscr{H}_\f}$ sending $\overline{X}\in\mathfrak{l}(\f)$ to the vector field in $\mathfrak{X}(S_\f)^{\mathscr{H}_\f}$ given, on each $S_i$, by $\overline{X}_{S_i}$. In fact, in general, the study of the transverse geometry of $\f$ is the study of the $\mathscr{H}_\f$-invariant geometry of $S_\f$.

\begin{example}[Holonomy of suspensions]\label{example: holonomy of suspensions}
If $(M,\f)$ is given by the suspension of a homomorphism $h:\pi_1(B,x_0)\to\mathrm{Diff}(S)$ (see Example \ref{example: suspensions}) we can choose a cocycle $\{(U_i,\pi_i,\gamma_{ij})\}$ representing $\f$ where each $U_i$ is the domain of a trivialization of $\tau:M\to B$ and $\pi_i:U_i\to S$ is the trivial projection. Then $\mathscr{H}_\f$ is just the pseudogroup generated by $h(\pi_1(B,x_0))<\mathrm{Diff}(S)$, encoding the recurrence of the leaves on $S$.
\end{example}

The notion of fundamental group can be generalized to pseudogroups by considering homotopy classes of $\mathscr{H}$-loops, that is, sequences of continuous paths $c_i:[t_{i-1},t_i]\to S$, for $1\leq i\leq n$, and elements $h_i\in\mathscr{H}$ such that $h_ic_i(t_i)=c_{i+1}(t_i)$, for $1\leq i\leq n-1$, and  $c_1(0)=h_nc_n(1)=x$. We refer to \cite{salem} and \cite[Section 2.2]{caramello3} for details. In particular, for the holonomy pseudogroup $\mathscr{H}_\f$ a foliation $(M,\f)$ this furnishes an invariant $\pi_1(\f,\overline{x})$, the \textit{transverse fundamental group of $\f$}, which captures information of both the topology of $\f$ and the holonomy of the leaves. Its isomorphism class does not depend on the Haefliger cocycle representing $\f$ nor on the base point, when $M$ is connected (in this case we omit it, denoting simply $\pi_1(\f)$).

If $L:=L_x=L_y$, choose a path $c:[0,1]\to L$ joining $x$ to $y$. Fix a cocycle $\{(U_i,\pi_i,\gamma_{ij})\}$ representing $\f$ and a subdivision $0=t_1<\dots<t_{m+1}=1$ such that $s([t_k,t_{k+1}])\subset U_{i_k}$ for some $U_{i_k}$. Then, there is a diffeomorphism
$$\gamma_{i_mi_{(m-1)}}\circ\gamma_{i_{(m-1)}i_{(m-2)}}\circ\dots\circ\gamma_{i_2i_1}=\gamma_{i_mi_1}$$
between small enough neighborhoods of $\overline{x}=\pi_1(x)$ and $\overline{y}=\pi_m(y)$. If we identify $S_\gamma$ with a total transversal $\bigsqcup_i S_i$ for $\f$ containing $x$ and $y$, this becomes the ``sliding along the leaves'' notion from \cite{molino}, Section 1.7 (see Figure \ref{holonomy}). Let us denote the germ of $\gamma_{i_mi_1}$ at $\overline{x}$ by $h_c$. This germ actually depends only on the $\partial[0,1]$-relative homotopy class of $c$ \cite[Proposition 2.3.2]{candel}, hence, if we consider in particular the \textit{holonomy group} of $L$ at $x$, that is, the group
$$\mathrm{Hol}_x(L)=\{h_c\ |\ c:[0,1]\to L\ \mbox{\ is a loop}\},$$
we have a surjective homomorphism $h:\pi_1(L,x)\to\mathrm{Hol}_x(L)$.

\begin{figure}
\centering{
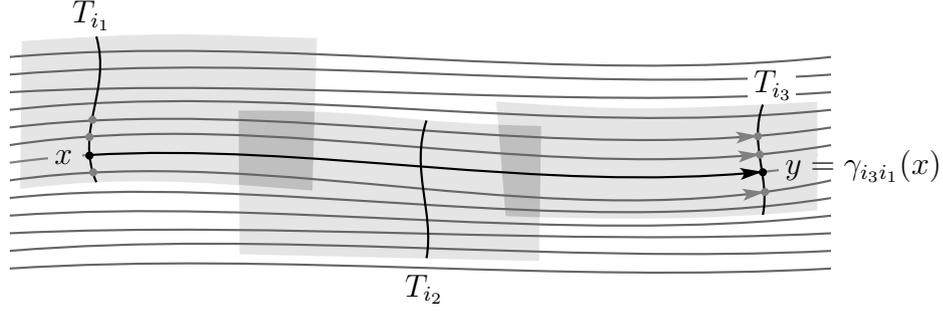}
\caption{Sliding along the leaves}
\label{holonomy}
\end{figure}

As the isomorphism class of $\mathrm{Hol}_x(L)$ does not depend on $x$, we often omit $x$ in this notation. In particular, we can say that $L$ is a \textit{leaf without holonomy} (or a \textit{generic leaf}) when $\mathrm{Hol}(L)=0$. It follows immediately from the surjectivity of $h$ that simply-connected leaves are without holonomy. Also, it can be shown that leaves without holonomy are generic, in the sense that $\{x\in M\ |\ \mathrm{Hol}_x(L)=0\}$ is residual in $M$ \cite[Theorem 2.3.12]{candel}.

Suppose $\mathrm{Hol}_x(L)$ is finite and identify it with a subgroup of $\mathrm{Diff}(S)$, where $S$ is a small local transversal of $\f$ passing through $x$. With this in mind we can state the famous Reeb Stability Theorem as follows (see \cite[Theorem 2.9]{mrcun} or \cite[Theorems 2.4.3 and Theorem 3.1.5]{candel}).

\begin{theorem}[Generalized local Reeb stability]\label{teorem: reeb}
Let $\f$ be a smooth foliation with a compact leaf $L_x$. If $\mathrm{Hol}_x(L)$ is finite then there is a saturated tubular neighborhood $\mathrm{pr}:\mathrm{Tub}(L_x)\to L_x$ restricted to which $\f$ is congruent to the foliation given by the suspension of $h:\pi_1(L,x)\to \mathrm{Hol}_x(L)<\mathrm{Diff}(S)$, where $S=\mathrm{pr}^{-1}(x)$. 
\end{theorem}

In particular, for every $y\in \mathrm{Tub}(L)$ the projection $\mathrm{pr}:L_y\to L_x$ is a finitely-sheeted covering map, the number of sheets being the index $|\mathrm{Hol}_x(L_x) : \mathrm{Hol}_y(L_y)|$. This indicates that leaf holonomy plays the same role of the stabilizer in the case of group actions. In fact, using Theorem \ref{teorem: reeb} one proves the following \cite[Theorem 2.15]{mrcun}.

\begin{proposition}[Leaf space of closed foliations with finite holonomy]\label{proposition: quotients of closed foliations are orbifolds}
Suppose $(M,\f)$ is a $q$-codimensional foliation whose every leaf is compact and with finite holonomy. Then $M/\f$ has a canonical $q$-dimensional orbifold structure, relative to which the local group of a leaf in $M/\f$ is its holonomy group.
\end{proposition}

When $\f$ is as in Proposition \ref{proposition: quotients of closed foliations are orbifolds} it is convenient to adopt a specific notation for the leaf space, so that we are promptly reminded that it is being considered as an orbifold: we will denote it by $M/\!/\f$ in this case. So to recap, $M/\f$ will denote the (topological) leaf space, with the quotient topology, and $M/\!/\f$ will denote it endowed with its canonical orbifold structure. The holonomy pseudogroup $(S_\f,\mathscr{H}_\f)$ is then a representative of $M/\!/\f$, viewed as an equivalence class of pseudogroups (see Example \ref{exe orbiflds}). Notice that we have
$$M/\f\cong |M/\!/\f| \cong S_\f/\mathscr{H}_\f.$$
Moreover, in this case $\pi_1(\f)$ coincides with the orbifold fundamental group $\pi_1^{\mathrm{orb}}(M/\!/\f)$, as defined by Thurston \cite{thurston}.

\subsection{Basic Cohomology}\label{section: basic cohomology}

Let $(M,\f)$ be a smooth foliation. A covariant tensor field $\xi$ on $M$ is \textit{$\f$-basic} if $\xi(X_1,\dots,X_i)=0$, whenever some $X_i\in\mathfrak{X}(\f)$, and $\mathcal{L}_X\xi=0$ for all $X\in\mathfrak{X}(\f)$. In particular, we say that a differential form $\omega\in\Omega^i(M)$ is \textit{basic} when it is basic as a tensor field. By Cartan's formula, $\omega$ is basic if, and only if, $i_X\omega=0$ and $i_X(d\omega)=0$ for all $X\in\mathfrak{X}(\f)$. These are the differential forms that project to differential forms in the local quotients $S$ and are invariant by the holonomy pseudogroup of $\f$ \cite[Proposition 2.3]{molino}. We denote the $\Omega^0(\f)$-module of basic $i$-forms of $\f$ by $\Omega^i(\f)$. Then
$$\Omega(\f):=\bigoplus_{i=0}^q\Omega^i(\f)$$
is the $\wedge$-graded \textit{algebra of basic forms} of $\f$.

By definition, $\Omega(\f)$ is closed under the exterior derivative, so we can consider the complex
$$\dots \stackrel{d}{\longrightarrow} \Omega^{i-1}(\f) \stackrel{d}{\longrightarrow} \Omega^i(\f) \stackrel{d}{\longrightarrow} \Omega^{i+1}(\f) \stackrel{d}{\longrightarrow} \cdots .$$
The cohomology groups of this complex are the \textit{basic cohomology groups} of $\f$, that we denote by $H^i(\f)$. A foliate map $f:(M,\f)\to(N,\g)$ pulls basic forms on $N$ back to basic forms on $M$ and hence induces a linear map $f^*:H^i(\g)\to H^i(\f)$.

When the dimensions $\dim(H^i(\f))$ are all finite (see Example \ref{example: infinite dimensional basic cohomology}), we define the \textit{basic Euler characteristic} of $\f$ as the alternate sum
$$\chi(\f)=\sum_i(-1)^i\dim(H^i(\f)).$$
In analogy with the manifold case, we say that $b^i(\f):=\dim(H^i(\f))$ are the \textit{basic Betti numbers} of $\f$. When $\f$ is the trivial foliation by points we recover the classical Euler characteristic and Betti numbers of $M$.

Since we have an identification between $\f$-basic forms and $\mathscr{H}_\f$-invariant forms on $S_\f$ and an identification between differential forms on an orbifold $\mathcal{O}$ and $\mathscr{H}_{\mathcal{O}}$-invariant forms on $U_{\mathcal{O}}$, Proposition \ref{proposition: quotients of closed foliations are orbifolds} gives us the following.

\begin{proposition}\label{prop: basic cohomology of closed foliations}
Let $(M,\f)$ be a foliation such that every leaf is compact and with finite holonomy. Then the projection $\pi:M\to M/\!/\f$ induces an isomorphism of differential complexes $\pi^*:\Omega(M/\!/\f)\to\Omega(\f)$. In particular, $H(\f)\cong H_{\mathrm{dR}}(M/\!/\f)$.
\end{proposition}

\subsection{Foliations of Orbifolds}\label{section: riemannian foliations}

Let $\mathcal{O}$ be an orbifold with atlas $\mathcal{A}=\{(\widetilde{U}_i,H_i,\phi_i)\}$ and associated pseudogroup $(U_\mathcal{A},\mathscr{H}_\mathcal{A})$ (see Example \ref{exe orbiflds}). Following \cite[Section 3.2]{haefliger2}, we define a smooth foliation $\f$ of $\mathcal{O}$ as a smooth foliation of $U_\mathcal{A}$ which is invariant by $\mathscr{H}_\mathcal{A}$. The atlas can be chosen so that on each $\widetilde{U}_i$ the foliation is given by a surjective submersion with connected fibers onto a manifold $S_i$. The holonomy pseudogroup of $\f$, therefore, will be generated by the local diffeomorphisms of the disjoint union $\bigsqcup_{i\in I} S_i$ that are projections of elements of $\mathscr{H}_\mathcal{A}$. All notions defined so far for foliations on manifolds therefore extend to foliations on orbifolds.

\section{Riemannian foliations}\label{section: riemann foliations}

Let $\f$ be a smooth foliation of $M$. A \textit{transverse metric} for $\f$ is a symmetric, positive, $\f$-basic $(2,0)$-tensor field $\mathrm{g}^T$ on $M$. In this case $(M,\f,\mathrm{g}^T)$ is called a \textit{Riemannian foliation}. A Riemannian metric (in the usual sense) $\mathrm{g}$ on $M$ is called \textit{bundle-like} for $\f$ if for any open set $U$ and any $Y,Z\in\mathfrak{L}(\f|_U)$ perpendicular to the leaves we have $\mathrm{g}(Y,Z)\in\Omega^0(\f|_U)$. In this case, setting
$$\mathrm{g}^T(X,Y):=\mathrm{g}(X^\bot,Y^\bot)$$
defines a transverse metric for $\f$, where we write $X=X^\top+X^\bot$ with respect to the decomposition $TM=T\f\oplus T\f^\perp$. Conversely, given $\mathrm{g}^T$ one can always choose a bundle-like metric on $M$ that induces it \cite[Proposition 3.3]{molino}. With a bundle-like metric chosen, we will identify the bundles $\nu\f\equiv T\f^\perp$.

\begin{example}
If a foliation $\f$ on $M$ is given by the action of a Lie group $G$ (i.e, such that all orbits have the same dimension, see Example \ref{exe: foliated actions}) and $\mathrm{g}$ is a Riemannian metric on $M$ such that $G$ acts by isometries, then $\mathrm{g}$ is bundle-like for $\f$ \cite[Remark 2.7(8)]{mrcun}. In other words, a foliation induced by an isometric action is Riemannian.
\end{example}

\begin{example}[Gromoll--Grove {\cite[Theorem 5.4]{gromoll2}}]\label{example: 1foliations of the sphere}
The $1$-dimensional Riemannian foliations of the euclidean sphere $\mathbb{S}^n$ where classified by D.~Gromoll and K.~Grove. They exist only if $n$ is odd, say $n=2k+1$, and are all homogeneous, given (up to isometric congruence) by $\mathbb{R}$-actions of the type
$$t\cdot(z_0,\dots,z_k)=(e^{2\pi\mathrm{i}\lambda_0t}z_0,\dots,e^{2\pi\mathrm{i}\lambda_kt}z_k),$$
where $\lambda_i\in(0,1]$ and $z_i\in\mathbb{S}^n\subset\mathbb{C}^{k+1}$. We will call these foliations \textit{generalized Hopf fibrations}, since we get the usual Hopf fibration when $\lambda_i=1$ for each $i$. In particular, such an action correspond to a closed Riemannian $1$-foliation $\f$ of $\mathbb{S}^n$ precisely when all $\lambda_i$ are rational, say $\lambda_i=p_i/q_i$. Notice that in this case we can equivalently assume that $\lambda_i\in\mathbb{N}$, by changing the parameter $t$ to $\mathrm{lcm}(q_1,\dots,q_k)t$, hence $\mathbb{S}^n/\!/\f$ is a weighted projective space $\mathbb{CP}^{k}[\lambda_0,\dots,\lambda_k]$.

Let us visualize these foliations in the case of the $3$-dimensional sphere, that is, for $k=1$. Consider the action of $\mathbb{T}^2=\mathbb{S}^1\times\mathbb{S}^1$ on $\mathbb{S}^3$ by $(t_0,t_1)\cdot(z_0,z_1)=(t_0z_0,t_1z_1)$. This action has two singular orbits, $\mathbb{T}^2(1,0)$ and $\mathbb{T}^2(0,1)$, that are diffeomorphic to $\mathbb{S}^1$. The other orbits are tori and coincide with the distance tubes of the two singular orbits. The $1$-dimensional Riemannian foliations of $\mathbb{S}^3$, up to congruence, can be identified with the $1$-dimensional Lie subalgebras of $\mathbb{R}^2\cong\mathrm{lie}(\mathbb{T}^2)$ via the induced action of the corresponding $1$-parameter subgroup. They restrict to Kronecker foliations on each regular $\mathbb{T}^2$-orbit (see Figure \ref{hopf}).

\begin{figure}
\centering{
\resizebox{0.4\textwidth}{!}{
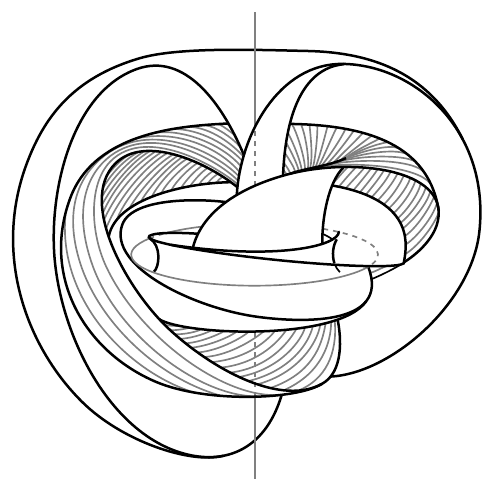}}
\caption{The $1$-dimensional foliations of $\mathbb{S}^3$ (via stereographic projection)}
\label{hopf}
\end{figure}
\end{example}

\begin{example}\label{example: Riemannian suspensions}
Let $(S,\mathrm{g})$ be a Riemannian manifold. A foliation $\f$ defined by the suspension of a homomorphism (see Examples \ref{example: suspensions} and \ref{example: holonomy of suspensions}) $h:\pi_1(B,x_0)\to\mathrm{Iso}(S)$ is naturally a Riemannian foliation \cite[Section 3.7]{molino}.
\end{example}

\begin{example}
By the description via Haefliger cocycles, the pullback of a Riemannian foliation is obviously a Riemannian foliation (see Example \ref{exe: pullback foliation}).
\end{example}

We now associate a canonical transverse connection for a Riemannian foliation $(M,\f,\mathrm{g}^T)$. Choose a bundle-like metric for $(\f,\mathrm{g}^T)$ and denote its Levi-Civita connection by $\nabla$. Via the identification $\nu\f=T\f^\perp$, we define a connection $\nabla^B$ on $\nu\f$ by
\begin{equation}\label{eq: conn basic}\nabla^B_XY:=\begin{cases} [X,Y]^\perp & \mbox{if}\ X\in\Gamma(T\f)=\mathfrak{X}(\f),\\
(\nabla_XY)^\perp & \mbox{if}\ X\in\Gamma(T\f^\perp).\end{cases}\end{equation}
This connection on $\nu\f$ does not depend on the choice of the bundle-like metric, being completely determined by $\mathrm{g}^T$. It is in fact the unique $\mathrm{g}^T$-metric and torsion-free connection on $\nu\f$ \cite[Theorem 5.9]{tondeur}, so in analogy with the classical case of Riemannian manifolds we call it the \textit{basic Levi-Civita connection} of $\f$. The partial connection on $\nu\f$ defined only for $X\in \Gamma(T\f)$ by $[X,Y]^\perp$ is called the \textit{Bott connection on $\nu\f$}. The connection $\nabla^B$ induces a covariant derivative on $\mathfrak{l}(\f)$, which in terms of a submersion $\pi:U\to S$ locally defining $\f$ corresponds to the effect of the Levi-Civita connection $\nabla^S$ of $(S,\pi_*(\mathrm{g}^T))$, that is, $\pi_*(\nabla^B_X\overline{Y})=\nabla^S_{\pi_*X}\pi_*\overline{Y}$, for $\overline{Y}\in\mathfrak{l}(\f|_U)$ and $X\in TU$. The following characterization of bundle-like metrics is related to this property:

\begin{proposition}[{\cite{reinhart}}]\label{prop: Reinhart characterization}
A Riemannian metric $\mathrm{g}$ is bundle-like for $(M,\f)$ if and only if a geodesic that is perpendicular to a leaf at one point remains perpendicular to all the leaves it intersects. Moreover, geodesic segments perpendicular to the leaves project to geodesic segments in the local quotients $S$.
\end{proposition}

It follows from this result that the leaves of a Riemannian foliation are locally equidistant. Contrarily to the classical case of Riemannian metrics on manifolds, on the other hand, not every smooth foliation admits a transverse metric so that it becomes a Riemannian foliation. This will become more apparent when we study Molino's structural theorem in Section \ref{subsection: Molino Theory}, but we can already conclude that from the fact that the basic cohomology of Riemannian foliations on compact manifolds have finite dimension (Theorem \ref{theorem: dim basic cohomology is finite} below), which is not true for smooth foliations in general:

\begin{example}[{\cite{ghys2}}]\label{example: infinite dimensional basic cohomology}
Consider
$$A=\begin{pmatrix}
1 & 0 \\
1 & 1\end{pmatrix}\in\mathrm{SL}_2(\mathbb{Z})$$
and its induced diffeomorphism $\overline{A}:\mathbb{R}^2/\mathbb{Z}^2\cong\mathbb{T}^2\to\mathbb{T}^2$. Then we have an homomorphism $\mathbb{Z}\cong\pi_1(\mathbb{S}^1)\to\mathrm{Diff}(\mathbb{T}^2)$ given by $n\mapsto\overline{A}^n$. Ghys shows in \cite{ghys2} that the $1$-dimensional foliation $\f$ given on the torus bundle $\mathbb{T}^3_A\to\mathbb{S}^1$ by the suspension of this homomorphism has $\dim(H^2(\f))=\infty$ (see Examples \ref{example: suspensions} and \ref{example: holonomy of suspensions}). In fact, it can be verified by direct calculations that an $\overline{A}$-invariant $1$-form on $\mathbb{T}^2$ (corresponding to a basic $1$-form) is of the type $f(x)dx$, thus closed, while an $\overline{A}$-invariant $2$-form is of the type $g(x)dx\wedge dy$. Therefore $H^2(\f)$ must be infinite-dimensional.
\end{example}

\begin{theorem}[Alaoui--Sergiescu--Hector {\cite[Théorème 0]{kacimi2}}]\label{theorem: dim basic cohomology is finite}
Let $\f$ be a Riemannian foliation of a compact manifold $M$. Then $\dim(H^i(\f))<\infty$.
\end{theorem}

As remarked in \cite[Proposition 3.11]{goertsches}, the hypothesis that $M$ is compact can be relaxed to $\f$ being transversely compact, provided that $\f$ is a complete Riemannian foliation, in the following sense.

\begin{definition}[Complete Riemannian foliation] A Riemannian foliation $\f$ of a manifold $M$ is \textit{complete} if $M$ is a complete Riemannian manifold with respect to some bundle-like metric for $\f$.
\end{definition}

It follows that $\chi(\f)$ is always defined for transversely compact (i.e., such that $M/\f$ is compact) complete Riemannian foliations. We mention the following transverse analogue of the Bonnet--Myers Theorem due to J.~Hebda:

\begin{theorem}[{\cite[Theorem 1]{hebda}}]\label{theorem: hebda}
Let $\f$ be a complete Riemannian foliation satisfying $\ric_\f\geq c>0$. Then $\f$ is transversely compact and $H^1(\f)\cong 0$.
\end{theorem}

Basic cohomology of Riemannian foliations can be studied via the basic Laplacian $\Delta_B$. Let $\f$ be a transversely oriented Riemannian foliation of a compact oriented manifold $M$ endowed with a bundle-like metric $\mathrm{g}$. Consider the scalar product $\proin{\cdot}{\cdot}_B$ in $\Omega^i(\f)$ given by the restriction of the usual scalar product in $\Omega^i(M)$ (see, e.g., \cite[Section 2 of Chapter 7]{petersen}). The \textit{basic laplacian} is the operator $\Delta_B:\Omega^i(\f)\to\Omega^i(\f)$ given by $\Delta_B=d\delta+\delta d$, where $\delta$ is the formal adjoint of $d$ with respect to $\proin{\cdot}{\cdot}_B$. We denote by $\mathcal{H}^i(\f)$ the space of basic \textit{harmonic} $i$-forms, that is, basic $i$-forms $\alpha$ satisfying $\Delta_B\alpha=0$. For a thorough introduction to this objects, we refer to \cite[Chapter 7]{tondeur}.

There is a basic version of Hodge's decomposition theorem for $\Delta_B$ that gives an orthogonal decomposition (see \cite[Theorem 7.22]{tondeur})
$$\Omega^i(\f)\cong \mathrm{Im}(d)\oplus\mathrm{Im}(\delta)\oplus\mathcal{H}^i(\f)$$
and so provides an isomorphism (see also \cite[Theorem 7.51]{tondeur})
$$H^i(\f)\cong\mathcal{H}^i(\f).$$
This leads to duality theorems for the basic cohomology. Poincaré duality in its expected form, however, is only available for the so-called taut foliations: a foliation $\f$ of $M$ is \textit{taut} it there exists a Riemannian metric on $M$ with respect to which every leaf of $\f$ is a minimal submanifold.

\begin{theorem}[{\cite{kamber2}}, {\cite{kacimi3}} and {\cite{sergiescu}}]
Let $\f$ be a transversely oriented Riemannian foliation of codimension $q$ of a compact manifold $M$. Then $\f$ is taut if and only if $H^i(\f)\cong H^{q-i}(\f)$.
\end{theorem}

Tautness is also characterized in \cite[Theorem 6.4]{lopez} by the vanishing of a degree $1$ cohomology class, the mean curvature class of $\f$. In particular, if $\ric_\f\geq c>0$, one concludes from Theorem \ref{theorem: hebda} that $\f$ is taut. We also mention the following characterization for tautness by Rummler.

\begin{proposition}[{\cite{rummler}}]
A $p$-dimensional orientable smooth foliation $\f$ of $M$ is taut if, and only if, there exists $\theta\in\Omega^p(M)$ which is non-singular along the leaves and satisfies
$$d\theta(v_1,\dots,v_{p+1})=0$$
whenever $p$ of the $p+1$ vectors $v_i$ are tangent to $\f$.
\end{proposition}

Although tangent orientability appears in Rummler's criterion, tautness is a transverse property: it depends only on the holonomy pseudogroup $\mathscr{H}_\f$ (see \cite[Theorem 4.1]{haefliger}). We refer to \cite[Section 10.5]{candel} for more on taut foliations.

\subsection{Complete Pseudogroups of Local Isometries}

As we mentioned earlier, the transverse information of a Riemannian foliation corresponds to the holonomy-invariant information on a total transversal, so, in considering transverse geometry, one can focus on the later. In this section we survey this point of view, pioneered by A.~Haefliger, focusing mainly on the study the closures of the orbits of a complete  pseudogroup of isometries. The main references are \cite{haefliger6}, \cite{salem1} and \cite{salem}.

It follows directly from the definition of a Riemannian foliation $(M,\f,\mathrm{g}^T)$ that the transverse metric $\mathrm{g}^T$ projects to Riemannian metrics on the local quotients $S_i$ of a Haefliger cocycle $\{(U_i,\pi_i,\gamma_{ij})\}$ defining $\f$ (see \cite[Section 3.2]{molino}, and also \cite[Remark 2.7(2)]{mrcun}). Therefore the holonomy pseudogroup $\mathscr{H}_\f$ becomes a pseudogroup of local isometries of $S_\f$. Moreover, by choosing a bundle-like metric on $M$, the submersions $\pi_i$ become Riemannian submersions.

\begin{definition}[Complete pseudogroups]
We say that a pseudogroup of local isometries $(\mathscr{H},S)$ is \textit{complete} when, given $x,y\in S$, there exists neighborhoods $U\ni x$ and $V\ni y$ such that every germ of an element of $\mathscr{H}$ with source in $U$ and target in $V$ is the germ of an element of $\mathscr{H}$ defined on the whole of $U$.
\end{definition}

This property is invariant by differentiable equivalences \cite[p. 278]{salem}. It is also independent of the concept of completeness in the sense of Riemannian manifolds, as the following example shows.

\begin{example}[{\cite[Example 2.8]{salem}}]
Suppose $\mathscr{H}$ is a pseudogroup of local diffeomorphisms of $S$ whose equivalence class represents an orbifold $\mathcal{O}$. One can always choose a Riemannian metric on $\mathcal{O}$, which corresponds to an $\mathscr{H}$-invariant Riemannian metric on $S$. Then $\mathcal{O}$ is not necessarily complete as a Riemannian orbifold, but we claim that $\mathscr{H}$ is a complete pseudogroup of local isometries. In fact, for every point $x\in S$ one can find an $\mathscr{H}$-invariant neighborhood $U\ni x$. Hence, if $x,y\in S$ are in the same $\mathscr{H}$-orbit then every germ of an element of $\mathscr{H}$ with source and target in $U$ is the germ of an element defined on the whole of $U$ (here we take $V=U\ni y$ to be the neighborhood of $y$ used in the definition of completeness). On the other hand, if $x$ and $y$ are in different orbits, then since $S/\mathscr{H}$ is Hausdorff one can separate the orbits $\mathscr{H}x$ and $\mathscr{H}y$ by two disjoint open neighborhoods $U\supset\mathscr{H}x$ and $V\supset\mathscr{H}y$. Therefore there are no germs of elements of $\mathscr{H}$ with source in $U$ and target in $V$.
\end{example}

The example below establishes the connection between complete Riemannian foliations and complete pseudogroups of local isometries. A proof can be seen in \cite[p. 281]{salem}. It follows essentially from Proposition \ref{prop: Reinhart characterization}.

\begin{example}[{\cite[Example 1.2.1]{haefliger6}}]
The holonomy pseudogroup of a complete Riemannian foliation is a complete pseudogroup of local isometries.
\end{example}

Let $\mathscr{H}$ be a complete pseudogroup of local isometries. Its closure $\overline{\mathscr{H}}$ is defined as the pseudogroup on $S$ whose elements are locally the limits, in the $C^1$ topology, of elements of $\mathscr{H}$.

\begin{proposition}[{\cite[Proposition 3.1]{haefliger6}}]\label{prop: closure of pseudogroup}
The closure $\overline{\mathscr{H}}$ of a complete pseudogroup of local isometries $\mathscr{H}$ is a complete pseudogroup of local isometries, unique up to equivalence. Moreover, $S/\overline{\mathscr{H}}$ is Hausdorff and, for any $x\in S$,
$$\overline{\mathscr{H}x}=\overline{\mathscr{H}}x.$$
\end{proposition}

We say that $\mathscr{H}$ is \textit{closed} when $\mathscr{H}=\overline{\mathscr{H}}$.

\begin{example}[{\cite[Example at p. 279]{salem}}]\label{example: Salem closure pseudogroup generated by lie group}
Let $G<\mathrm{Iso}(M)$, for a Riemannian manifold $M$. If $\mathscr{H}$ is the pseudogroup generated by the restriction of elements of $G$ to open sets, then $\overline{\mathscr{H}}$ is the pseudogroup generated by the closure $\overline{G}<\mathrm{Iso}(M)$, in the compact-open topology.
\end{example}

\subsection{A brief interlude on sheaves}

Before we continue it will be convenient to recall the notion of sheaves, which are tools for working with locally defined data on topological spaces. A \textit{presheaf} $\mathscr{P}$ on a topological space $(X,\tau)$ consists of an assignment of a set $\mathscr{P}(U)$, to each $U\in\tau$, and a \textit{restriction map} $\res^U_V:\mathscr{P}(U)\to\mathscr{P}(V)$, to each $U,V\in\tau$ with $V\subset U$, such that $\res^U_U$ is always the identity map and $\res^V_W\circ \res^U_V=\res^U_W$ whenever $W\subset V\subset U$. An element $s\in\mathscr{P}(U)$ is a \textit{section} over $U$.

One often is interested in local data (the sets $\mathscr{P}(U)$) that have additional structure, such as algebraic operations. In this case one requires that the restriction maps preserve the additional structure. This leads to the definition of presheaves of groups, rings and so on. For example, if each $U$ is assigned to a (real) Lie algebra $\mathscr{P}(U)$ and each $\res^U_V$ is a Lie algebra homomorphism, then $\mathscr{P}$ is a \textit{presheaf of Lie algebras}.

\begin{example}
Let $M$ be a smooth manifold. The assignment $U\mapsto C^\infty(U)$, of an open set $U$ to the ring of smooth functions $f:U\to\mathbb{R}$, together with the usual restriction of functions is a presheaf of rings $\mathscr{C}^\infty_M$.  
\end{example}

Given a presheaf $\mathscr{P}$ on $(X,\tau)$ and $x\in X$, let $\mathcal{U}_x$ be the collection of open sets that contain $x$. For $U_1,U_2\in\mathcal{U}_x$, declare $s_1\in\mathscr{P}(U_1)$ and $s_2\in\mathscr{P}(U_2)$ to be equivalent if there exists $V\in\mathcal{U}_x$ such that $V\subset U_1\cap U_2$ and $\res^{U_1}_V(s_1)=\res^{U_2}_V(s_2)$. The equivalence class of $s\in\mathscr{P}(U)$ is the \textit{germ} of $s$ at $x$, denoted by $\res^U_x(s)$ or simply by $[s]_x$. The set $\mathscr{P}_x$ of germs at $x$ is called the \textit{stalk} of $\mathscr{P}$ at $x$. Notice that the stalks of a presheaf of structured sets (say groups or Lie algebras) inherit that structure in a natural way.

A \textit{sheaf} on a topological space on $(X,\tau)$ is a presheaf $\mathscr{S}$ on $(X,\tau)$ such that, for any $U\in\tau$ and any open covering $\{U_i\}_{i\in I}$ of $U$,
\begin{enumerate}[(i)]
\item if $s,t\in\mathscr{S}(U)$ satisfy $\res^U_{U_i}(s)=\res^U_{U_i}(t)$ for every $i$, then $s=t$, and
\item \label{gluing property} if $s_i\in\mathscr{S}(U_i)$ satisfy $\res^{U_{i_1}}_{U_{i_1}\cap U_{i_2}}(s_{i_1})=\res^{U_{i_2}}_{U_{i_1}\cap U_{i_2}}(s_{i_2})$ for every $i_1,i_2\in I$, then there exists $s\in\mathscr{S}(U)$ with $\res^U_{U_i}(s)=s_i$.
\end{enumerate}

A sheaf of groups (or rings, Lie algebras etc.) is just a presheaf of groups (or rings, Lie algebras etc.) that is a sheaf in the above sense.

\begin{example}
It is not difficult to check that, for a smooth manifold $M$, the presheaf $\mathscr{C}^\infty_M$ is a sheaf of rings. Now let $\pi:E\to M$ be a smooth vector bundle. The presheaf that assigns to each open set $U$ the space of smooth local sections of $E$ over $U$, with the usual restriction maps, is a sheaf of $\mathscr{C}^\infty_M$-modules.
\end{example}

Sections of a sheaf $\mathscr{S}$ can be realized as (usual) sections of its étalé space. In fact, given a presheaf $\mathscr{P}$ on $(X,\tau)$, its \textit{étalé space} is the space
$$\etale(\mathscr{P}):=\bigsqcup_{x\in X}\mathscr{P}_x$$
endowed with the (in general non-Hausdorff) topology whose basis is given by the sets of the form $V_{U,s}=\{[s]_x\ |\ x\in U\}$, for $U\in\tau$ and $s\in\mathscr{P}(U)$. There is a canonical projection
$$\pi_\mathscr{P}:\etale(\mathscr{P})\ni[s]_x\mapsto x\in X,$$
which is a local homeomorphism, by construction. The presheaf $\Gamma(\etale(\mathscr{P}))$ of local sections of $\pi_\mathscr{P}$ (that is, continuous maps $s:U\to\etale(\mathscr{P})$ with $\pi_\mathscr{P}\circ s=\mathrm{id}_U$) is a sheaf, called the \textit{sheafification} of $\mathscr{P}$, or the \textit{sheaf of germs of sections of} $\mathscr{P}$. When $\mathscr{P}$ is already a sheaf, it is isomorphic to $\Gamma(\etale(\mathscr{P}))$.

\begin{example}[Constant sheaves]
A sheaf $\mathscr{S}$ on $X$ is \textit{constant} when all its stalks are equal to the same set $Z$. In this case we can identify $\etale(\mathscr{P})\cong X\times Z$, where $Z$ is given the discrete topology, and $\pi_\mathscr{P}:X\times Z\to X$ is just the projection on the first factor. Under this identification a section $s\in\mathscr{S}(U)$ is identified with a locally constant map $U\to Z$. More generally, a sheaf $\mathscr{S}$ is called \textit{locally constant} when every $x\in X$ admits an open neighborhood $U\ni x$ where $\mathscr{S}|_U$ is constant. In this case the identification $\etale(\mathscr{S}|_U)\cong U\times Z$ is a \textit{local trivialization} of $\mathscr{S}$.

It is instructive to compare this notion with the case of a presheaf $\mathscr{P}$ on $X$ that satisfies $\mathscr{P}(U)=Z$ for all $U$. In this case $\mathscr{P}$ is usually not a sheaf, since the ``gluing property'' \eqref{gluing property} does not hold in general (unless $\tau$ has the peculiar feature that all open sets are connected, or $\#Z\leq1$). In fact, the constant sheaf $\mathscr{S}$ with stalk $Z$ is the sheafification of $\mathscr{P}$.
\end{example}

Sheaves can be ``transported'' through continuous maps, as follows. Suppose $f:X\to Y$ is continuous and $\mathscr{R}$ is a sheaf on $X$. Then we define the \textit{direct image of $\mathscr{R}$ by $f$} as the sheaf $f_*\mathscr{R}$ on $Y$ given by $f_*\mathscr{R}(U)=\mathscr{R}(f^{-1}(U))$ (which is in fact a sheaf on $Y$). The restriction maps $f_*\res$ of $f_*\mathscr{R}$ satisfy $f_*\res^U_V=\res^{f^{-1}(U)}_{f^{-1}(V)}$.

On the other hand, if we have a sheaf $\mathscr{S}$ on $Y$ then we can also obtain a sheaf $f^{-1}\mathscr{S}$ on $X$, called the \textit{inverse image of $\mathscr{S}$ by $f$}, which consists of the sheaf of germs of sections of the étalé space $f^{-1}\etale{\mathscr{S}}:=\{(x,[s]_y)\in X\times \etale(\mathscr{S})\ |\ f(x)=y\}$ over $X$. It is instructive to try to understand $f^{-1}\mathscr{S}(U)$ in terms of the values of $\mathscr{S}$ on open sets of $Y$, although this is a little involved since $f(U)$ is not necessarily an open set. To circumvent this, we need to generalize the notion of germ, as follows. For any subset $A\subset Y$, let $U_1$ and $U_2$ be open neighborhoods of $A$. We will say that two sections $s_1\in\mathscr{S}(U_1)$ and $s_2\in\mathscr{S}(U_2)$ are equivalent if there exists a neighborhood $W\subset U_1\cap U_2$ of $A$ such that $\res^{U_1}_W(s_1)=\res^{U_2}_W(s_2)$. We denote the set of equivalence classes by $\mathscr{S}_A$, and an equivalence class by $[s]_A$. Notice that if $B\subset A$ we can define a restriction $\res^A_B[s]_A$, since any open neighborhood of $A$ will also be an open neighborhood of $B$. With this concept we can now consider the presheaf $f^{-1}_{\mathrm{pre}}\mathscr{S}$ on $X$ given by $f^{-1}_{\mathrm{pre}}\mathscr{S}(U)=\mathscr{S}_{f(U)}$. It can fail to be a sheaf (even when $\mathscr{S}$ is a sheaf) but $f^{-1}\mathscr{S}$ is isomorphic to its sheafification $\Gamma(\etale(f^{-1}_{\mathrm{pre}}\mathscr{S}))$. The restriction maps $f^{-1}\res$ of $f^{-1}\mathscr{S}$ satisfy $f^{-1}\res^U_V=\res^{f(U)}_{f(V)}$.

\subsection{Infinitesimal sheaf of a complete pseudogroup}\label{section: infinitesimal sheaf}

There is a structural theorem for complete pseudogroups of local isometries, due to E.~Salem \cite{salem1}, which describes the closures of the orbits of such a pseudogroup as orbits of a sheaf of Lie algebras on it. This result can also be seen as a generalization of Myers--Steenrod Theorem for closed, complete pseudogroups of local isometries.

Let $\mathscr{H}$ be a pseudogroup of local isometries of $S$. We consider the sheaf which to each open set $U\subset S$ associates the space $\mathfrak{iso}_\mathscr{H}(U)$ of vector fields $X$ on $U$ with the property that for all $x\in U$ there exists a neighborhood $V_x\ni x$ and $\varepsilon>0$ such that $\exp(tX)$ is defined on $V_x$, when $|t|<\varepsilon$, and $\exp(tX)\in\mathscr{H}$. 

\begin{definition}[Sheaf of infinitesimal transformations]
With the notation above, the sheaf $\mathfrak{iso}_\mathscr{H}$ is called the \textit{sheaf of infinitesimal transformations} of $\mathscr{H}$.
\end{definition}

Notice that $\mathfrak{iso}_\mathscr{H}$ plays an analog role, in the context of pseudogroups, as the Lie algebra of Killing vector fields induced by an isometric Lie group action on a manifold (cf. Example \ref{example: infinitesimal sheaf of lie group}). For a complete pseudogroup of local isometries $\mathscr{H}$ the infinitesimal the sheaf $\mathfrak{iso}_\mathscr{H}$ is a locally constant sheaf of Lie algebras of germs of Killing vector fields on $S$ (see \cite[Proposition]{salem1}).

\begin{example}[{\cite[Example at p. 188]{salem1}}]\label{example: infinitesimal sheaf of lie group}
If $\mathscr{H}$ is generated by a closed subgroup $G$ of isometries of a Riemannian manifold $M$, then $\mathfrak{iso}_\mathscr{H}$ is the sheaf whose sections are the restrictions of the fundamental Killing vector fields of the action, that is, elements in the Lie algebra $\mathfrak{iso}_G$ which is the image of the map $\mathfrak{g}\ni X\mapsto X^\#\in\mathfrak{X}(M)$, where
$$X^\#(x)=\left.\od{}{t}\exp(tX)x\right|_{t=0}.$$
Notice that the sheaf $\mathfrak{iso}_\mathscr{H}$ is isomorphic to the constant sheaf with stalk $\mathfrak{g}^{-1}$ on $M$.
\end{example}

In view of the facts above, the following definition is natural. A complete pseudogroup of local isometries $\mathscr{H}$ is a \textit{Lie pseudogroup} when any element of $\mathscr{H}$ that is close enough to the identity is of the form $\exp(X)$, for a local section $X$ of $\mathfrak{iso}_\mathscr{H}$ close to $0$.

\begin{theorem}[Myers--Steenrod theorem for complete pseudogroups {\cite[Théorème]{salem1}}]\label{theorem: salem}
Every complete pseudogroup of local isometries $\mathscr{H}$ which is closed in the $C^1$ topology is a Lie pseudogroup.
\end{theorem}

As a corollary, it follows that the orbits of the closure $\overline{\mathscr{H}}$ are closed submanifolds of $S$, since they are given by the orbits of its sheaf of infinitesimal transformations. Combining this with Proposition \ref{prop: closure of pseudogroup} we have the following.

\begin{corollary}[Structural theorem for complete pseudogroups]
Let $\mathscr{H}$ be a complete pseudogroup of local isometries of a manifold $S$. Then there is a locally constant sheaf $\mathscr{C}_\mathscr{H}:=\mathfrak{iso}_{\overline{\mathscr{H}}}$ of Lie algebras of germs of local Killing vector fields on $S$ whose orbits describe the closures of the orbits of $\mathscr{H}$.
\end{corollary}

We will call $\mathscr{C}_\mathscr{H}$ the \textit{Molino sheaf of $\mathscr{H}$}. Since it is locally constant, if $S/\mathscr{H}$ is connected all stalks of $\mathscr{C}_\mathscr{H}$ are isomorphic to a Lie algebra $\mathfrak{g}^{-1}$.

\begin{definition}[Structural Lie algebra]
The Lie algebra $\mathfrak{g}$ will be called the \textit{structural Lie algebra} of $\mathscr{H}$.
\end{definition}

We are specially interested in the case of the holonomy pseudogroup of a complete Riemannian foliation, for which we can use Theorem \ref{theorem: salem} to similarly describe the closures of the leaves as orbits of a sheaf, since we have $M/\overline{\f}=S/\overline{\mathscr{H}}$ by Proposition \ref{prop: closure of pseudogroup}. To make this more precise, we need the following definition.

\begin{definition}[Transverse Killing vector field]
A field $X\in\mathfrak{X}(M)$ is a \textit{foliate Killing vector field} if $\mathcal{L}_X\mathrm{g}^T=0$. These fields form a Lie subalgebra of $\mathfrak{L}(\f)$ and there is, thus, a corresponding Lie algebra of \textit{transverse Killing vector fields}, that we will denote by $\mathfrak{iso}(\f,\mathrm{g}^T)$. We will omit the transverse metric when it is clear from the context, writing just $\mathfrak{iso}(\f)$. In a similar way we define \textit{local foliate/transverse Killing vector fields} on an open set $U$ and denote the corresponding algebra by $\mathfrak{iso}(\f|_U)$ 
\end{definition}

In terms of the holonomy pseudogroup $\mathscr{H}_\f$, the vector fields in $\mathfrak{iso}(\f)$ are precisely those that project to $\mathscr{H}_\f$-invariant Killing vector fields on $S_\f$. Local Killing vector fields are more flexible: if $\pi:U\to S$ is a submersion locally defining $\f$, the elements of $\mathfrak{iso}(\f|_U)$ are the transverse fields that project to Killing vector fields on $S$ (not necessarily $\mathscr{H}_\f$-invariant). The inverse images $\pi_i^{-1}(\mathscr{C}_{\mathscr{H}_\f})$ of the Molino sheaf of $\mathscr{H}_\f$ hence patch together on $M$ to form a sheaf $\mathscr{C}_\f$ of Lie algebras of germs of local transverse Killing fields (see \cite[\S 3.4]{salem}, also \cite[Remark at p. 711]{haefliger2}).

\begin{definition}[Molino sheaf]\label{def: molino sheaf}
The sheaf $\mathscr{C}_\f$ is called the \textit{Molino sheaf of $\f$}.
\end{definition}

For $\overline{X}\in\mathfrak{iso}(\f|_U)$, we define the orbit $\overline{X}\cdot x$ of $x\in M$ as the saturation of the orbit of $x$ under the flow of a representative $X\in\mathfrak{L}(\f)$ (notice that this is well defined, i.e., it is independent of the choice of the representative, since different representatives differ by a vector field in $\mathfrak{X}(\f)$). We define the orbits of $\mathscr{C}_\f$ similarly: the orbit of $x$ consists of all leaves that can be reached by continuous paths starting at $x$ and contained in orbits of sections of the sheaf. We see that the closures of the leaves of a complete Riemannian foliation $\f$ are the orbits of $\mathscr{C}_\f$ (see Figure \ref{molinosheaf}). In Section \ref{subsection: Molino Theory} we will revisit this result from a completely different approach, obtaining Molino's original definition of $\mathscr{C}_\f$.

\begin{figure}
\centering{
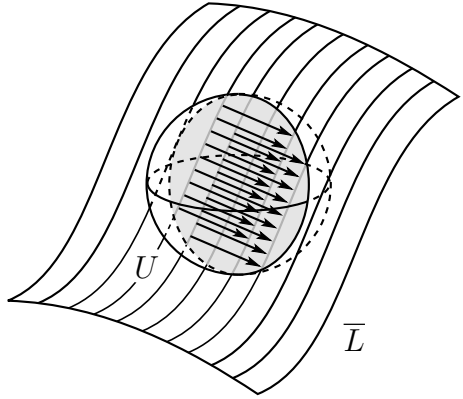}
\caption{The orbits of the Molino sheaf are the closures of the leaves}
\label{molinosheaf}
\end{figure}

\begin{example}[Molino sheaf of suspensions {\cite[Exemple III.1]{molino3}}]\label{example: molino sheaf of suspensions}
Let $S$ be a complete Riemannian manifold and let $\f$ be the Riemannian foliation of $M=\widehat{B}\times_{\pi_1(B)}S$ defined by the suspension of $h:\pi_1(B)\to\mathrm{Iso}(S)$ (see Examples \ref{example: suspensions} and \ref{example: Riemannian suspensions}). Denote $G=\overline{h(\pi_1(B))}$, let $\mathscr{H}$ be the pseudogroup of local isometries of $S$ generated by $G$ and consider its sheaf $\mathfrak{iso}_\mathscr{H}$ of infinitesimal transformations, which is the constant sheaf given by the restrictions of the fields in the Lie algebra $\mathfrak{iso}_G$ (recall Example \ref{example: infinitesimal sheaf of lie group}). Then the inverse image of $\mathfrak{iso}_\mathscr{H}$ via the projection $\widetilde{M}\to S$ is a constant sheaf $\widetilde{\mathfrak{iso}}_\mathscr{H}$ on $\widetilde{M}$ whose sections are restrictions of the transverse fields in the pullback $\widetilde{\mathfrak{iso}}_G < \mathfrak{l}(\widetilde{\f})$ of $\mathfrak{iso}_G$. The Molino sheaf $\mathscr{C}_\f$ coincides with the direct image, by $\pi:\widetilde{M}\to M$, of $\widetilde{\mathfrak{iso}}_\mathscr{H}$.
\end{example}

\section{Molino Theory}\label{subsection: Molino Theory}

Molino theory consists of a structural theory for Riemannian foliations developed by P.~Molino and others in the decade of 1980. In this section we summarize it, following mostly the brief presentations in \cite[Section 4.1]{goertsches} and \cite[Section 3.2]{toben}. A thorough introduction can be found in \cite{molino}. Roughly speaking, the fundamental underlying idea is that one can ``uncoil'' the holonomy of a Riemannian foliation by considering its action on the transverse frames. One obtains this way a simpler foliation, with trivial holonomy, which is intimately related to the original foliation.

More precisely, let $\pi^\Yup:M^\Yup\to M$ be the principal $\mathrm{O}(q)$-bundle of $\f$-transverse orthonormal frames\footnote{When $\f$ is transversely orientable, $M^\Yup$ consists of two $\mathrm{SO}(q)$-invariant connected components that correspond to the possible orientations. In this case we will assume that one component was chosen and, by abuse of notation, denote it also by $M^\Yup$. Everything stated in this section then will carry over to this case by changing $\mathrm{O}(q)$ to $\mathrm{SO}(q)$.}, which we call the \textit{Molino bundle} of $\f$. We lift $\f$ to a foliation $\f^\Yup$ of $M^\Yup$ as follows. The flow of a foliate vector field $X$ acts by foliate diffeomorphisms on $M$ and thus induces a flow $\nu\f\to\nu\f$ of bundle automorphisms. If $X$ is a foliate Killing vector field this flow further preserves $\mathrm{g}^T$, hence maps transverse orthonormal frames to transverse orthonormal frames, naturally inducing an $\mathrm{O}(q)$-equivariant flow on $M^\Yup$. The associated fundamental vector field $X^\Yup\in\mathfrak{X}(M^\Yup)$ of this flow is the \textit{natural lifting} of $X$. Notice that, in particular, every $X\in\mathfrak{X}(\f)$ is automatically a foliate Killing vector field and can be lifted. The image of $\mathfrak{X}(\f)$ under natural liftings spans an involutive bundle $T\f^\Yup$, whose integral foliation is the \textit{lifted foliation} $\f^\Yup$ we wanted.

Alternatively, $\f^\Yup$ can be described as follows: if $x=\pi^\Yup(x^\Yup)$ and $y=\pi^\Yup(y^\Yup)$ then $y^\Yup$ belongs to the leaf $L_{x^{\Yup}}\in\f^\Yup$ if and only if the orthonormal frame $y^\Yup$ of $\nu_y\f$ is the parallel transport of the frame $x^{\Yup}$, with respect to the Bott connection on $\nu\f$ (or $\nabla^B$ instead, recall the definition in \eqref{eq: conn basic}), along some smooth path in $L_x$ from $x$ to $y$.

\begin{figure}
\centering{
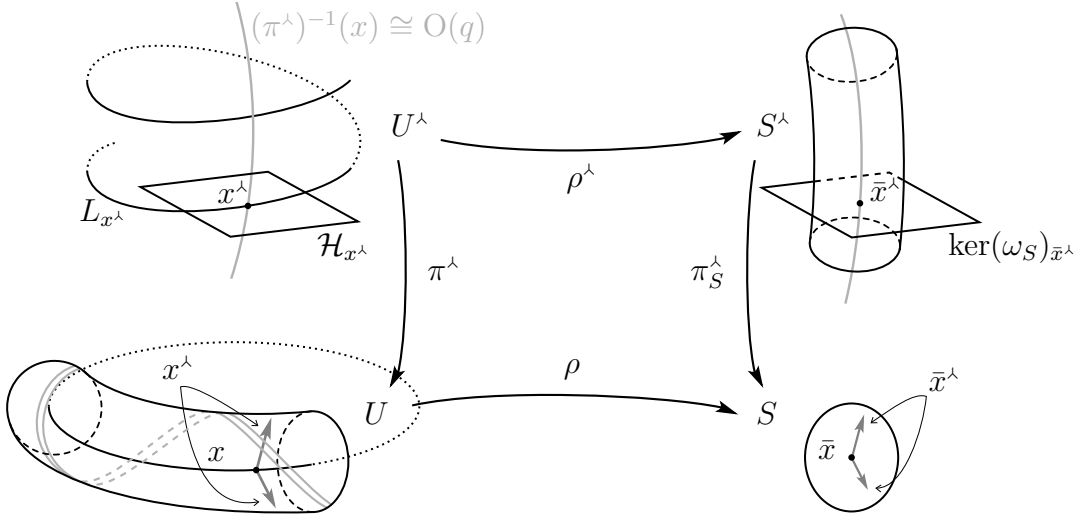}
\caption{The lifted foliation}
\label{liftedfoliation}
\end{figure}

The partial connection on $M^\Yup$ given by natural liftings can be extended to a unique $\f^\Yup$-basic torsion-free principal connection \cite[Lemma 3.3]{molino}, whose associated $\mathrm{O}(q)$-invariant connection $1$-form we denote by $\omega_{\f}\in\Omega(M^\Yup,\mathfrak{so}(q))$. The lifted foliation $\f^\Yup$ is, again, given by the horizontal liftings of the leaves of $\f$ with respect to the $\mathrm{O}(q)$-invariant horizontal distribution $\mathcal{H}=\ker(\omega_{\f})$. Let us provide a local description of $\omega_\f$. A locally defining submersion $\rho:U\to S$ for $\f$ induces a submersion $\rho^{\Yup}:U^\Yup\to S^\Yup$ whose fibers describe the restriction of $\f^{\Yup}$ to $U^{\Yup}=(\pi^{\Yup})^{-1}(U)$ (here $\pi^\Yup_S:S^\Yup\to S$ is the $O(q)$ principal bundle of orthonormal frames of $(S, \rho_*\mathrm{g}^T)$). Note that these maps commute, i.e., $\rho\circ\pi^\Yup=\pi^\Yup_S\circ\rho^\Yup$. If $\omega_S$ is the linear Riemannian connection induced on $S^\Yup$, then $\omega_{\f}=(\rho^\Yup)^*\omega_S$ (see Figure \ref{liftedfoliation}, keeping in mind that $U^{\Yup}$ should be $4$-dimensional there).

The Molino bundle also comes equipped with the \textit{tautological form} $\theta_{\f}:\nu\f^\Yup\to\mathbb{R}^q$ defined by $\theta_{\f}(X_{x^\Yup})=(x^\Yup)^{-1}(\dif\pi^\Yup(X_{x^\Yup}))$, where $x^\Yup$ is an orthonormal basis of $\nu_x\f$, understood as an isomorphism $x^\Yup:\mathbb{R}^q\to\nu_x\f$, and $X_{x^\Yup}\in\nu_{x^\Yup}\f^\Yup$. The tautological form $\theta_{\f}$ is $\f^\Yup$-basic \cite[Lemma 2.1(i)]{molino}, therefore (regarding $\omega_{\f}$ as a map $\nu\f^\Yup\to\mathfrak{so}(q)$) we get an $\f^\Yup$-basic, $\mathrm{O}(q)$-equivariant map $\omega_{\f}\oplus\theta_{\f}:\nu\f^\Yup\to\mathfrak{so}(q)\oplus\mathbb{R}^q$ which restricts to an isomorphism at each fiber $\nu_{x^\Yup}\f^\Yup$.

This allows us to define a natural transverse metric for the lifted foliation, as follows. The pullback of the sum of an arbitrary (which is unique up to scalar $\lambda$) bi-invariant scalar product on $\mathfrak{so}(q)$ with the standard scalar product on $\mathbb{R}^q$ by $\omega_{\f}\oplus\theta_{\f}$ yields an $\mathrm{O}(q)$-invariant transverse metric $(\mathrm{g}^T)^\Yup$ for $\f^\Yup$, which is hence a Riemannian foliation. We can fix $\lambda$ by requiring that the fibers of $\pi^\Yup$ satisfy $\vol((\pi^\Yup)^{-1}(x))=1$.

The advantage of lifting $\f$ to $\f^\Yup$ is that the latter admits a complete global transverse parallelism, that is, $\nu\f^\Yup$ is parallelizable by fields in $\mathfrak{l}(\f^\Yup)$ \cite[p. 82 and p. 148]{molino}. In fact, via $\omega_{\f}\oplus\theta_{\f}$, to choose such a transverse parallelism amounts to choosing bases for $\mathfrak{so}(q)$ and $\mathbb{R}^q$. If we assume that $\f$ is complete, then those fields admit complete representatives\footnote{Compare this with the definition of complete Riemannian foliations of Molino \cite[Remark on p. 88]{molino}.} in $\mathfrak{L}(\f^\Yup)$, since $M$ is complete and they have constant length with respect to $(\mathrm{g}^T)^\Yup$ \cite[Section 4.1]{goertsches}. From the theory of transversely parallelizable foliations it then follows that the partition $\overline{\f^\Yup}$ of $M^\Yup$ is a \textit{simple foliation}, that is, $W:=M^\Yup/\overline{\f^\Yup}$ is a manifold and $\overline{\f^\Yup}$ is given by the fibers of a locally trivial fibration $b:M^\Yup\to W$ \cite[Proposition 4.1']{molino}, the \textit{basic fibration}. Since $\f^\Yup$ is $\mathrm{O}(q)$-invariant, by continuity so is $\overline{\f^\Yup}$, hence the action of $\mathrm{O}(q)$ on $M^\Yup$ descends to an action on $W$ such that $b$ is now $\mathrm{O}(q)$-equivariant. A leaf closure $\overline{L}\in\overline{\f}$ is the image by $\pi^\Yup$ of a leaf closure of $\f^\Yup$, which implies that each leaf closure is an embedded submanifold of $M$ \cite[Lemma 5.1]{molino}\footnote{Molino's results are usually stated for a compact $M$, but completeness of $\f$ is sufficient (see \cite[Section 4.1]{goertsches} and \cite[Section 3.2]{toben}.}. Moreover, the leaf closures in $\overline{\f^\Yup}$ projecting by $b$ to the same $\mathrm{O}(q)$-orbit in $W$ all project over the same leaf closure in $\overline{\f}$. This induces an identification $M/\overline{\f}\equiv W/\mathrm{O}(q)$ and gives a commutative diagram (see Figure \ref{molinobundle}, again keeping in mind that $M^\Yup$ is $4$-dimensional)
$$\xymatrix{
(M^\Yup,\f^\Yup,\mathrm{O}(q)) \ar[r]^-{b} \ar[d]^{\pi^\Yup} &  (W,\mathrm{O}(q))\ar[d]\\
(M,\f) \ar[r] & M/\overline{\f}\equiv W/\mathrm{O}(q).}$$

\begin{figure}
\centering{
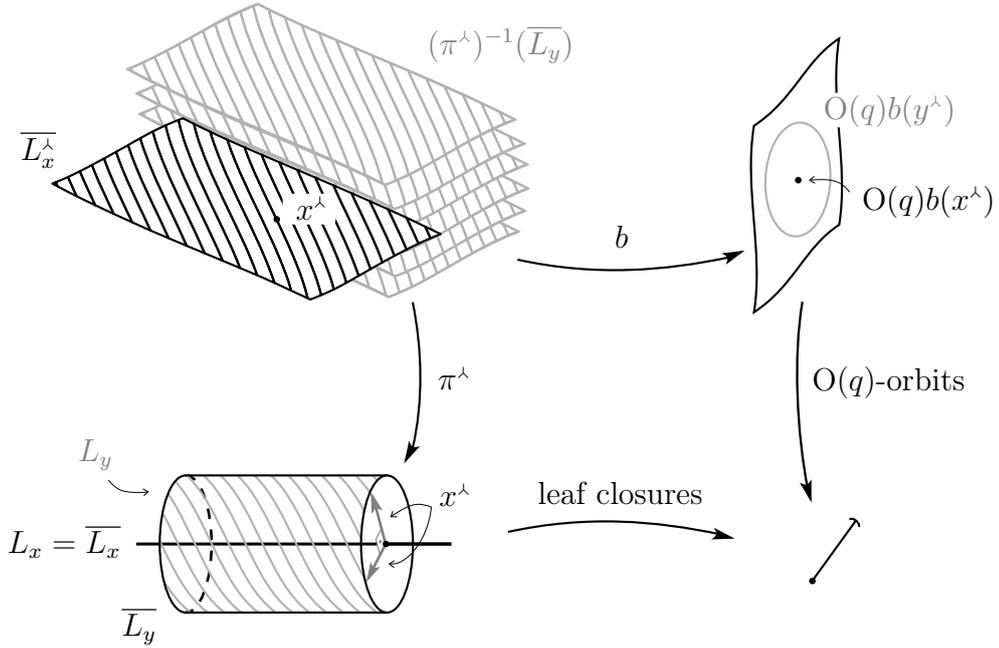}
\caption{The Molino construction}
\label{molinobundle}
\end{figure}

We now study the restriction of $\f$ to a leaf closure through this construction. Fix $L^\Yup\in\f^\Yup$, denote $J=\overline{L^\Yup}$, consider the foliation $(J,\f^\Yup|_J)$ and define $\mathfrak{g}:=\mathfrak{l}(\f^\Yup|_J)$. The restriction of $\f^\Yup$ to the closure of a different leaf is isomorphic to $(J,\f^\Yup|_J)$, so $\mathfrak{g}$ is an algebraic invariant of $\f$.

\begin{definition}[Structural algebra]
The Lie algebra $\mathfrak{g}$ is the \textit{structural algebra of $\f$}. We will always denote $\dim(\mathfrak{g})$ by $d$.
\end{definition}

The foliation $\f^\Yup|_J$ is a complete $\mathfrak{g}$-Lie foliation in the terminology of E.~Fedida \cite{fedida}, that is, it admits a complete transverse parallelism $\{\overline{Z_1},\dots,\overline{Z_d}\}$ such that the Lie algebra it spans is $\mathfrak{g}$. Equivalently, for a real Lie algebra $\mathfrak{g}$, a complete $\mathfrak{g}$-Lie foliation $\f$ on a manifold $J$ is given by an $\mathfrak{g}$-valued $1$-form $\alpha\in\Omega^1(J,\mathfrak{g})$ such that $\alpha_x:T_xJ\to\mathfrak{g}$ is surjective for each $x\in J$ and $d\alpha+\frac{1}{2}[\alpha,\alpha]=0$. For example, in the previous case $J=\overline{L^\Yup}$ the $1$-form $\alpha$ is given by
$$\alpha_x(X_x)=\xi_1\overline{Z_1}+\dots+\xi_d\overline{Z_d},$$
where $X_x=\xi_1Z_1+\dots+\xi_dZ_d+X_x^\top$ is the unique expression of $X_x\in T_xJ$ with $X_x^\top\in T_xL^\Yup$.

Fedida's work establishes, by a classical argument of C.~Ehresmann, that complete $\mathfrak{g}$-Lie foliations are \textit{developable}, that is, they lift to simple foliations on some covering space (see \cite[Theorem 4.1]{molino}). In fact, let $G$ be the unique simply connected Lie group with Lie algebra $\mathfrak{g}$ and consider $J\times G$ with projections $\mathrm{pr}_1$ and $\mathrm{pr}_2$ on the first and second factors, respectively. Let $\mathfrak{L}_\alpha:=\{X\in\mathfrak{L}(\f)\ |\ \alpha(X)\ \mbox{is constant}\}$ be the subalgebra of foliate vector fields whose corresponding transverse fields are in $\mathfrak{g}$. Via the identifications $T(J\times G)\cong TJ\oplus TG$ and $\mathfrak{g}\cong T_eG$, define the lift of $X\in\mathfrak{L}_\alpha$ by
$$\widetilde{X}=(X,\alpha(X)),$$
which is a $G$-invariant vector field on $J\times G$, with respect to the natural left action of $G$. This lifting is $\mathbb{R}$-linear and commutes with the Lie brackets, so the lift of $\mathfrak{L}_\alpha$ is a Lie algebra of left invariant vector fields which defines a left invariant integrable distribution $\Delta$ of rank $\dim(J)$ on $J\times G$. Let $\widetilde{J}$ be a leaf of the corresponding foliation.

\begin{theorem}[Fedida's theorem {\cite[Theorem 4.1]{molino}}]\label{theorem: fedida}
With the notation established above, $\mathrm{pr}_1:\widetilde{J}\to J$ is a covering map and $\mathrm{pr}_2:\widetilde{J}\to G$ is a locally trivial fibration. Moreover, the foliation $\mathrm{pr}_1^*(\f)$ on $\widetilde{J}$ agrees with the simple foliation defined by the fibers of $\mathrm{pr}_2$.
\end{theorem}

We see that a complete $\mathfrak{g}$-Lie foliation admits a Haefliger cocycle $(U_i,\pi_i:U_i\to G,\gamma_{ij})$ such that the transitions $\gamma_{ij}$ are restrictions of left translations on $G$. Moreover, since $\mathrm{pr}_2$ is $G$-equivariant, the holonomy pseudogroup of $\f$ is equivalent to the pseudogroup generated by the induced action of the group $\Gamma$ of deck transformations of $\mathrm{pr}_1$ on $G$.

Let us now return to a complete Riemannian foliation $(M,\f,\mathrm{g}^T)$. Consider on $M^\Yup$ the sheaf of Lie algebras $\mathscr{C}_{\f^\Yup}$ that, to an open set $U^\Yup\subset M^\Yup$, associates the Lie algebra $\mathscr{C}_{\f^\Yup}(U^\Yup)$ of the transverse fields in $U^\Yup$ that commute with all the global fields in $\mathfrak{l}(\f^\Yup)$. The orbits of $\mathscr{C}_{\f^\Yup}$ are the closures of the leaves of $\f^\Yup$ \cite[Theorem 4.3']{molino} and all stalks of $\mathscr{C}_{\f^\Yup}$ are isomorphic to the Lie algebra $\mathfrak{g}^{-1}$ opposed to the structural algebra $\mathfrak{g}$ of $\f$ \cite[Proposition 4.4]{molino}. Each field in $\mathscr{C}_{\f^\Yup}(U^\Yup)$ is the natural lift of a local $\f$-transverse Killing vector field on $\pi^\Yup(U^\Yup)$ \cite[Proposition 3.4]{molino}, which in turn is the lift of a section of the sheaf of infinitesimal transformations of $\mathscr{H}_\f$. So we conclude that the direct image $\pi^\Yup_*(\mathscr{C}_{\f^\Yup})$ coincides with the Molino sheaf $\mathscr{C}_\f$ (recall Definition \ref{def: molino sheaf}). In fact, this is how $\mathscr{C}_\f$ was originally defined by Molino\footnote{In Molino's terminology $\mathscr{C}_\f$ is called the \textit{commuting sheaf} \cite{molino}, also sometimes referred to as the \textit{central transverse sheaf} \cite{molino2}.}.

The stalks of $\mathscr{C}_{\f}$ and $\mathscr{C}_{\mathscr{H}_\f}$ are isomorphic, so the structural algebra of $\f$ coincides with the structural algebra of $\mathscr{H}_\f$. As we already stated, the main motivation for the study of $\mathscr{C}_{\f}$ is that its orbits describe the closures of the leaves of $\f$. In other words, this means that
$$\{X_x\ |\ X\in(\mathscr{C}_{\f})_x\}\oplus T_xL_x=T_x\overline{L}_x,$$
that is, for a small open set $U$, fixing a basis $\overline{X_1},\dots,\overline{X_d}$ for $\mathscr{C}_{\f}(U)$ we have $T\overline{L}|_U=TL|_U\oplus\spannn\{ X_1,\dots, X_d \}$ for any $L\in\f$, where $X_1,\dots, X_d\in\mathfrak{L}(\f)$ are representatives for that basis.

Let us summarize the properties seen in this section in the following theorem, known as Molino's structural theorem.

\begin{theorem}[Molino's structural theorem]\label{Theo: molino structural}
Let $\f$ be a complete Riemannian foliation of codimension $q$ of $M$. Then:
\begin{enumerate}[(i)]
\item The lifted foliation $\f^\Yup$ on the transverse frame bundle $M^\Yup$ is transversely parallelizable, hence $\overline{\f^\Yup}$ is s simple foliation, given by the fibers of the basic fibration $b:M^\Yup\to W$.
\item The restriction of $\f^\Yup|_J$ to a leaf closure $J=\overline{L^\Yup}$ is a complete $\mathfrak{g}$-Lie foliation.
\item The closures of the leaves of $\f$ are embedded submanifolds and coincide with the projections of the closures of the leaves of $\f^\Yup$.
\item The quotient $M/\overline{\f}$ can be identified with the orbit space $W/\mathrm{O}(q)$ of the $\mathrm{O}(q)$-action on $W$ induced by its natural action on $M^\Yup$.
\item There is a locally constant sheaf $\mathscr{C}_\f$ of Lie algebras of germs of transverse Killing vector fields whose stalks are $\mathfrak{g}^{-1}$ and whose orbits are the closures of the leaves of $\f$.
\end{enumerate}
\end{theorem}

\section{Killing foliations}\label{section: killing foliations}

From now on we will be mostly interested in the subclass of complete Riemannian foliations consisting of those foliations $\f$ for which $\mathscr{C}_{\f}$ is globally constant. Such foliations are called \textit{Killing foliations}, following the terminology of W.~Mozgawa in \cite{mozgawa}. In other words, if $\f$ is a Killing foliation then there exists global fields $\overline{X_1},\dots,\overline{X_d}\in\mathscr{C}_\f(M)<\mathfrak{iso}(\f)$ (global sections of $\mathscr{C}_\f$) such that
$$T\overline{\f}=T\f\oplus\spannn\{ X_1,\dots, X_d \}.$$
In particular, notice that any closed Riemannian foliation is Killing.

Let us understand what the definition means from the point of view of the holonomy pseudogroup $\mathscr{H}_\f$. Since $\mathscr{C}_{\f}$ is obtained from the gluing of the pullbacks of the Molino sheaf of $\mathscr{C}_{\mathscr{H}_\f}$ of $\mathscr{H}_\f$ by the local submersions defining $\f$, for $\mathscr{C}_{\f}$ to be constant $\mathscr{C}_{\mathscr{H}_\f}$ has to admit a global trivialization which is invariant by holonomy. That is, the global sections given by the trivialization have to be $\mathscr{H}_\f$-invariant so that they lift to global $\f$-transverse fields on $M$. Let us see this in detail in the case of a pseudogroup generated by the action of a Lie group $G$ (recall Example \ref{example: infinitesimal sheaf of lie group}).

\begin{proposition}\label{prop-equivalenciaAbeliano}
Let $\mathscr{H}$ be the pseudogroup of local isometries generated by a connected subgroup $G$ of isometries of a Riemannian manifold $M$. Then $\mathscr{H}$ is a Killing pseudogroup if, and only if, $G$ is Abelian.
\end{proposition}

\begin{proof}
First recall from Examples \ref{example: Salem closure pseudogroup generated by lie group} and \ref{example: infinitesimal sheaf of lie group} that the sections of $\mathscr{C}_{\mathscr{H}}=\mathfrak{iso}_{\overline{\mathscr{H}}}$ are the restrictions of the fundamental fields of the action of the closure $\overline{G}<\mathrm{Iso}(S)$. Hence $\mathscr{C}_{\mathscr{H}}$ admits an invariant global trivialization if, and only if, $\dif g_xX^{\#}_x=X^{\#}_{gx}$, for all $x\in M$, $g\in \overline{G}$ and $X\in\overline{\mathfrak{g}}$.

Using that $\exp(t\Ad_gX)=g\exp(tX)g^{-1}$ one verifies that in general
\begin{equation}
\label{eq1C-prop-equivalenciaAbeliano}
\dif g_x X^{\#}_x=(\Ad_gX)^{\#}_{gx}.
\end{equation}
Recall also that a connected Lie group is Abelian if, and only if, its adjoint representation is trivial (see, e.g., \cite[Section 1.3]{alex}). Therefore, if $G$ (hence $\overline{G}$) is Abelian, it follows from equation \eqref{eq1C-prop-equivalenciaAbeliano} that $\dif g_x X^{\#}_x=X^{\#}_{gx}$, hence $\mathscr{H}$ is Killing.

Conversely, assume $\mathscr{H}$ is Killing. Then combining equation \eqref{eq1C-prop-equivalenciaAbeliano} with the hypothesis $\dif g_x X^{\#}_x=X^{\#}_{gx}$ we obtain $\dif g_x X^\#_x = X^\#_{gx} = (\Ad_g X)^\#_{gx}=\dif g_x(\Ad_gX)^\#_x$, hence, as $\dif g_x$ is an isomorphism,
\begin{equation}
\label{eq2C-prop-equivalenciaAbeliano}
X^\#_x=(\Ad_gX)^\#_x
\end{equation}
for all $x\in M$, $g\in \overline{G}$ and $X\in\overline{\mathfrak{g}}$. Recall that $X^\#_x=\dif \mu_x X$, where $\mu_x$ is the orbit map $G\ni g\mapsto gx\in M$. Therefore we have from \eqref{eq2C-prop-equivalenciaAbeliano} that $\dif\mu_xX=\dif\mu_x\Ad_gX$, hence $X-Ad_gX\in\ker(\dif\mu_x)=T_eG_x$, for all $x\in M$. It then follows that $X=\Ad_gX$, since $\bigcap_{x\in M}T_eG_x=\{0\}$ because the $\overline{G}$-action is effective. Therefore the adjoint representation is trivial and $\overline{G}$ (hence $G$) is Abelian.
\end{proof}

The existence of an $\mathscr{H}_\f$ invariant trivialization for $\mathscr{C}_{\mathscr{H}_\f}$ can be expressed more elegantly as follows. There is a natural action of a complete pseudogroup of local isometries $\mathscr{H}$ on its Molino sheaf $\mathscr{C}_{\mathscr{H}}=\mathfrak{iso}_{\overline{\mathscr{H}}}$, the action of $h\in\mathscr{H}$ on a local section $X\in\mathscr{C}_{\mathscr{H}}$ being given by
$$h\cdot X=\left.\od{}{t} h\circ\exp(tX)\circ h^{-1}\right|_{t=0}.$$
We say that $\mathscr{H}$ is a \textit{Killing pseudogroup} if $\mathscr{C}_{\mathscr{H}}$ admits a global trivialization which is invariant by the action of $\mathscr{H}$. Then we have, by our previous discussion, that $\f$ is a Killing foliation if, and only if, $\mathscr{H}_\f$ is a Killing pseudogroup. This provides a criterion for a Riemannian foliation given by a suspension to be Killing.

\begin{example}[Killing foliations given by suspension]\label{example: killing foliations given by suspensions}
Let $\f$ be the Riemannian foliation of $M=\widehat{B}\times_{\pi_1(B)}S$ defined by the suspension of $h:\pi_1(B)\to\mathrm{Iso}(S)$, where $S$ is complete (recall Examples \ref{example: suspensions} and \ref{example: Riemannian suspensions}). Denote $G=\overline{h(\pi_1(B))}<\mathrm{Iso}(S)$ and let $\mathscr{H}$ be the pseudogroup generated by $G$ on $S$. As we saw in Example \ref{example: molino sheaf of suspensions}, the Molino sheaf $\mathscr{C}_{\f}$ is the image by $\pi:\widetilde{M}\to M$ of the constant sheaf $\widetilde{\mathfrak{iso}}_\mathscr{H}$, which in turn is the inverse image of $\mathfrak{iso}_\mathscr{H}$. Therefore $\mathscr{C}_{\f}$ is globally constant if, and only if, the constant sheaf $\widetilde{\mathfrak{iso}}_\mathscr{H}$ is $\pi_1(B)$-invariant, which in turn happens if, and only if, $\mathfrak{iso}_\mathscr{H}$ is $\mathscr{H}$-invariant, i.e., $\mathscr{H}$ is a Killing pseudogroup. Thus, by Proposition \ref{prop-equivalenciaAbeliano}, in order for $\f$ to be Killing it is sufficient that $G$ be connected and Abelian.
\end{example}

As Example \ref{example: killing foliations given by suspensions} suggests, the structural algebra of any Killing foliation is necessarily Abelian. This can be seen via $\mathscr{C}_{\mathscr{H}_\f}$ by generalizing the arguments in Proposition \ref{prop-equivalenciaAbeliano} (e.g. using \cite[Chapter 2]{agrachev}) or, more quickly, by using the fact that $\mathscr{C}_\f$ is the direct image of the sheaf $\mathscr{C}_{\f^\Yup}$ of the lifted foliation $\f^\Yup$. Then one sees that a complete Riemannian foliation is a Killing foliation if and only if $\mathscr{C}_{\f^\Yup}$ constant, and in this case, by definition, $\mathscr{C}_{\f^\Yup}(M^\Yup)$ is the center of $\mathfrak{l}(\f^\Yup)$. Hence $\mathscr{C}_{\f}(M)$ is central in $\mathfrak{l}(\f)$ (but not necessarily its full center). The structural algebra of $\f$ is thus is Abelian, because $\mathfrak{g}^{-1}\cong(\mathscr{C}_{\f})_x\cong\mathscr{C}_{\f}(M)$ for any $x\in M$. For this reason, when $\f$ is Killing we will often denote its structural algebra by $\mathfrak{a}$.

\begin{example}[Riemannian foliations on simply-connected manifolds]
A complete Riemannian foliation $\f$ of a simply-connected manifold is automatically a Killing foliation \cite[Proposition 5.5]{molino}, since in this case $\mathscr{C}_{\f}$ cannot have holonomy. In fact, for $\f$ to be Killing it is sufficient that $\pi_1(\mathscr{H}_\f)$ be trivial (i.e., that $\mathscr{H}_\f$ be simply connected). The fundamental group of a pseudogroup is a generalization of the usual notion of fundamental group, defined in terms of $\mathscr{H}$-homotopy classes of $\mathscr{H}$-loops in $S$., i.e., finite collections of paths on $S$ whose endpoints are glued by elements of $\mathscr{H}$ (details can be seen in \cite[Sections 1.11 and Remark 3.8]{salem}). For the case of a foliation there is a surjective homomorphism $\pi_1(M)\to\pi_1(\mathscr{H}_\f)$, hence the condition on $\pi_1(\mathscr{H}_\f)$ for $\f$ to be Killing is weaker than that of $M$ being simply connected.
\end{example}

\begin{example}[Isometric homogeneous foliations {\cite[Lemme III]{molino3}}]\label{example: homogeneous foliations are killing}
Homogeneous Riemannian foliations provide another important class of examples. In fact, if $\f$ is a Riemannian foliation of a compact manifold $M$ given by the foliated action of $H<\mathrm{Iso}(M)$, then $\f$ is a Killing foliation because its Molino sheaf $\mathscr{C}_{\f}(M)$ consists of the transverse Killing vector fields induced by the action of $\overline{H}<\mathrm{Iso}(M)$, hence is constant. Notice the contrast with Proposition \ref{prop-equivalenciaAbeliano}: here $H$ is not necessarily Abelian, since we are not interested in the pseudogroup of local isometries generated by $H$, but rather the holonomy pseudogroup of $\f$. We already saw specific examples in this class of Killing foliations: the $\lambda$-Kronecker foliations (see Example \ref{exe: foliated actions}) and the Riemannian $1$-foliations of the round sphere (see Example \ref{example: 1foliations of the sphere}).
\end{example}

One can construct examples of Killing foliations which are not homogeneous and whose ambient manifolds are not simply connected by using suspensions. For example, take $S$ to be an inner product vector space and $B$ a negatively curved compact Riemannian manifold whose fundamental group has a nontrivial Abelian subgroup $\langle \gamma\rangle$ (which is infinite cyclic, by Preissman's theorem). Define $h$ on the generators by mapping $\gamma$ to an irrational rotation and any other generator to the identity. The foliation defined by suspension of $h$ is then a Killing foliation, by Example \ref{example: killing foliations given by suspensions}. It is non-homogeneous, since it has the zero section $L_0\cong B$ as one of its leaves, which is a non-homogeneous manifold (since $\mathrm{Iso}(B)$ is finite), and the total space $M$ is not simply connected, since it deformation retracts to $L_0$.

Finally, we cite the following example of a non-homogeneous Killing foliation on a non-simply connected manifold which is moreover not constructed by the suspension method.

\begin{example}[{\cite[p. 287]{mozgawa}}]\label{example: a killing foliation non-homogeneous and not simply connected}
Consider $T=\mathbb{T}^2\times\mathbb{T}^2$. For $A\in\mathrm{SL}_2(\mathbb{Z})$, if $v$ is an eigenvector of $A$, the foliation given by lines in $\mathbb{R}^2$ that are parallel to $v$ projects to a Kronecker foliation $\f_v$ of $\mathbb{T}^2$ (see Example \ref{exe: foliated actions}). We choose $A$ so that $\f_v$ is not closed, e.g. by requiring that $\mathrm{tr}(A)>2$. Seeing this torus $\mathbb{T}^2$ as the second factor of $T$, the product foliation of the trivial foliation $\{\mathbb{T}^2\}$ on the first factor with $\f_v$ gives us a codimension $1$ foliation $\f_T$ of $T$ with dense leaves. Consider the diffeomorphism $\Phi_A:=\mathrm{id}\times\overline{A}:T\to T$, where $\overline{A}:\mathbb{T}^2\to\mathbb{T}^2$ is the diffeomorphism determined by $A$. The suspension of the homomorphism $\pi_1(\mathbb{S}^1)\to \mathrm{Diff}(T)$ given by $n\mapsto \Phi_A^n$ furnishes us a fiber bundle $\tau:M\to \mathbb{S}^1$ with fiber $T$ and structural group $\langle\Phi_A\rangle$. Here we are not interested in the foliation given by this suspension, but rather the foliation $\f$ induced fiberwise on $M$ by $\f_T$, which is well defined since $\f_T$ is invariant by $\Phi_A$. One sees immediately that $\f$ is Riemannian and its leaf closures are the fibers of $\tau$.

Since $\Phi_A$ acts trivially on $\mathfrak{l}(\f_T)$, one sees that $\f_T$ is transversely parallelizable and it follows that $\mathscr{C}_\f$ is globaly trivial, that is $\f$ is a Killing foliation. Notice, however, that $M$ is not simply connected, by construction. It only remains to verify that $\f$ is also not homogeneous. In fact, if this were the case, $\f$ would be given by the orbits of a connected Lie subgroup $H<\mathrm{Iso}(M,\mathrm{g})$, with respect to some Riemannian metric $\mathrm{g}$ on $M$. Then $\overline{\f}$, given by the fibers of $\tau$, would coincide with the orbits of $\overline{H}$, hence one could conclude that $\tau$ is associated to a principal $\overline{H}$-bundle $E\to\mathbb{S}^1$. Since $\overline{H}$ is connected, $E$ should be trivial, hence also $M\to \mathbb{S}^1$ would be trivial. But this does not happen by construction: the map $\Phi_A^*:H(T)\to H(T)$ induced by the generator $\Phi_A$ of its structural group on the homology of the fibers is non-trivial, hence $M\to \mathbb{S}^1$ is not topologically trivial.
\end{example}

In \cite{mozgawa} W.~Mozgawa establishes some implications of Molino's structural theorems in the case of Killing foliations:

\begin{theorem}[Mozgawa's Theorem {\cite[Théorème]{mozgawa}}]\label{theorem: Mozgawa}
Let $\f$ be a $q$-codimensional Killing foliation of a compact manifold $M$ of dimension $n=p+q$. If $r+p=\min_{\overline{L}\in\overline{\f}}\dim(\overline{L})$ then:
\begin{enumerate}[(i)]
\item There exists $r$ commuting transverse Killing vector fields $\overline{X_1},\dots,\overline{X_r}\in\mathfrak{iso}(\f)$ which are everywhere linearly independent, and
\item The orbits of the Lie algebra $\spannn(\overline{X_1},\dots,\overline{X_r})$ define a Riemannian foliation $\f'$ of $M$ of codimension $q-r$ which has at least one closed leaf and satisfies $\overline{\f'}=\overline{\f}$.
\end{enumerate}
\end{theorem}

In particular, it follows easily from this theorem that if $\chi(M)\neq 0$, then every Killing foliation $\f$ on $M$ has at least one closed leaf: in that case, by the Hopf index theorem any vector field (hence any $\overline{X}\in\mathfrak{l}(\f)$) must vanish at some point, where it is thus not linearly independent.  More recently it was shown in \cite{caramello} that a much stronger conclusion holds when $M$ is compact: if $\chi(M)\neq 0$ then every leaf of $\f$ is closed (see Theorem \ref{theorem: closed leaf + transverse symmetry implies charM vanishes}).

\subsection{Transverse structure of Killing foliations}

The transverse structure of a Killing foliation coincides with that of an (Abelian) homogeneous foliation on an orbifold, as established by A.~Haefliger and E.~Salem in \cite{haefliger2}. More precisely, by comparing the local models of the transverse structure of a Killing foliation on a neighborhood of a leaf closure and the local model of an orbit of a torus action on an orbifold, the authors obtain the following.

\begin{theorem}[Haefliger--Salem Theorem {\cite[Theorem 3.4]{haefliger2}}]\label{theorem: Haefliger-Salem 3.4}
There are canonical correspondences between:
\begin{enumerate}[(i)]
\item The set $A_1$ of equivalence classes of Killing foliations $\f$ with compact leaf closures on a manifold $M$, two foliations being equivalent when their holonomy pseudogroups are equivalent,
\item\label{HS item 2} The set $A_2$ of equivalence classes of Killing pseudogroups $\mathscr{H}$ such that $\mathscr{H}$ restricted to a generic orbit closure is equivalent to the pseudogroup generated by a rank $N$ subgroup $\Gamma$ of translations of $\mathbb{R}^d$,
\item the set $A_3$ of equivalence classes of quadruples $(\mathcal{O},\mathbb{T}^N,H,\mu)$, where $\mathcal{O}$ is an orbifold, $\mu:\mathbb{T}^N\times\mathcal{O}\to \mathcal{O}$ is an effective action and $H<\mathbb{T}^N$ is a dense, contractible subgroup whose action is locally free, two quadruples $(\mathcal{O},\mathbb{T}^N,H,\mu)$ and $(\mathcal{O}',\mathbb{T}^{N'},H',\mu')$ being equivalent if there is an isomorphism between $\mathbb{T}^N$ and $\mathbb{T}^{N'}$ (sending $H$ to $H'$) and a diffeomorphism of $\mathcal{O}$ onto $\mathcal{O}'$ that conjugates $\mu$ and $\mu'$.
\end{enumerate}
Moreover, for a foliation $(M,\f)$ whose class is in $A_1$, there is a smooth map $\Upsilon:M\to\mathcal{O}$, for $\mathcal{O}$ a corresponding orbifold whose class is in $A_3$, such that $\f=\Upsilon^*(\f_H)$, where $\f_H$ is the foliation of $\mathcal{O}$ given by the orbits of $H$.
\end{theorem}

The correspondences $A_1\to A_2$ and $A_3\to A_2$ are just $[\f]\mapsto [\mathscr{H}_\f]$ and $[\f_H]\mapsto[\mathscr{H}_{\f_H}]$, respectively. Notice that $A_1$ indeed maps to $A_2$: the restriction of $\f$ to a generic leaf closure is a complete $\mathfrak{a}$-Lie foliation, since $\mathscr{C}_{\f}(M)$ restricts to a complete transverse parallelism for it, so it follows from Theorem \ref{theorem: fedida} that $\mathscr{H}$ restricted to a generic orbit closure is generated by subgroup $\Gamma$ of translations of $\mathbb{R}^d$. The isomorphism
$$\mathbb{T}^N\cong \frac{\Gamma\otimes\mathbb{R}}{\Gamma\otimes\mathbb{Z}},$$
where $\Gamma$ is the corresponding subgroup of translations of $\mathbb{R}^d$ in $A_2$ helps clarifying the relation between $\mathbb{T}^N$ and $\f$. Note, in particular, that $N\geq d$.

The existence of $\Upsilon$ follows non-trivially from the theory of classifying spaces of pseudogroups, developed by Haefliger in \cite{haefliger3}: the classifying space of $\mathscr{H}_\f$ is a space $B\mathscr{H}_\f$ with a foliation $B\f$ such that the holonomy covering of each leaf is contractible and $\mathscr{H}_\f\cong\mathscr{H}_{B\f}$. As in the classical case of classifying spaces in homotopy theory, there is a map $\Upsilon:M\to B\mathscr{H}_\f$, whose homotopy class is unique up to homotopy along the leaves, which is transverse to $B\f$ and such that $\f=\Upsilon^*(B\f)$. The point is that $A_1\to A_3$ associates the class of $\f$ to a canonical representative $(\mathcal{O},\f_H)$ of $[(B\mathscr{H}_\f,B\f)]$.

\begin{example}\label{example: Hae-Sal transverse structure of gen Hopf fibration}
In the simple case of an irrational generalized Hopf fibration $\f$ of $\mathbb{S}^3$ (see Eaxample \ref{example: 1foliations of the sphere}), the construction of $(\mathcal{O},\mathbb{T}^N,H,\mu)$ is trivial: $\mathcal{O}=\mathbb{S}^3$ with the action of $\mathbb{T}^N=\mathbb{T}^2$ by restriction of the multiplication on $\mathbb{C}^2$, and $H$ is the subgroup determined by the $\mathbb{R}$-action that defines the foliation. To also illustrate item \eqref{HS item 2} of Theorem \ref{theorem: Haefliger-Salem 3.4}, recall that the restriction of $\f$ to the closure of a generic leaf is an irrational Kronecker foliation $\f(\lambda)$ (see Example \ref{exe: foliated actions}). Notice that $\f(\lambda)$ is a Lie $\mathbb{R}$-foliation, so in view of Theorem \ref{theorem: fedida}, $\mathscr{H}_{\f(\lambda)}$ is equivalent to the pseudogroup generated by the group $\Gamma$ of translations of $\mathbb{R}$ induced, via projection along the lifted foliation $\widetilde{\f}(\lambda)$ of the universal covering $\mathbb{R}^2$, by the action of $\pi_1(\mathbb{T}^2)$ by deck transformations. Notice that in fact we have $\rank(\Gamma)=2=\rank(\pi_1(\mathbb{T}^2))$, since the generators of $\pi_1(\mathbb{T}^2)$ project to rationally independent translations $a_1$ and $a_2$ of $\mathbb{R}$ (see Figure \ref{liftedkronecker}).

\begin{figure}
\centering{
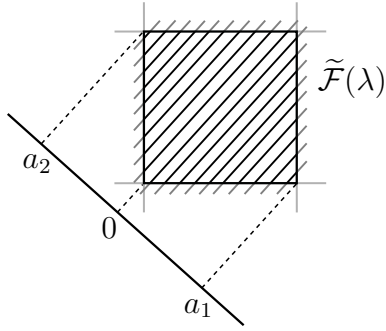}
\caption{Generators of $\Gamma$}
\label{liftedkronecker}
\end{figure}

More generally, for a Killing foliation with compact leaf closures $\f$, if $L\in\f$ is a generic leaf, the authors establish in \cite[Theorem 1.4]{haefliger2} that $\dim(\overline{L})-N\geq 0$, with equality holding if and only if $L$ is contractible, and in this case $\mathcal{O}$ is a manifold, $\dim(M)=\dim(\mathcal{O})$, and $\Upsilon$ is a homotopy equivalence.
\end{example}

\subsection{Deformations of Killing foliations}\label{section: deformations}

Two smooth foliations $\f_0$ and $\f_1$ of $M$ are $C^\infty$-\textit{homotopic} if there is a smooth foliation $\f$ of $M\times [0,1]$ of the same dimension such that $M\times\{t\}$ is saturated by leaves of $\f$, for each $t\in[0,1]$, and
$$\f_i=\f|_{M\times\{i\}},$$
for $i=0,1$. Here we will simply say that $\f_t$ is a \textit{deformation} of $\f_0$ into $\f_1$.

For a Riemannian foliation $\f$ on a simply connected, compact manifold $M$, E.~Ghys showed in \cite[Théorème 3.3]{ghys} that is possible to deform $\f$ into a closed foliation $\g$, in such a way that the deformation respects $\overline{\f}$, that is, it occurs within the closures of the leaves of $\f$. As remarked by the authors in \cite{haefliger2}, Theorem \ref{theorem: Haefliger-Salem 3.4} generalizes this result: for a Killing foliation $\f$ on a compact manifold $M$, consider a corresponding orbifold $(\mathcal{O},\mathbb{T}^N,H,\mu)$ and the map $\Upsilon:M\to\mathcal{O}$ such that $\f=\Upsilon^*(\f_H)$. Let $\mathfrak{h}$ be the Lie algebra of $H$ and slightly perturbate it into a Lie subalgebra $\mathfrak{k}<\mathrm{Lie}(\mathbb{T}^N)\cong\mathbb{R}^N$, with $\dim(\mathfrak{k})=\dim(\mathfrak{h})$, such that its corresponding Lie subgroup $K<\mathbb{T}^N$ is closed. If $\mathfrak{k}$ is close enough to $\mathfrak{h}$ (as points in the Grassmannian $\mathrm{Gr}^{\dim\mathfrak{h}}(\mathrm{Lie}(\mathbb{T}^N))$), it is possible to choose a smooth path $\mathfrak{h}(t)$ connecting $\mathfrak{h}$ to $\mathfrak{k}$ such that for each $t$ the action $\mu|_{H(t)}$ of the corresponding Lie subgroup $H(t)$ is locally free and the induced foliation $\f_{H(t)}$ remains transverse to $\Upsilon$. Then $\f_t:=\Upsilon^*(\f_{H(t)})$ defines a deformation of $\f=\f_0$ into $\g=\f_1$. It is possible to prove that $\mathscr{H}_{\f_t}$ is equivalent to $\mathscr{H}_{\f_{H(t)}}$ for each $t$. Moreover, since $K$ is closed, $\g$ is a closed foliation and, by construction, the deformation respects $\overline{\f}$.

The ``transverse homogeneous'' nature of this deformation allows one to preserve some geometric properties of $\f$ in $\g$. This was investigated in \cite{caramello}:

\begin{theorem}[{\cite[Theorem B]{caramello}}]\label{theorem: deformation}
Let $(\f,\mathrm{g}^T)$ be a Killing foliation of a compact manifold $M$. Then there is a deformation $\f_t$ of $\f$ respecting $\overline{\f}$, called a \emph{regular deformation}, into a closed foliation $\g$ which can be chosen arbitrarily close to $\f$, such that
\begin{enumerate}[(i)]
\item \label{injection of tensor algebra} for each $t$ there is an injection $\iota:\mathcal{T}(\f)\to\mathcal{T}(\f_t)$ that smoothly deforms transverse geometric structures given by $\f$-basic tensors, such as the metric $\mathrm{g}^T$, into respective transverse geometric structures for $\f_t$,
\item \label{toric action on the quotient} the quotient orbifold $M/\!/\g$ admits an effective isometric action of a torus $\mathbb{T}^d$, with respect to the metric induced from $\iota\mathrm{g}^T$, such that $M/\overline{\f}\cong(M/\g)/\mathbb{T}^d$, where $d=\dim\mathfrak{a}$.
\item \label{symmetries and basic tensors} $\mathcal{T}(\f)$ is isomorphic to the algebra $\mathcal{T}(M/\!/\g)^{\mathbb{T}^d}$ of $\mathbb{T}^d$-invariant tensor fields on $M/\!/\g$, the isomorphism being given by $\pi_*\circ\iota$, where $\pi_*:\mathcal{T}(\g)\to\mathcal{T}(M/\!/\g)$ is the pushforward by the canonical projection.
\end{enumerate}
In particular, if $\g$ is chosen sufficiently close to $\f$, upper and lower bounds on transverse sectional and Ricci curvature of $\f$ are maintained.
\end{theorem}

One can then use the Riemannian geometry and topology of $M/\!/\g$ to study $\f$. In the next sections we will summarize some applications of this technique.

\section{Transverse topology of Killing foliations}\label{section: transverse topology of killing}

In this section we will survey some recent results concerning the transverse algebraic topology of Riemannian and Killing foliations. We begin with results on the basic Euler characteristic of a Riemannian foliation $\f$ of a compact manifold $M$. Recall from Section \ref{section: basic cohomology} that
$$\chi(\f)=\sum_i(-1)^i\dim(H^i(\f))$$
is always well defined for such an $\f$ (see Theorem \ref{theorem: dim basic cohomology is finite}) and generalizes the usual Euler characteristic in the sense that $\chi(\f)=\chi(M)$ when $\f$ is the trivial foliation by points. In this particular case the classical Hopf index theorem states that for a vector field $X\in\mathfrak{X}(M)$ with isolated zeros one has $\chi(M)=\sum_{p}\mathrm{ind}_{p}(X)$, where the sum ranges over the set $\Zero(X)$ of zeros of $X$. This theorem was generalized to Riemannian foliations in \cite[Theorem 3.18]{belfi}. To state it precisely we will need some definitions. Endow $M$ with a bundle-like metric and fix a foliate vector field $X\in\mathfrak{L}(\f)$. A leaf closure $J=\overline{L}$ is \textit{critical for $X$} if $\overline{X}=0$ over $J$ (which by continuity happens if, and only if, $X$ is tangent to $L$ at all its points). We say that $X$ is \textit{$\f$-nondegenerate} when its \textit{linear part} $X_{\mathrm{Lin}}:\nu_x J\to\nu_x J$, given by $v\mapsto [V,X]^{\perp}_x$ (where $V\in\mathfrak{X}(M)$ is any extension of $v$), is an isomorphism for every point $x$ of each critical leaf closure. In this case the leaf closures are isolated (hence finite) and we define the \textit{index of $X$ at $J$} by $\mathrm{ind}_{J}(X)=\mathrm{sgn}(\det(X_{\mathrm{Lin}}))$. It coincides with the classical index of a vector field when $\f$ is the trivial foliation by points.

One could then expect that the transverse version of the Hopf index theorem would simply state that $\chi(\f)=\sum_{J}\mathrm{ind}_{J}(X)$ for an $\f$-nondegenerate $X$, but this is not the case. Since the transverse analog of a classical critical point is a leaf closure, some information from its topology must also be taken into account. This information is encoded in $\chi(J,\f,\mathrm{Or}_{J}(X))$, the alternate sum of the cohomology groups of the complex of $\f|_J$-basic forms with values in the orientation line bundle of $X$ at $J$ (for more details, see \cite[Section 3]{belfi}). We can now state:

\begin{theorem}[Basic Hopf index theorem {\cite[Theorem 3.18]{belfi}}]\label{theorem: basic hopf index} Let $\f$ a Riemannian foliation of a compact manifold $M$. If $X\in\mathfrak{L}(\f)$ is $\f$-nondegenerate, then
$$\chi(\f)=\sum_{J}\mathrm{ind}_{J}(X)\chi(J,\f,\mathrm{Or}_{J}(X)),$$
where the sum ranges over all critical leaf closures $J$ of $\f$.
\end{theorem}

By constructing an appropriate $X\in\mathfrak{L}(\f)$, one can use Theorem \ref{theorem: basic hopf index} to show that $\chi(\f)$ localizes to the strata of closed leaves:

\begin{theorem}[{\cite[Theorem D]{caramello}}]\label{theorem: basic euler char localizes to closed leaves} If $\f$ is a Killing foliation of a compact manifold $M$, then
$$\chi(\f)=\chi(\Sigma^{\dim\f}/\f).$$
In particular, if $\f$ has no closed leaves, then $\chi(\f)=0$.
\end{theorem}

In fact, in \cite[Theorem 7.1]{caramello} the authors prove something stronger: if $\overline{X}\in\mathfrak{iso}(\f)$, then $\chi(\f)=\chi(\f|_{\Zero(\overline{X})})$. This is in analogy to the classical localization of the Euler characteristic of a Riemannian manifold to the zero set of a Killing vector field (see, e.g. \cite[Theorem 40]{petersen}) or, alternatively, to the fixed point set of a torus action.

Combining Theorems \ref{theorem: basic euler char localizes to closed leaves} and \ref{theorem: deformation}, if $\g$ is a closed foliation approximating $\f$, then we have that
$$\chi(\g)=\chi(M/\g)=\chi\left((M/\g)^{\mathbb{T}^d}\right)=\chi(\Sigma^{\dim(\f)}/\f) = \chi(\f).$$
In fact, this holds for any $t$, so it proves the following.

\begin{theorem}[{\cite[Theorem 7.4]{caramello}}]\label{theorem: basic euler char is preserved by deformations}
Let $\f$ be a Killing foliation of a compact manifold $M$ and let $\f_t$ be a regular deformation. Then $\chi(\f_t)$ is constant in $t$.
\end{theorem}

In particular, for the closed foliation $\g=\f_1$, Theorem \ref{theorem: basic euler char is preserved by deformations} reduces questions about $\chi(\f)$ to questions about $\chi(M/\!/\g)$, which also coincides with $\chi(M/\g)$ (in the sense of singular homology, see \cite[Theorem 3]{satake}). An interesting application is the following.

\begin{theorem}[{\cite[Theorem 9.1]{caramello}}]\label{theorem: closed leaf + transverse symmetry implies charM vanishes}
Let $\f$ be a Killing foliation of a compact manifold $M$. If $\chi(M)\neq 0$ then $\f$ is closed.
\end{theorem}

\begin{proofoutline}
We use the results in \cite{haefliger3}, where Haefliger studies the classifying space $B\mathscr{H}_\f$. It follows from his work that, similarly to the case of fiber bundles, $\chi(M)$ has the product property
\begin{equation}\label{eq: product property 1}\chi(M)=\chi(L)\chi(\g),\end{equation}
for $L\in\g$ a generic leaf (see \cite[Corollaire 3.1.5]{haefliger3}). Now assume $\f$ is non-closed and, by Theorem \ref{theorem: Mozgawa}, fix a closed leaf $L$. By regular deformations, choose a sequence $\g_i$ of closed foliations approaching $\f$. Then $L\in\g_i$ for each $i$, since the deformations preserve $\overline{\f}$. In particular, using Theorem \ref{teorem: reeb} we can rewrite equation \eqref{eq: product property 1} as
\begin{equation}\chi(M)=h_i(L)\chi(L)\chi(\f),\label{eq: fibration property basic euler char 2}\end{equation}
where $h_i(L)=|\mathrm{Hol}_{\g_i}(L)|<\infty$. But since $\f$ is non-closed and $\g_i\to\f$, one verifies that $h(\g_i)\to\infty$ (see \cite[Lemma 4.3]{caramello}) which violates equation \eqref{eq: fibration property basic euler char 2}.
\end{proofoutline}

Theorem \ref{theorem: closed leaf + transverse symmetry implies charM vanishes} is in fact a slight improvement of \cite[Theorem 9.1]{caramello} (which is stated for a simply-connected $M$), but the proof is essentially the same. By lifting to the universal covering, one has the following corollary for Riemannian foliations:

\begin{corollary}[{\cite[Theorem F]{caramello}}]\label{teo: topological obstruction}
Any Riemannian foliation of a compact manifold $M$ with $|\pi_1(M)|<\infty$ and $\chi(M)\neq0$ is closed.
\end{corollary}

Although $\chi(\f)$ is preserved throughout regular deformations, the basic Betti numbers are not, in general, as the following example by H.~Nozawa shows.

\begin{example}\label{exe: nozawa}
Consider $M=\mathbb{S}^3\times\mathbb{S}^1$ with the $\mathbb{T}^2=\mathbb{S}^1\times\mathbb{S}^1$-action given by
$$((s_1,s_2),((z_1,z_2),z))\longmapsto ((s_1z_1,s_1z_2), s_2z),$$
and let $\f$ be the Killing foliation of $M$ by the orbits of a dense $1$-parameter subgroup of $\mathbb{T}^2$. As we saw in Example \ref{example: Hae-Sal transverse structure of gen Hopf fibration}, the construction of the corresponding orbifold $(\mathcal{O}_\f,\mathbb{T}^N, H)$ is trivial: $\mathcal{O}_\f=M$ and $H$ is the $1$-parameter subgroup defining $\f$. It is clear that $\f$ can be deformed to both the foliations $\g_1$ and $\g_2$, defined by the actions of $\mathbb{S}^1\times \{1\}$ and $\{1\}\times \mathbb{S}^1$, respectively. But we have $H(\g_1)=H(M/\!/\g^1)=H(\mathbb{S}^2\times \mathbb{S}^1)$ and $H(\g_2)=H(M/\!/\g_2)=H(\mathbb{S}^3)$. That is, $b_i(\g_1)\neq b_i(\g_2)$ for $i=1,2$.
\end{example}

We conclude from Example \ref{exe: nozawa} that the basic cohomology groups $H(\f)$ are not preserved by deformations. In the next section we will see, however, that there is a cohomological invariant, namely, basic \emph{equivariant} cohomology, that is preserved. This will, in particular, provide sufficient conditions for the basic Betti numbers to be preserved as well.

\subsection{Equivariant basic cohomology}\label{section: equivariant basic cohomology}

When a group $G$ acts on a space $M$, there is a cohomology theory that captures information on both the topological space $M$ and the action of $G$ on it. It is called equivariant cohomology, and defined as the singular cohomology of the Borel construction:
$$H_G(M,R):=H\left(\frac{EG\times M}{G},R\right),$$
where $EG$ is a contractible space on which $G$ acts freely (e.g., the total space of the universal $G$-bundle $EG\to BG$). The motivation for this is that the diagonal action of $G$ on $EG\times M$ is free, so the quotient is a well-behaved space (in contrast to $M/G$). A remarkable feature of equivariant cohomology, with no counterpart in classic cohomology, is that the non-torsion part of the module structure of $H_{\mathbb{T}}(M)$, for a torus space $M$, can be recovered from the fixed point set $M^{\mathbb{T}}$. This is known as Borel localization. We will see a transverse counterpart of this result below. We refer to \cite{goertsches3} and \cite{meinrenken} for more detailed introductions to the classical equivariant cohomology theory, and to \cite{allday} and \cite{guillemin} for thorough treatments of this topic.

It turns out that when $G$ is compact and connected and $M$ is a $G$-manifold, there is another way to compute $H_G(M,\mathbb{R})$. It is due to H.~Cartan (actually before Borel's definition of $H_G(M,R)$), who defined a cohomology $H_{\mathfrak{g}}(M)$ in terms of the de Rham complex $\Omega(M)$ and the Lie algebra $\mathfrak{g}$. The fact that $H_G(M,\mathbb{R})\cong H_{\mathfrak{g}}(M)$ is considered as the equivariant analog of the classical de Rham theorem (see, e.g., \cite[Theorem 2.5.1]{guillemin}). We are interested in Cartan's model for equivariant cohomology because its algebraic nature makes it readily generalizable to our transverse setting.

Recall that a differential $\mathfrak{g}^\star$-algebra is a $\mathbb{Z}$-graded-commutative differential algebra $(A,d)$ endowed, for each $X\in\mathfrak{g}$, with derivations $\mathcal{L}_X$ and $\iota_X$, of degree $0$ and $-1$, respectively,  satisfying
$$\iota_X^2=0,\ \ \ \ [\mathcal{L}_X,\mathcal{L}_Y]=\mathcal{L}_{[X,Y]},\ \ \ \ [\mathcal{L}_X,\iota_Y]=\iota_{[X,Y]}\ \ \ \ \mbox{and}\ \ \ \ \mathcal{L}_X=d \iota_X+\iota_X d.$$
If $A$ and $B$ are $\mathfrak{g}^\star$-algebras, an algebra morphism $f:A\to B$ is a morphism of $\mathfrak{g}^\star$-algebras if it commutes with $d$, $\mathcal{L}_X$ and $\iota_X$.

\begin{example}\label{exe: induced infinitesimal action}
An infinitesimal action of a Lie algebra $\mathfrak{g}$ on an orbifold $\mathcal{O}$ is a Lie algebra homomorphism $\mu:\mathfrak{g}\to \mathfrak{X}(\mathcal{O})$. A differential $\mathfrak{g}^\star$-algebra structure on $(\Omega(\mathcal{O}),d)$ is then given by the usual Lie derivative $\mathcal{L}_X=\mathcal{L}_{\mu(X)}$ and interior product $\iota_X:=\iota_{\mu(X)}$.

In particular, if a Lie group $G$ acts smoothly on $\mathcal{O}$ (on the left), there is an induced infinitesimal action of its Lie algebra given by $\mathfrak{g}\ni X\mapsto -X^\#\in\mathfrak{X}(\mathcal{O})$. Recall that $X\mapsto X^\#$ is a Lie algebra anti-homomorphism, that is why the minus sign is needed.
\end{example}

Consider also the coadjoint action of a Lie algebra $\mathfrak{g}$ on its dual algebra $\mathfrak{g}^\vee$ given, for $X,Y\in\mathfrak{g}$ and $\phi\in\mathfrak{g}^\vee$, by $(\mathrm{ad}^\vee_{X}\phi)(Y)=\phi(-[X,Y])$. It extends naturally to the symmetric algebra $\sym(\mathfrak{g}^\vee)$ over $\mathfrak{g}^\vee$. The space
$$C_{\mathfrak{g}}(A):=(\sym(\mathfrak{g}^\vee)\otimes A)^{\mathfrak{g}}$$
of those elements on $\sym(\mathfrak{g}^\vee)\otimes A$ which are $\mathfrak{g}$-invariant, with respect to the coadjoint action and the derivation $\mathcal{L}$ on the first and second factors, respectively, is the \textit{Cartan complex of $A$}. Notice that an element $\omega\in C_{\mathfrak{g}}(A)$ can be identified with a polynomial map $\omega:\mathfrak{g}\to A$. Under this identification, $\mathfrak{g}$-invariancy of $\omega$ as an element of $C_{\mathfrak{g}}(A)$ becomes $\mathfrak{g}$-equivariancy of $\omega$ as a polynomial map $\mathfrak{g}\to A$:
$$\omega(\mathrm{ad}_XY)=\mathcal{L}_X\omega(Y).$$
Notice that in the case of an Abelian Lie algebra $\mathfrak{g}$, for which the coadjoint action is trivial, an element of $C_{\mathfrak{g}}(A)$ is hence nothing but a polynomial map $\mathfrak{g}\to A^{\mathfrak{g}}$.

The \textit{equivariant differential} $d_\mathfrak{g}$ of the Cartan complex is defined as
$$(d_\mathfrak{g}\omega)(X)=d(\omega(X))-\iota_X(\omega(X)).$$
In order for it to be a derivation of degree $1$, the grading on $C_{\mathfrak{g}}(A)$ is defined by
$$C_{\mathfrak{g}}^n(A)=\bigoplus_{2k+l=n}(\sym_k(\mathfrak{g}^\vee)\otimes A^l)^{\mathfrak{g}}.$$
The \textit{Cartan model for the equivariant cohomology of $A$} is
$$H_{\mathfrak{g}}(A):=H(C_{\mathfrak{g}}(A),d_\mathfrak{g}).$$
A morphism $f:A\to B$ of $\mathfrak{g}^\star$-algebras induces $f^*:H_{\mathfrak{g}}(A)\to H_{\mathfrak{g}}(B)$, by $f^*\omega(X)=f^*(\omega(X))$. The ring $H_{\mathfrak{g}}(A)$ becomes a $\sym(\mathfrak{g}^\vee)^{\mathfrak{g}}$-algebra with module multiplication induced by $\sym(\mathfrak{g}^\vee)^{\mathfrak{g}}\ni f\mapsto f\otimes 1\in C_{\mathfrak{g}}(A)$.

\begin{example}
In the case of a $G$-orbifold $\mathcal{O}$, the $\sym(\mathfrak{g}^\vee)^{\mathfrak{g}}$-module structure $\sym(\mathfrak{g}^\vee)^{\mathfrak{g}}\to C_{\mathfrak{g}}(\mathcal{O})$ coincides with the cohomology map induced by the constant map $\mathcal{O}\to\{*\}$.
\end{example}

A $\mathfrak{g}^\star$-algebra $A$ is said to be \textit{equivariantly formal} if $H_\mathfrak{g}(A)\cong\sym(\mathfrak{g}^\vee)^\mathfrak{g}\otimes H(A)$ as $\sym(\mathfrak{g}^\vee)^\mathfrak{g}$-modules. Equivalently, $A$ is equivariantly formal when $H_\mathfrak{g}(A)$ is a free $\sym(\mathfrak{g}^\vee)$-module. There are several relevant classes of equivariantly formal algebras. For instance, for a manifold $M$ with a torus action, and $A=\Omega(M)$ with the induced $\mathfrak{t}^\star$-structure, $A$ is equivariantly formal when $H^{\mathrm{odd}}(M)=0$, or when $M$ is symplectic and the torus action is Hamiltonian.

We now go back to the case of a foliation $\f$ on an orbifold $\mathcal{O}$. A \textit{transverse infinitesimal action} of a Lie algebra $\mathfrak{g}$ on $\f$ is a Lie algebra homomorphism
$$\mu:\mathfrak{g}\longrightarrow \mathfrak{l}(\f).$$
It induces a $\mathfrak{g}^\star$-algebra structure on $\Omega(\f)$, with $d$ being the usual exterior derivative and the derivations $\mathcal{L}_X$ and $\iota_X$ defined as $\mathcal{L}_X\omega:=\mathcal{L}_{\widetilde{X}}\omega$ and $\iota_X\omega:=\iota_{\widetilde{X}}\omega$ (see \cite[Proposition 3.12]{goertsches}). We can therefore define the \textit{$\mathfrak{g}$-equivariant basic cohomology of $\f$} as the $\mathfrak{g}$-equivariant cohomology of $\Omega(\f)$, which we will denote
$$H_\mathfrak{g}(\f):=H_\mathfrak{g}(\Omega(\f))=H(C_\mathfrak{g}(\Omega(\f),d_\mathfrak{g})).$$

Now consider a Killing foliation $\f$ on $M$. In this case we have a natural transverse infinitesimal action of its structural algebra $\mathfrak{a}$, given by the isomorphism $\mathfrak{a}\cong\mathscr{C}_\f(M)$. Notice that the fixed point set $M^\mathfrak{a}=\{x\in M\ |\ \mathfrak{a}_x=\mathfrak{a}\}$ is precisely the union of the closed leaves of $\f$, since $\mathfrak{a}\f=\overline{\f}$. These two facts, that the infinitesimal $\mathfrak{a}$-action is canonical and that $\mathfrak{a}\f=\overline{\f}$, makes the study of $H_\mathfrak{a}(\f)$ very relevant.

The equivariant basic cohomology $H_\mathfrak{a}(\f)$ was first introduced in \cite{goertsches}, where the authors show that, in analogy to classical equivariant cohomology, it satisfies a Borel-type localization. Before we state their result it will be useful to recall the notion of $R$-module localization from commutative algebra. Given an $R$-module $A$ and a multiplicative subset $S\subset R$ we define the \textit{localization of $A$ at $S$} by $S^{-1}A=(A\times S)/\sim$, where $(a,s)\sim (a',s')$ if there is $r\in S$ such that $r(s'a-sa')=0$. We can think of an equivalence class $(a,s)$ as fraction $a/s$. Notice $S^{-1}A$ is an $S^{-1}R$-module with the usual operation rules for fractions. Additionally, map of $R$-modules $\varphi:A\to B$ induces a map of $S^{-1}R$-modules $S^{-1}\varphi:S^{-1}A\to S^{-1}B$ by $a/s\mapsto \varphi(a)/s$.

\begin{theorem}[Borel localization {\cite[Theorem 5.2]{goertsches}}]\label{theorem: Borel loc}
Let $\f$ be a transversely compact Killing foliation. Then the inclusion $i:M^{\mathfrak{a}}\to M$ induces an isomorphism
$$S^{-1}i^*:S^{-1}H_\mathfrak{a}(\f)\longrightarrow S^{-1}H_\mathfrak{a}(\f|_{M^{\mathfrak{a}}}),$$
where $S=\sym(\mathfrak{a}^\vee)\setminus {0}$.
\end{theorem}

This result was recently generalized to transverse actions of Abelian Lie algebras on transversely compact Riemannian foliations in \cite{lin}. It follows from Theorem \ref{theorem: Borel loc} that the kernel of $i^*:H_\mathfrak{a}(\f)\to H_\mathfrak{a}(\f|_{M^{\mathfrak{a}}})$ is the torsion submodule $\mathrm{Tor}(H_\mathfrak{a}(\f))$ of $H_\mathfrak{a}(\f)$, that is, the submodule consisting of those classes $[\omega]$ for which there is $p\in S$ with $p[\omega]= 0$. Since $M^{\mathfrak{a}}$ is the union of the closed leaves of $\f$, this gives algebraic conditions for the existence of closed leaves:

\begin{corollary}[{\cite[Corollary 5.4]{goertsches}}]
Let $\f$ be a transversely compact Killing foliation. The following are equivalent:
\begin{enumerate}[(i)]
\item $\f$ has a closed leaf, i.e., $M^{\mathfrak{a}}\neq\emptyset$.
\item The map $\sym(\mathfrak{a}^\vee)\to H_\mathfrak{a}(\f)$ that defines the $\sym(\mathfrak{a}^\vee)$-module structure is injective.
\item $H_\mathfrak{a}(\f)\neq \mathrm{Tor}(H_\mathfrak{a}(\f))$.
\end{enumerate}
\end{corollary}

For the next result, we recall that the transverse action of $\mathfrak{a}$ on $\f$ is equivariantly formal when $\Omega(\f)$ is an equivariantly formal $\mathfrak{a}^\star$-algebra. In this case we also say that $\f$ is equivariantly formal. A transversely orientable Killing foliation $\f$ is equivariantly formal, for example, when some of the following conditions hold (see \cite{goertsches}):
\begin{enumerate}[(i)]
\item $H^{\mathrm{odd}}(\f)=0$.
\item $\dim H(M^{\mathfrak{a}}/\!/\f)=\dim H(\f)$.
\item $\f$ admits a basic Morse-Bott function whose critical set is equal to $M^{\mathfrak{a}}$,
\end{enumerate}

The dimension $\dim(H(\f))$ of basic cohomology can be studied via equivariant cohomology, providing another consequence of Theorem \ref{theorem: Borel loc}:

\begin{theorem}[{\cite[Theorem 5.5]{goertsches}}]
Let $\f$ be a transversely compact Killing foliation. Then
$$\dim(H(M^\mathfrak{a}/\!/\f))=\dim(H_(\f|_{M^\mathfrak{a}}))\leq \dim(H(\f)),$$
and equality holds if, and only if, the $\mathfrak{a}$-action is equivariantly formal.
\end{theorem}

The behavior of equivariant basic cohomology under regular deformations was studied in \cite{caramello2}. Recall the construction of $\f_t$ as a pullback $\Upsilon^*(\f_{H(t)})$ from section \ref{section: deformations}. Notice there is a transverse action of $\mathfrak{t}/\mathfrak{h}(t)$ on $\f_{H(t)}$ for each $t$. All those Lie algebras are isomorphic to $\mathfrak{a}$, although in a non-canonical way. We define an $\mathfrak{a}$-action on $\f_{H(t)}$ (and thus on $\f_t$, since $\mathscr{H}(\f_t)\cong\mathscr{H}(\f_{H(t)})$), by passing through an isomorphism $\mathfrak{t}/\mathfrak{h}(t)\to\mathfrak{a}$, which amounts to identifying $\mathfrak{a}$ with a subalgebra of $\mathfrak{t}$ complementary to each $\mathfrak{h}(t)$, that by abuse we will also denote by $\mathfrak{a}<\mathfrak{t}$.

\begin{proposition}[{\cite[Proposition 5.2]{caramello2}}]\label{prop: induced transverse a action on G}
The structural algebra $\mathfrak{a}$ of $\f$ acts transversely on each $\f_t$ and its induced action on the quotient orbifold $M/\!/\g$ (for the closed foliation $\g=\f_1$) integrates to the $\mathbb{T}^d$-action given by item \eqref{toric action on the quotient} of Theorem \ref{theorem: deformation}.
\end{proposition}

It is now possible, therefore, to consider the $\mathfrak{a}$-equivariant basic cohomology of $\f_t$, that is, $H_\mathfrak{a}(\f_t)$. Of course, one will be specially interested in $H_\mathfrak{a}(\g)$, for which one has
$$H_\mathfrak{a}(\g)\cong H_\mathfrak{a}(M/\!/\g)\cong H_{\mathbb{T}^d}(M/\!/\g),$$
by the equivariant de Rham theorem for orbifolds \cite[Theorem 3.5]{caramello2}.

\begin{theorem}[{\cite[Theorem A]{caramello2}}]\label{thrm: invariance of H_a under deformations}
Let $\f_t$ be a regular deformation of a Killing foliation $\f$. For each $t$ there is an $\mathbb{R}$-algebra isomorphism
$$H_{\mathfrak{a}}(\f)\cong H_{\mathfrak{a}}(\f_t).$$
\end{theorem}

In particular, for $t=1$ we have $H_{\mathfrak{a}}(\f)\cong H_{\mathfrak{a}}(\g)\cong H_{\mathbb{T}^d}(M/\!/\g)$, as rings, for a closed foliation $\g$ arbitrarily close to $\f$, thus reducing the study of $H_{\mathfrak{a}}(\f)$ to equivariant cohomology of torus actions on orbifolds. Moreover, the authors show in \cite[Proposition 6.2]{caramello2} that equivariant formality is preserved by regular deformations, that is, if $\f$ is equivariantly formal, then each $\f_t$ is equivariantly formal with respect to the transverse $\mathfrak{a}$-action on given in Proposition \ref{prop: induced transverse a action on G}. Hence, in this case
\begin{equation}\label{betti numbers preserved formal}\sym(\mathfrak{a}^\vee)\otimes H(\f)\cong H_{\mathfrak{a}}(\f)\cong H_{\mathfrak{a}}(\f_t)\cong\sym(\mathfrak{a}^\vee)\otimes H(\f_t).\end{equation}
Recall that the Poincaré series of an $\mathbb{N}$-graded vector space $V$ is the formal power series
$\poin_{V}(s)=\sum_{k=0}^\infty (\dim V^k)s^k$ (provided $\dim V^k$ finite for each $k$), which has the following product property: $\poin_{V\otimes W}(s)=\poin_{V}(s)\poin_{W}(s)$. Passing to the Poincaré series in equation \eqref{betti numbers preserved formal} and canceling out $\poin_{\sym(\mathfrak{a}^\vee)}(s)$ on both sides yields $\poin_{H(\f)}(s)=\poin_{H(\f_t)}(s)$. This proves the following:

\begin{theorem}[{\cite[Theorem B]{caramello2}}]\label{theo: Betti numbers are invariant}
If $\f$ is equivariantly formal and $\f_t$ is a regular deformation, then $b_i(\f_t)$ is constant on $t$, for each $i$.
\end{theorem}

It is therefore possible to reduce, at least in the equivariantly formal case, results concerning basic Betti numbers to results about Betti numbers of orbifolds, since Theorem \ref{theo: Betti numbers are invariant} gives $b_i(\f)=b_i(M/\!/\g)$ when $t=1$. An application appears in Theorem \ref{theorem: Gromov intro}.

\section{Transverse geometry of Killing foliations}\label{section: transverse geometry of Killing}

Many techniques from classical Riemannian geometry can be used in the study of the transverse geometry of Riemannian foliations, as our brief survey on Section \ref{section: riemannian foliations} already illustrates, and many classical theorems admit a transverse generalization. We also cite here the following result by G.~Oshikiri:

\begin{theorem}[Oshikiri {\cite[Theorem 2]{oshikiri}}]\label{oshikiri closed leaf}
Let $\f$ be a Riemannian foliation on a compact manifold $M$ with $\sec_M>0$.
\begin{enumerate}[(i)]
\item If $\codim(\f)$ is even then $\f$ admits a closed leaf.\label{item oshikiri codim even}
\item If $\codim(\f)$ is odd then there is $L\in\f$ with $\codim(\overline{L})=\codim(\f)-1$.
\end{enumerate}
\end{theorem}

This is obtained by studying zeros of transverse Killing fields via classic techniques. The existance of a closed leaf in item \eqref{item oshikiri codim even} corresponds to the existence of a zero for a transverse Killing vector field, and thus is a transverse analog of classical Berger's theorem on zeros of Killing vector fields (see, e.g., \cite[Theorem 38]{petersen}). Notice that if $\sec_M>0$, with respect to a bundle-like metric for $\f$, then $\sec_\f>0$, since by O'Neil's formula \cite{oneil} applied to a Riemannian submersion locally defining $\f$ one has
$$\sec_\f(\overline{X},\overline{Y})=\sec_M(X,Y)+\frac{3}{4}\|[X,Y]\|^2,$$
for $X,Y\in\mathfrak{L}(\f)$. Also in positive transverse curvature, Hebda's Theorem \ref{theorem: hebda} is obtained essentially by the study of focal points of leaves over horizontal geodesics. In the case of non-positive transverse curvature, Hebda proves that leaves have no focal points, which then leads to the following.

\begin{theorem}[Hebda {\cite[Theorem 2]{hebda}}]\label{Theorem: Hebda negative sectional}
Let $\f$ be a complete Riemannian foliation of $M$ with $\sec_\f\leq 0$. Then the universal covering of $M$ is a product $\widetilde{M}=\widetilde{L}\times N$, for $L\in\f$ and $N$ a Hadamard manifold, and the lifted foliation $\widetilde{\f}$ is given by the fibers of the canonical projection $\widetilde{L}\times N\to N$.
\end{theorem}

\subsection{Transverse geometry via deformations}

An inherent difficulty often encountered in these aforementioned transverse generalizations of classical theorems from Riemannian geometry is that the leaf space of a Riemannian foliation has, in general, an ill-behaved topology which in many cases renders direct generalizations of ``local-to-global'' theorems impossible. For Killing foliations this difficulty can in some cases be circumvented by the deformation technique we presented in Section \ref{section: deformations}, since some aspects of transverse geometry are preserved by regular deformations. In this section we will see several applications of this approach, that appeared in \cite{caramello} and \cite{caramello2}. For instance,  by combining the deformation method with the Synge--Weinstein theorem for orbifolds \cite[Theorem 2.3.5]{yeroshkin} one can relax the hypothesis on Theorem \ref{oshikiri closed leaf}:

\begin{theorem}[{\cite[Theorem C]{caramello}}]\label{theorem: Berger for foliations}
Let $(\f,\mathrm{g}^T)$ be an even-codimensional complete Riemannian foliation of a manifold $M$ satisfying $|\pi_1(M)|<\infty$. If $\sec_{\f}\geq c>0$, then $\f$ possesses a closed leaf.
\end{theorem}

There is also an application involving Bochner's theorem on Killing vector fields in the context of negative Ricci curvature \cite[Theorem 36]{petersen}. This result adapts directly to orbifolds \cite[Theorem 2.5]{caramello2} and, via deformations, implies the nonexistence of transverse Killing fields for a Ricci negatively curved Killing foliation, which is therefore closed (cf. Theorem \ref{Theorem: Hebda negative sectional}).

\begin{theorem}[{\cite[Theorem F]{caramello2}}]\label{theorem: bochner intro}
Let $(M,\f)$ be a complete Riemannian foliation with transverse Ricci curvature satisfying $\ric_\f\leq c <0$. If either
\begin{enumerate}[(i)]
\item $\f$ is a Killing foliation and $M$ is compact, or
\item $\f$ is transversely compact and $|\pi_1(M)|<\infty$,
\end{enumerate}
then $\f$ is closed.
\end{theorem}

For the next result, recall the notion of $\pi_1(\f)$ from Section \ref{section: holonomy}. Recall also that the growth function $\#$ of a finitely generated group $\Gamma=\langle g_1,\dots,g_k \rangle$ is the function that associates to $j\in\mathbb{N}$ the number of distinct elements in $\Gamma$ which can be written as words with at most $j$ letters in the alphabet $\{g_1,\dots,g_k,g_1^{-1},\dots,g_k^{-1}\}$. Then $\Gamma$ is said to have \textit{exponential growth} if $\#(j)\geq \alpha^j$ for some $\alpha>1$ (this property is independent of the set of generators \cite[Lemma 1]{milnor}). Milnor's theorem establishes that the fundamental group of a negatively curved compact manifold has exponential growth \cite[Theorem 2]{milnor}. Milnor's proof of this result adapts to orbifolds \cite[Theorem 2.6]{caramello2}, and then the deformation method can be used to show the following.

\begin{theorem}[{\cite[Theorem G]{caramello3}}]\label{theo: Milnor trasnverso intro}
Let $\f$ be a Killing foliation on a compact manifold $M$ such that $\sec_\f<0$. Then $\f$ is closed and $\pi_1(\f)$ grows exponentially. In particular, $\pi_1(M)$ grows exponentially.
\end{theorem}

One should compare Theorem \ref{theo: Milnor trasnverso intro} with \cite[Theorem 3]{hebda}, which implies that a compact manifold whose fundamental group is nilpotent does not admit a Riemannian foliation with $\sec_\f<0$, recalling Gromov's theorem that states that a finitely generated group has polynomial growth if and only if it has a nilpotent subgroup with finite index \cite[Main Theorem]{gromov3}. Another result by Gromov establishes an upper bound for the sum of Betti numbers of negatively curved manifolds, in terms of their dimension and volume \cite[p. 12]{gromov3}. An analogous bound holds for orbifolds, as shown by I.~Samet in \cite[Theorem 1.1]{samet}. Combining it with with the fact that a negatively curved Killing foliation is closed, by Theorem \ref{theo: Milnor trasnverso intro}, we get:

\begin{corollary}
There exists a constant $C=C(q)$ such that, for any Killing foliation $\f$ on a compact manifold $M$ with $\sec_\f<0$, say $-k^2\leq \sec_\f <0$, one has
$$\sum_{i=1}^q b_i(\f)\leq Ck^q\vol(M/\!/\f).$$
\end{corollary}

The classical Singe's theorem also has an orbifold version, proved by D.~Yeroshkin in \cite[Corollary 2.3.6]{yeroshkin}. By the deformation technique, it yields the following transverse generalization:

\begin{theorem}[{\cite[Theorem H]{caramello2}}]\label{theorem: synge intro}
Let $\f$ be a Killing foliation of a compact manifold $M$, with $\sec_\f >0$. Then
\begin{enumerate}[(i)]
\item if $\codim\f$ is even and $\f$ is transversely orientable, then $M/\overline{\f}$ is simply connected, and
\item if $\codim\f$ is odd and, for each $L\in \f$, the germinal holonomy of $L$ preserves transverse orientation, then $\f$ is transversely orientable.
\end{enumerate}
\end{theorem}

Recall that the symmetry rank $\symrank(M)$ of a Riemannian manifold $M$ is the rank of its isometry group, that is, the dimension of a maximal torus in $\mathrm{Iso}(M)$. It was proven by K.~Grove and C.~Searle in \cite{grove} that, for a positively curved compact Riemannian manifold $M$, one has
$$\symrank(M)\leq\left\lfloor\frac{\dim(M)+1}{2}\right\rfloor,$$
with equality holding if and only if $M$ is diffeomorphic to either a sphere, a real or complex projective space or a lens space. A generalization of this result for orbifolds was obtained recently in \cite[Corollary E]{harvey}. Now consider a Killing foliation $\f$ with structural algebra $\mathfrak{a}$. By what we saw in Section \ref{section: killing foliations}, we have
$$\dim(\overline{\f})-\dim(\f)=\dim(\mathfrak{a})\leq\symrank(\f):=\max_{\mathfrak{h}}\Big\{\dim(\mathfrak{h})\Big\},$$
where $\mathfrak{h}$ runs over all the Abelian subalgebras of $\mathfrak{iso}(\f)$. Combining  the deformation technique with \cite[Corollary E]{harvey} one then obtains the following.

\begin{theorem}[{\cite[Theorem A]{caramello}}]\label{theorem: harvey-searle para folheações}
Let $\f$ be a $q$-codimensional, transversely orientable Killing foliation of a compact manifold $M$. If $\sec_{\f}>0$, then
$$\dim(\overline{\f})-\dim(\f)\leq\left\lfloor\frac{\codim(\f)+1}{2}\right\rfloor$$
and if equality holds, there is a closed Riemannian foliation $\g$ of $M$ arbitrarily close to $\f$ with $M/\g$ homeomorphic to either
\begin{enumerate}[(i)]
\item $\mathbb{S}^q/\Lambda$, where $\Lambda$ is a finite subgroup of the centralizer of the maximal torus in $\mathrm{O}(q+1)$, or
\item $|\mathbb{CP}^{q/2}[\lambda]|/\Lambda$, where $\Lambda$ is a finite subgroup of the torus acting linearly on $\mathbb{CP}^{q/2}[\lambda]$.
\end{enumerate}
\end{theorem}

The symmetry rank $\symrank(M)$ also plays an important role in partial solutions to Hopf's conjecture that every even-dimensinal positively curved Riemannian manifold has positive Euler characteristic. It was proved by T.~Püttmann and C.~Searle in \cite[Theorem 2]{puttmann}, for instance, that Hopf's conjecture holds for manifolds satisfying $\symrank(M)\geq \dim(M)/4-1$. This linear bound was subsequently weakened by X.~Rong and X.~Su in \cite[Theorem A]{rong}, and further improved by L.~Kennard to the logarithmic bound $\symrank(M)\geq \log_2(n-2)$ in the case $\dim(M)=0\mod 4$, \cite[Theorem A]{kennard}. In the transverse setting, Theorem \ref{theorem: basic euler char is preserved by deformations} guarantees that one can study the basic Euler characteristic by the deformation method. A generalization of the Püttmann--Searle theorem for orbifolds was proven in \cite[Theorem 8.9]{caramello}, from which a transverse version for Killing foliations follows by deformation:

\begin{theorem}[{\cite[Theorem E]{caramello}}]\label{theorem: puttmann for foliations}
Let $\f$ be a $q$-codimensional transversely orientable Killing foliation of a compact manifold $M$. If $q$ is even, $\sec_\f>0$ and $\symrank(\f)\geq q/4-1$, then $\chi(\f)>0$.
\end{theorem}

Finally, Theorem \ref{theo: Betti numbers are invariant} shows that the basic Betti numbers of Killing foliations can also be studied via deformations, provided the transverse action of the structural algebra is equivariantly formal. For instance, a theorem by Gromov establishes the existence of a constant $C=C(n)$ that bounds the total sum of Betti numbers of any positively curved Riemannian manifold of dimension $n$ \cite[\S 0.2A]{gromov2}. An analogous result holds for orbifolds, as it follows by \cite[Theorem 1]{koh}. Thus, by deformations, one obtains the following transverse generalization:

\begin{theorem}[{\cite[Theorem E]{caramello2}}]\label{theorem: Gromov intro}
There exists a constant $C=C(q)$ such that every $q$-codimensional Killing foliation $\f$ of a compact manifold $M$ with $\sec_\f>0$ and whose transverse action of the structural algebra $\mathfrak{a}$ is equivariantly formal satisfies
$$\sum_{i=0}^q b_i(\f) \leq C.$$
\end{theorem}

\section{Singular Riemannian foliations}\label{section: singular riemannian foliations}

In this section we will briefly present singular Riemannian foliations and survey some classical and recent results about them. The notion of singular foliation generalizes that of regular foliations by allowing the dimensions of the leaves to vary. More precisely, given an $n$-dimensional connected manifold $M$, a \textit{singular foliation} of $M$ is a partition $\f$ of $M$ into connected, immersed submanifolds, called \textit{leaves}, such that the module $\mathfrak{X}(\f)$ of smooth vector fields that are tangent to the leaves is transitive on each leaf. This means, as in the regular case, that for each $L\in\f$ and each $x\in L$ one can find smooth vector fields $X_i$ whose values at $x$ form a basis for $T_xL$. We maintain most of the notation from regular foliations, e.g. we denote the distribution of varying rank defined by the tangent spaces of the leaves by $T\f$ and the leaf containing $x$ by $L_x$. The algebra of foliate vector fields can also be defined similarly, as $\mathfrak{L}(\f)=\{X\in\mathfrak{X}(M)\ |\ [X,\mathfrak{X}(\f)]\subset\mathfrak{X}(\f)\}$, and consists of those fields whose flows take leaves to leaves. The \textit{transverse vector fields} are the elements of $\mathfrak{l}(\f):=\mathfrak{L}(\f)/\mathfrak{X}(\f)$. The \textit{dimension of $\f$} is defined as
$$\dim(\f)=\max_{L\in\f}\dim(L).$$

\begin{example}[Homogeneous singular foliations]\label{ex: homogeneous singular}
Consider a manifold $M$ with an action of a Lie group $H$. Then we have an induced infinitesimal action $\mu$ of the Lie algebra $\mathfrak{h}$ of $H$ (see Example \ref{exe: induced infinitesimal action}). One easily verifies that $T_xHx=\mu(\mathfrak{h})|_x$, that is, the space generated by the fundamental vector fields of the $H$-action at $x$ is the tangent space of the orbit $Hx$ at $x$. This shows that the partition $\f_H$ of $M$ into the connected components of the orbits of $H$ is a singular foliation. In analogy with the regular case, such a foliation is an \textit{homogeneous singular foliation}. One also verifies that $\f_H=\f_{H^e}$, where $H^e<H$ is the connected component of the identity, so supposing that $H$ is connected usually does not affect the study of $\f_H$ and has the advantage that in this case the leaves (which are connected by definition) coincide with the orbits.
\end{example}

Singular Riemannian foliations are defined by generalizing Reinhart's characterization of bundle-like metrics (Proposition \ref{prop: Reinhart characterization}): if $M$ can be endowed with a Riemannian metric $\mathrm{g}$ such that every geodesic which is perpendicular to a leaf of $\f$ remains perpendicular to all leaves it intersects, then we say that $\f$ is a \textit{singular Riemannian foliation} and that $\mathrm{g}$ is \textit{adapted} to $\f$. Any partition of $M$ into submanifolds (not necessarily a smooth singular foliation) having this property is called a \textit{transnormal system} on $(M,\mathrm{g})$, following the terminology of \cite{bolton}. We can say, hence, that a singular Riemannian foliation $\f$ of $M$ is a singular foliation of $M$ which is also a transnormal system with respect to some Riemannian metric.

For a leaf $L\in\f$ we denote the normal space at $x\in L$ by $\nu_xL=(T_xL)^\perp$. It is clear from Proposition \ref{prop: Reinhart characterization} that every regular Riemannian foliation $\f$ is a singular Riemannian foliation. \textit{Homogeneous singular Riemannian foliations} form another very significant class:

\begin{proposition}[{\cite[Section 6.1]{molino}}]\label{prop: homogeneous srf}
Let $(M,\mathrm{g})$ be a Riemannian manifold on which a Lie group $H$ acts by isometries. Then $\mathrm{g}$ is an adapted metric for $\f_H$, which is thus a singular Riemannian foliation.
\end{proposition}

It follows from Molino's structural theorem (Theorem \ref{Theo: molino structural}) that the closure $\overline{\f}$ of a complete regular Riemannian foliation is a singular foliation. One has, in fact, the following.

\begin{proposition}[{\cite[Proposition 6.2]{molino}}]
Let $(M,\f)$ be a complete (regular) Riemannian foliation and $\mathrm{g}$ be a bundle-like metric. Then $\overline{\f}$ is a singular Riemannian foliation to which $\mathrm{g}$ is adapted.
\end{proposition}

In Section \ref{section: molino conjecture and its proof} we will see that a similar result holds for a complete singular Riemannian foliation $\f$: the partition $\overline{\f}$ of $M$ into the closures of leaves of $\f$ is again a singular Riemannian foliation.

One defines basic cohomology in complete analogy with the regular case: for a singular Riemannian foliation $\f$, a differential form $\omega\in\Omega^i(M)$ is \textit{basic} if $\iota_X\omega=0$ and $\mathcal{L}_X\omega=0$ for all $X\in\mathfrak{X}(\f)$. The $d$-subcomplex of $\f$-basic forms will be denoted by $\Omega(\f)$. It is a $\mathbb{Z}$-graded differential algebra with respect to the usual exterior derivative and wedge product. The basic cohomology of $\f$ is the cohomology $H(\f)$ of $(\Omega(\f),d)$.

\begin{theorem}[{\cite[Theorem 1]{wolak}}]
If $\f$ is a singular Riemannian foliation of a compact manifold $M$, then $\dim H(\f)<\infty$.
\end{theorem}

\subsection{Slice foliation, homothetic lemma and canonical stratification}\label{section: homothetic lemma}

In this section we review some basic technical notions that will be useful. Let $L\in\f$ be a leaf of a complete singular Riemannian foliation of $M$, and consider a tubular neighborhood $U:=\tub_\varepsilon(P)$ of radius $\varepsilon>0$ of a connected, relatively compact, open subset $P\subset L$. That is, $U$ is the image of $B_\varepsilon^P:=\{V\in\nu P\ |\ \|V\|<\varepsilon\}$ by the normal exponential map $\exp^\perp:\nu L\to M$, where $\varepsilon$ is taken small enough so that $\exp^\perp|_{B_\varepsilon^P}$ is a diffeomorphism onto $U$. There is an orthogonal projection $\pi_P:U\to P$. By decreasing $\varepsilon$ and shrinking $P$ if necessary, we can further assume that $U$ is a \textit{distinguished tubular neighborhood}, i.e., that it also satisfies the following:
\begin{enumerate}[(i)]
\item $L_y$ is transverse to the \textit{slice} $S_x:=\pi_P^{-1}(x)=\exp^\perp_x(B_\varepsilon(0))$, for each $y\in U$, $x=\pi_P(y)$, and
\item \label{item: condition plaques} $P$ is a leaf of a (regular) simple subfoliation of $\f|_U$ given by the fibers of a submersion $\rho:U\to \pi_P^{-1}(x)$.
\end{enumerate}
The connected component $P_y$ of $L_y\cap\tub_\varepsilon(P)$ containing $y$ is a \textit{plaque} through $y$. Condition \eqref{item: condition plaques} is a natural generalization of the definition of regular foliations, and it is possible to check that the existence of distinguished tubular neighborhoods is in fact equivalent to the transitivity of $\mathfrak{X}(\f)$ in each leaf, in definition of singular foliations. 

\begin{definition}[Slice foliation]
With the notation above, we define the \textit{slice foliation at $x$} as the foliation $\f|_{S_x}$ of $S_x$ given by the intersections $P_y\cap S_x$, for $y\in\tub_\varepsilon(P)$.
\end{definition}

Given a distinguished tubular neighborhood $\tub_\varepsilon(P)$, if $\varepsilon_1,\varepsilon_2=\lambda \varepsilon_1\in(0,\varepsilon)$ one can define the \textit{homothetic transformation}
$$h_\lambda:\tub_{\varepsilon_1}(P)\ni\exp^\perp(V)\longmapsto \exp^\perp(\lambda V)\in\tub_{\varepsilon_2}(P).$$

\begin{figure}
\centering{
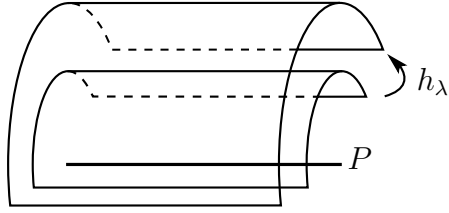}
\caption{The homothetic transformation sends plaque to plaque}
\label{homothetictransformation}
\end{figure}

\begin{lemma}[Homothetic transformation lemma {\cite[Lemma 6.2]{molino}}]\label{lemma: homothetic}
The map $h_\lambda$ sends plaque to plaque (see Figure \ref{homothetictransformation}).
\end{lemma}

Lemma \ref{lemma: homothetic} actually holds more generally when $L$ is, instead of a leaf, a submanifold $N$ which is saturated by leaves of $\f$, all of them of the same dimension, and the definition on $U$ is adapted accordingly. This result is fundamental for the theory of singular Riemannian foliations. It is used to prove, for instance, that the union $\Sigma_r$ of all leaves of $\f$ of dimension $r$, called a \textit{stratum}, is an embedded submanifold \cite[Proposition 6.3]{molino}. This provides a stratification
$$M=\bigsqcup_r \Sigma_r$$
of $M$ such that the restriction $\f_r:=\f|_{\Sigma_r}$ is a regular foliation, for each $r$. The stratum of the leaves of maximal dimension is the \textit{regular stratum} of $\f$, which we also denote by $\Sigma_{\mathrm{reg}}=\Sigma_{\dim(\f)}$, and all other strata are called \textit{singular}. The union $\Sigma_{\mathrm{sing}}$ of all singular strata is the \textit{singular locus} of $\f$. We will also often denote the most singular stratum by $\Sigma_{\mathrm{min}}$, called the \textit{minimal stratum}. Using Lemma \ref{lemma: homothetic} one proves moreover that:
\begin{enumerate}[(i)]
\item Each $\Sigma_{r}$ is transversely totally geodesic, meaning that a geodesic which is perpendicular to a leaf $L\in\Sigma_r$ and tangent to $\Sigma_r$ remains within $\Sigma_r$ and is, in particular, a geodesic of $\Sigma_r$ with respect to the restriction of the metric $\mathrm{g}$.
\item \label{item: restriction to stratum is RRF} Thus $\mathrm{g}_r:=\mathrm{g}|_{\Sigma_r}$ is a bundle-like metric for $\f_r$, which is hence a (regular) Riemannian foliation. The transverse metric it induces will be denoted by $\mathrm{g}_r^T$.
\item If $L\subset \Sigma_r$ then $\overline{L}\subset\Sigma_r$ \cite[Lemma 6.4]{molino}.
\end{enumerate}
Furthermore, each $\Sigma_r$ is obviously saturated, so Lemma \ref{lemma: homothetic} can also be applied for $N=\Sigma_r$, from what one concludes:
\begin{enumerate}[(i)]\addtocounter{enumi}{3}
\item All singular strata have codimension at least $2$, so $\Sigma_{\mathrm{reg}}$ is an open, dense submanifold of $M$.
\end{enumerate}

\begin{definition}[Transverse Killing vector fields]
We say that a transverse field $\overline{X}\in\mathfrak{l}(\f)$ is a \textit{transverse Killing vector field} of $\f$ if its restriction to each stratum $\Sigma_r$ is a transverse vector field for $(\f_r,\mathrm{g}_r^T)$ (see Item \eqref{item: restriction to stratum is RRF} above). The algebra of $\f$-transverse Killing vector fields will be denoted by $\mathfrak{iso}(\f)$.
\end{definition}

\subsection{Orbit-like, infinitesimally closed and linearized foliations}

As we previously saw, for each $x\in M$ we have a slice foliation $\f|_{S_x}$ of a slice $S_x$. Its pullback by the exponential map is a singular Riemannian foliation of $B_\varepsilon(0)\subset T_xS_x$ with respect to $\metric_x$ and thus, by Lemma \ref{lemma: homothetic}, can be extended via homotheties to a singular Riemannian foliation $\f_x$ on the whole of $T_xS_x$, called the \textit{infinitesimal foliation at $x$}. Notice that if $\f$ is regular, then $\f_x$ is the trivial foliation of $T_xS_x$ by points.

\begin{definition}[Infinitesimally closed/homogeneous and orbit-like foliations]
A singular Riemannian foliation $\f$ is said to be:
\begin{itemize}
\item \textit{infinitesimally closed foliation} if the infinitesimal foliations $\f_x$ are closed for all $x$,
\item \textit{infinitesimally homogenous} if the infinitesimal foliations $\f_x$ are homogenous for all $x$, and
\item \textit{orbit-like} if $\f$ is both infinitesimally closed and infinitesimally homogenous. 
\end{itemize}
\end{definition}

The property of being infinitesimally homogeneous is invariant by foliate diffeomorphisms, in the sense that if $\Phi:(M,\f)\to(N,\g)$ is a foliate diffeomorphism, then $\f$ is infinitesimally homogeneous if and only $\g$ is infinitesimally homogeneous \cite[Proposition 2.9]{alex4}. Hence the same is true for the property of being orbit-like.

\begin{example}[Closures of regular Riemannian foliations]
The closure $\overline{\f}$ of regular Riemannian foliation $\f$ is orbit-like.
\end{example}

The next example turns out to be very relevant in geometry (see \cite{Berndt-Console-Olmos} and \cite{Wilking}).

\begin{example}[Holonomy foliations]
Examples of orbit-like foliations can be constructed as follows. Suppose $L$ is a Riemannian manifold, and $E$ is an Euclidean vector bundle over $L$, with inner product $\langle\ ,\,\rangle_x$ on each fiber $E_x$, and suppose $\nabla^E$ is a metric connection on $E$, that is, it satisfies $X\langle \xi, \eta \rangle=\langle\nabla^E_X\xi, \eta\rangle+\langle \xi, \nabla^E_X\eta\rangle$. Then $\nabla^E$ induces a Riemannian metric $\metric^E$ on $E$, the \textit{connection (Sasakian) metric}, and a parallel transport on $E$ given as follows: for $X\in E_x$ and a curve $\gamma:[0,1]\to L$ with $\gamma(0)=x$, there exists a unique lift $X(t)$, $t\in [0,1]$ with $X(0)=X$ such that $\nabla^{E}_{\gamma'(t)}X(t)=0$ for every $t\in [0,1]$. We define the \textit{holonomy foliation} $\f^E$ on $E$ by declaring two vectors $X, Y\in E$ to be in the same leaf if they can be connected to one another via a composition of parallel transports with respect to $\nabla^{E}$ (see Figure \ref{holonomyfoliation}). This defines a singular Riemannian foliation on $E$ for which $\metric^E$ is adapted. For a point $x$ along the zero section $L$, the infinitesimal foliation $\f_x$ coincides with the homogeneous foliation given by the orbits of the holonomy group $\hol_x$ of the connection $\nabla^E$ acting by isometries on the fiber $E_x$. Similarly, at a point $X\in E_x$ the infinitesimal foliation is given by the orbits in $\nu_X L_X$ of the stabilizer $H_X\subset H_x$ of $X$. Therefore $\f^E$ is infinitesimally homogeneous. In addition if the connected component of $\hol_x$ is compact, then $\f^{E}$ is orbit-like. This happens for example if $E=TL$ or if $E=\nu(L)$ when $L$ is an embedded submanifold of an Euclidean space (see \cite{Berndt-Console-Olmos}).
\end{example}

\begin{figure}
\centering{
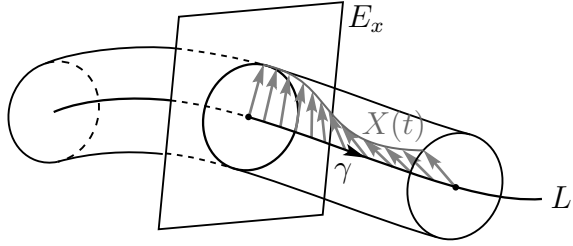}
\caption{The leaves of an holonomy foliation}
\label{holonomyfoliation}
\end{figure}

We end this section with a technical construction which will be needed later. Let $U=\tub_\varepsilon(P)$ be a distinguished tubular neighborhood, for $P$ contained in a saturated submanifold $N\subset\Sigma_r$. If $X\in\mathfrak{X}(\f|_{U})$ is given, we can produce another vector field $X^\ell$ on $U$ given by
$$X^\ell=\lim_{\lambda\to 0} \dif h_\lambda^{-1}(X\circ h_\lambda),$$
which is called the \textit{linearization of $X$ with respect to $P$}. It is smooth, invariant under homothetic transformations and coincides with $X$ along $P$ \cite[Proposition 13]{mendes}. The module $\mathfrak{X}(\f|_{U})^\ell$ of linearized vector fields spans a foliation $\f^\ell$ of $U$, in the sense that the leaves of $\f^\ell$ are the orbits of the pseudogroup of local diffeomorphisms $\mathscr{H}^\ell(U,\f)$ generated by the flows of the vector fields in $\mathfrak{X}(\f|_{U})^\ell$.

\begin{definition}[Linearized foliation]\label{definition: linearized foliation}
The foliation $\f^\ell$ of $U$ is called the \textit{linearization of $\f$ with respect to $P$}, or just the \textit{linearized foliation} when $\f$ and $P$ are clear from the context.
\end{definition}

The metric $\metric$ in general is not adapted to the linearization $\f^\ell$, but it is possible to construct a new metric $\metric^\ell$ on $U$ turning $\f^\ell$ into a singular Riemannian foliation (see \cite[Section 7]{alex4} for details. Now let $S_x=\pi_P^{-1}(x)$ be an $\f$-slice and, in analogy with the definition of the slice foliation $\f|_{S_x}$, consider the partition $\f^\ell|_{S_x}$ given by the connected components of $L\cap S_x$ as $L$ ranges between the leaves of $\f^\ell$. By construction, the pullback $\exp_x^*(\f^\ell|_{S_x})$ is invariant by rescalings, so we can also extend it to a foliation $(\f^\ell)_x$ of the whole $T_xS_x$. One proves that $(\f^\ell)_x$ coincides with the linearization of $(\f_x,\metric_x)$ with respect to the origin in $T_xS_x$ \cite[\S 6.4]{molino}, that is $(\f^\ell)_x=(\f_x)^\ell$, so it is safe to denote both of them by $\f_x^\ell$. The next result is the main reason why we are interested in this object.

\begin{proposition}[{\cite[Proposition 2.10]{alex4}}]\label{linearized infinitesimal is homogeneous}
The foliation $\f^{\ell}$ is the maximal infinitesimally homogenous subfoliation of $\f|_{U}$, for a tubular neighborhood $U$ of a leaf closure $J=\overline{L}$. Moreover, for each $x\in U$ the foliation $\f_x^\ell$ is given by the connected components of the orbits of the Lie group $\mathrm{O}(\f_x)$ of linear isometries of $(T_xS_x,\metric_x)$ sending each leaf of $\f_x$ to itself.
\end{proposition}

\section{Molino's conjecture and its proof}\label{section: molino conjecture and its proof}

In this section we address the question of whether the closure $\overline{\f}$ of a singular Riemannian foliation is again a singular Riemannian foliation. Molino made this conjecture in the 1980's, and it remained open until 2017, when it was proved positive in \cite{alex4}. Before we present the results in \cite{alex4}, let us see what can be achieved by applying the structural theorem for regular foliations (Theorem \ref{Theo: molino structural}) to each stratum of a singular Riemannian foliation.

\subsection{Molino sheaf of a singular Riemannian foliation}\label{section: molino sheaf singular}
Consider a complete singular Riemannian foliation $(M,\f)$. As we saw in item \eqref{item: restriction to stratum is RRF} above, the restriction of a singular Riemannian foliation $\f$ to each stratum $\Sigma_r$ is a regular Riemannian foliation. Although $\f_r$ is not necessarily complete, one can still apply Theorem \ref{Theo: molino structural} to it because its holonomy pseudogroup is complete. It therefore follows that its leaf closures are submanifolds of $\Sigma_r$, and in particular the leaf closures in $\overline{\f}$ are submanifolds of $M$. Moreover, the distance between two leaf closures is locally constant, since this is true for the leaves themselves, hence the partition $\overline{\f}$ is a transnormal system on $M$.

Each $\f_r$ has a locally constant Molino sheaf $\mathscr{C}_r:=\mathscr{C}_{\f_r}$ of germs of local transverse Killing vector fields that describes the closure $\overline{\f_r}$. Consider, in particular, the Molino sheaf $\mathscr{C}_{\mathrm{reg}}$ of the restriction $\f_{\mathrm{reg}}$ to the regular stratum. The opposite Lie algebra of its stalk is called the \textit{structural Lie algebra} of $\f$, and denoted simply by $\mathfrak{g}$. The motivation for this is that $\mathscr{C}_{\mathrm{reg}}$ extends continuously to a locally constant sheaf $\mathscr{C}_{\f}$ on $M$, the \textit{Molino sheaf of $\f$}, with stalk $\mathfrak{g}^{-1}$ \cite[Lemma 6.5]{molino}. In fact, let us briefly present how this extension is obtained. Suppose $\Sigma_{r}$ is the stratum of the singular leaves of maximal dimension and let $P\subset\Sigma_{r}$ be an open, relatively compact, simply connected subset. If $\codim(\Sigma_{r})=2$, then by Lemma \ref{lemma: homothetic} one concludes that the restriction of $\f$ to the boundary of a tubular neighborhood $\tub_\varepsilon(P)$ is the pullback $\pi_P^*(\f_r)$. So, $\mathscr{C}_{\mathrm{reg}}$ coincides with $\pi_P^{-1}(\mathscr{C}_r)$ on $\tub_\varepsilon(P)$, hence the result. Now, if $\codim(\Sigma_{r})>2$, then $\tub_\varepsilon(P)\setminus P$ is a simply connected open subset of $\Sigma_{\mathrm{reg}}$, on which $\mathscr{C}_{\mathrm{reg}}$ is therefore constant and thus extends to $\tub_\varepsilon(P)$ by continuity. The extension of $\mathscr{C}_{\mathrm{reg}}$ to the other strata is then done similarly.

\begin{example}[Molino sheaf of a homogeneous singular Riemannian foliation]\label{exe: molino sheaf of homogeneous singular foliation}
Suppose $\f$ is given by the connected components of the orbits of a Lie group $H<\mathrm{Iso}(M)$ (see Example \ref{ex: homogeneous singular} and Proposition \ref{prop: homogeneous srf}). Then, in analogy to the regular case (see Example \ref{example: homogeneous foliations are killing}), $\mathscr{C}_{\f}$ is the sheaf of germs of the transverse Killing vector fields induced by the fundamental Killing vector fields of the action of the closure $\overline{H}<\mathrm{Iso}(M)$.
\end{example}

Furthermore, as the above example suggests, it is possible to prove that each sheaf $\mathscr{C}_r$ is a quotient of $\mathscr{C}_{\f}$ \cite[Proposition 6.8]{molino}. In particular, the structural algebra $\mathfrak{g}_r$ of $\f_r$ is a quotient of $\mathfrak{g}$, for each $r$.

\subsection{Blow ups and desingularization}\label{Sec-blowup}

Let us review a useful technical tool for singular Riemannian foliations that allows one to ``desingularize'' a singular Riemannian foliation $\f$ on a compact manifold $M$ by constructing from it another compact Riemannian manifold $(M^{\mathrm{B}},\metric^{\mathrm{B}})$, a regular foliation $\f^{\mathrm{B}}$, and a foliate smooth map $\rho:M^{\mathrm{B}}\to M$ with good geometric properties. For instance, $\rho$ restricts to a foliate diffeomorphism outside $\Sigma^{\mathrm{B}}:=\rho^{-1}(\Sigma_{\mathrm{sing}})$ and to an isometry outside a narrow open neighborhood of $\Sigma^{\mathrm{B}}$. This technique, inspired by the blow-up methods for resolution of singularities in algebraic geometry, appeared in \cite{molino3} for closures of regular Riemannian foliations and was generalized to arbitrary singular Riemannian foliations on compact manifolds in \cite{alex5}.

The foliation $(M^{\mathrm{B}},\f^{\mathrm{B}},\metric^{\mathrm{B}})$ is obtained by successively blowing it up along the most singular strata. It is instructive to review this process in more detail. Denote by $\Sigma:=\Sigma_{\mathrm{min}}$ its minimal stratum and by $U:=\tub_\varepsilon(\Sigma)$ a tubular neighborhood of $\Sigma$. One proves the following \cite[Theorem 1.2]{alex5}: 
\begin{enumerate}[(i)]
\item $\widehat{U}:=\{ (x,[X]) \in U\times \mathbb{P}(\nu\Sigma)\ |\ x=\exp^{\nu}(tX)\mbox{ for }|t|<r \}$ is a smooth manifold, called the \textit{blow-up of $U$ along $\Sigma$}, and the \textit{blow-up projection} $\hat{\rho}:\widehat{U}\to U$ defined by $\hat{\rho}(x,[X])=x$ is smooth.
\item $\widehat{\Sigma}:=\hat{\rho}^{-1}(\Sigma)=\{\hat{\pi}([X], [X])\in \widehat{U}\}=\mathbb{P}(\nu \Sigma)$, where $\hat{\pi}:\mathbb{P}(\nu\Sigma)\to \Sigma$ is the canonical projection. 
\item There exists a singular foliation $\widehat{\f}$ on $\widehat{U}$ whose leaves in the minimal stratum have dimension strictly greater then those on the minimal stratum of $\f$ and so that $\hat{\rho}:(\widehat{U}\setminus\widehat{\Sigma}, \widehat{\f}) \to (U\setminus\Sigma, \f)$ is a foliate diffeomorphism. In addition, if $\f$ is homogeneous then the leaves of $\widehat{\f}$ are also homogenous.
\item \label{metric of blow up} There exists a Riemannian metric $\hat{\metric}$ on $\widehat{U}$ adapted to $\widehat{\f}$.
\end{enumerate}

Let us briefly recall the construction of the metric in item \eqref{metric of blow up}. Consider the smooth distribution $\mathcal{S}$ on $U$ given by $\mathcal{S}_y=T_yS_{x}$, where $S_x$ is a slice of $L_x$ at $x$  with respect to the original metric and denote $y=\exp(X)$, for $X\in\nu_{x}\Sigma$. Recall that there exists a metric $\tilde{\metric}$ so that the normal space $\mathcal{P}$ to $\mathcal{S}_{y}$ is tangent to the leaf $L_{y}$ and so that $(U,\f,\metric)$ has the same transverse geometry as $(U,\f,\tilde{\metric})$, i.e., the distance between the plaques is the same regardless which metric we use. Then we have the decomposition
$$ T_{y} M = \mathcal{P}_{y}\oplus 
\mathcal{S}_{y}^{\mathrm{s}}\oplus \mathcal{S}_{y}^{\mathrm{r}} \oplus \mathcal{S}_{y}^{\mathrm{c}},$$
where
\begin{itemize}
\item $\mathcal{P}_{y}$ is orthogonal to $\mathcal{S}_{y}$, with respect to $\tilde{\metric}$,
\item $\mathcal{S}_{y}^{\mathrm{s}}\subset \mathcal{S}_{y}$ is tangent to the spheres $\exp_x(\nu \Sigma\cap B_{\|X\|}(0))$,
\item $\mathcal{S}_{y}^{\mathrm{r}}$ is the line generated by $\od{}{t}\exp_{x}(tX)|_{t=1}$, and
\item $\mathcal{S}_{y}^{\mathrm{c}}$ is the orthogonal complement of $\mathcal{S}_{y}^{\mathrm{s}}\oplus\mathcal{S}_{y}^{\mathrm{r}}$ in $\mathcal{S}_{y}$.
\end{itemize}
We now define a metric $\hat{\metric}$ on $U\setminus \Sigma$ which is adapted to $\f$. Let $f:(0,r)\to \mathbb{R}$ be a smooth function so that $f(t)=\frac{r^{2}}{t^2}$ for $0<t<\frac{r}{8}$ and $f(t)=1$ for $\frac{r}{4}<t<\frac{r}{2}$, and set
$$\hat{\metric}_{y}(Z,W)
=\tilde{\metric}(Z^\perp, W^\perp)
+f(\|X\|)\tilde{\metric}(Z^{\mathrm{s}},W^{\mathrm{s}})
+\tilde{\metric}(Z^{\mathrm{r}},W^{\mathrm{r}})
+\tilde{\metric}(Z^{\mathrm{c}}, W^{\mathrm{c}}).$$
It follows that $\hat{\metric}_{\exp(X)}=\tilde{\metric}_{\exp(X)}$ if $\frac{1}{4} \leq \|X\|\leq \frac{1}{2}$. Notice that $\hat{\metric}$ is adapted to $\f$, because $f(\|\exp^{-1}(x)\|)$ is constant along $L_{x}$ and hence $\hat{\metric}$ is basic on each stratum.

Since the distribution $\mathcal{S}$, i.e., the normal distribution to $\mathcal{P}$ with respect to $\tilde{\metric}$, can be deformed to the normal distribution to $\mathcal{P}$ with respect to $\metric$ without changing the transverse metric, we can extend the metric $\hat{\metric}$ on $\tub_{r/4}(\Sigma)$ to a new metric $\hat{\metric}$ on $U$ so that $\hat{\metric}_{\exp(rX/4)}=\tilde{\metric}_{\exp(rX/4)}$ and $\hat{\metric}_{\exp(tX)}=\metric_{\exp(tX)}$, for $r/2<t$ and $\|X\|=1$. The pullback $(\hat{\rho})^{*}\hat{\metric}$ defines a smooth metric on $\widehat{U}$ so that $\widehat{\f}$ turns into a singular Riemannian foliation on $\widehat{U}$. Finally we can extend the metric $\hat{\metric}$ on $\widehat{U}$ to a metric on the connected sum
$$\widehat{M}=(\widehat{U},\partial U)\#(M\setminus U, \partial U)$$
such that $\widehat{\f}$ is a singular Riemannian foliation on $\widehat{M}$.

We can now define by induction $M_{1}:=\widehat{M}$ and $M_k:=\widehat{M_{k-1}}$, with blow-up projection $\rho_{k}:M_{k}\to M_{k-1}$ and blow-up foliation $\f_k:=\widehat{\f_{k-1}}$. We further define $(M^{\mathrm{B}},\f^{\mathrm{B}},\metric^{\mathrm{B}})$ as the last blow-up space in this process (which eventually ends since $\dim(\f)$ is finite) and $\rho:M^{\mathrm{B}}\to M$ as $\rho=\rho_n\circ \cdots \circ \rho_{1}$.

Although in this survey we are interested in understanding singular Riemannian foliations with non-closed leaves, we recall here the following result that illustrates further geometric properties of blow-up's.  

\begin{theorem}[{\cite[Theorem 1.5]{alex5}}]
Let $\f$ be a closed singular Riemannian foliation on a compact Riemannian manifold $M$. Then for each small positive $\varepsilon>0$ there exists a regular Riemannian foliation $\f^{\mathrm{B}}$ with compact leaves on a compact Riemannian manifold $M^{\mathrm{B}}$ and a smooth surjective desingularization map $\rho:M^{\mathrm{B}}\to M$ that is induces an $\varepsilon$-isometry between the leaf spaces, that is, if $x$ and $y$ are points in $M^{\mathrm{B}}$ then
$$| d(L_{\rho(x)}, L_{\rho(y)}) - d(L_{x}^{\mathrm{B}}, L_{y}^{\mathrm{B}})   |<\varepsilon.$$
In particular the metric space $M/\f$ is a Gromov--Hausdorff limit of a sequence of Riemannian orbifolds.
\end{theorem}

We will be specially interested in the following property.

\begin{proposition}\label{prop: lifting local transverse isometries to the blow up}
Let $\f$ be an infinitesimally closed singular Riemannian foliation on a compact Riemannian manifold $M$. Then every local $\f$-transverse Killing vector field $\overline{X}$ admits a lift to a local $\f^{\mathrm{B}}$-transverse Killing vector field $\overline{X^{\mathrm{B}}}$, in the sense that the flows of $X$ and $X^{\mathrm{B}}$ satisfy $\rho\circ\varphi^{\mathrm{B}}=\varphi\circ\rho$. 
\end{proposition}

\begin{remark}
In the general case where $\f$ is not necessarily infinitesimally compact one has a similar conclusion if in addition one supposes that the local flows of the transverse isometries are contained in the closure of linearized holonomies, a concept that we present below.
\end{remark}

\subsection{The proof of Molino's conjecture}\label{section: the proof of Molinos conjecture}

In \cite{alex9} the authors establish Molino's conjecture for orbit-like foliations using Theorem \ref{theo: smooth lift}. We already saw that $\overline{\f}$ is a transnormal system, so it remained to show that any vector $X_x\in\nu_xL_x\cap T_x\overline{L}_x$ can be extended to a smooth vector field $X\in\mathfrak{X}(\overline{\f})$. Notice that this is a local problem since we can find an extension $X$ in a neighborhood of $x$ and then use a partition of unity to extend it by $0$ to a vector field in $\mathfrak{X}(\overline{\f})$. The main ingredient in the proof of Molino's conjecture in \cite{alex4} is the following smooth lifting of (metric) isometries between leaf spaces.

\begin{theorem}{{\cite[Theorem 1.1]{alex9}}}\label{theo: smooth lift}
Let $M$ be a complete Riemannian manifold with a proper isometric action $G\times M\to M$ of a Lie group $G$, and suppose that $\varphi:\mathcal{D}\to M/G$ is a continuous local flow of isometries, where $\mathcal{D}$ is an open neighborhood of a point $(\bar{x},0)\in M/G\times\mathbb{R}$. Then $\varphi$ is the projection of a $G$-equivariant smooth flow on the preimage of $\mathcal{D}$ in $M$.
\end{theorem}

To comment on the proof we will need a generalization of the notion of slice, which will also be useful later in the study of the dynamical behavior of a singular Riemannian foliation $\f$. Let $J$ be an $\f$-saturated manifold contained in some stratum $\Sigma_r$ and let $S\subset{U}$ a slice of the restriction $\f|_J$. Moreover, consider the denote the restriction $(\nu J)|_S$ of the normal bundle $\nu J$ to $S$.

\begin{definition}[Reduced space]
We define the \textit{reduced space of $\f$ along $S$} as the manifold $\mathsf{N}:=\exp(\nu^\varepsilon J)|_{S})$, where $(\nu^\varepsilon J)|_{T}=\{\xi\in (\nu J)|_S\ |\ \|\xi\|<\varepsilon\}$ and $\varepsilon>0$ is small enough so that $\exp$ is a diffeomorphism onto $\mathsf{N}$.
\end{definition}

With the footpoint projection, the reduced space is a fiber bundle $\mathsf{p}_T:\mathsf{N}\to T$ whose fibers contain the leaves of the foliation $\f_{\mathsf{N}}$ defined by the intersections of $\f$ with $\mathsf{N}$. Furthermore, one can endow $\mathsf{N}$ with a metric $\tilde{\metric}$ which is adapted to $\f_{\mathsf{N}}$ and preserves the transverse metric of $\f$ \cite[Proposition 2.20]{alex9}.

The point of this construction is that when $\f$ is bundle-like, the foliation $\f_{\mathsf{N}}$ is homogeneous, given by the orbits of a compact Lie group \cite[Corollary 2.25]{alex9}. Hence one can apply (a slight generalization of) Theorem \ref{theo: smooth lift} to continuous flows of isometries on $\mathsf{N}/\f_{\mathsf{N}}$. Having this, it is then a matter of finding such a flow that corresponds to a given $X_x\in\nu_xL_x\cap T_x\overline{L}_x$. The authors accomplish that by first generalizing the notion of (regular) holonomy pseudogroup by obtaining, for $\mathsf{N}$ a reduced space along a slice, a pseudogroup of local metric isometries $\mathscr{H}(\f_{\mathsf{N}})$ acting on $\mathsf{N}/\f_{\mathsf{N}}$ and capturing the recurrence of the leaves on $\mathsf{N}$. Now consider the desingularization $(\mathsf{N}^{\mathrm{B}},\f^{\mathrm{B}}_{\mathsf{N}})$ of $\f_{\mathsf{N}}$. Since $\f^{\mathrm{B}}_{\mathsf{N}}$ is regular, Theorem \ref{theorem: salem} implies that $\overline{\mathscr{H}_{\f^{\mathrm{B}}_{\mathsf{N}}}}$ is a Lie pseudogroup, so lifting the local projection of $X_x\in\nu_xL_x\cap T_x\overline{L}_x$ to a vector in $\mathsf{N}^{\mathrm{B}}$, which is tangent to the leaf closures, one concludes that it extends to a local Killing vector field. But there is a bijection between local isometries in $\overline{\mathscr{H}_{\f^{\mathrm{B}}_{\mathsf{N}}}}$ and local isometries in the closure $\overline{\mathscr{H}(\f_{\mathsf{N}})}$ (in the compact-open topology) \cite[Lemma 4.2]{alex9}, so the authors obtain the desired continuous flow and hence a smooth vector field on $\mathsf{N}$ that can then be extended to the desired vector field on $M$ extending $X_x$.

Having Molino's conjecture for orbit-like foliations, the conclusion for an arbitrary singular Riemannian foliation is obtained by using the linearization $\f^\ell$. More precisely, on a local neighborhood a leaf $L\in\f$ the authors show in \cite{alex4} that there is an orbit-like foliation $\widehat{\f}^\ell$, obtained from $\f^\ell$ by taking the ``local closure'' of the leaves of $\f^\ell$ (see details in \cite[\S 6]{alex4}), such that:
\begin{enumerate}[(i)]
\item $\widehat{\f}^\ell$ coincides with $\f$ on $\overline{L}$, and
\item \label{item: property local closure}$\overline{\widehat{\f}^\ell}=\overline{\f^\ell}\subset \overline{\f}$.
\end{enumerate}
From this the proof of Molino's conjecture is clear: by the orbit-like case $X_x$ can be extended to a smooth vector field which is tangent to $\overline{\f^\ell}$ and hence tangent to $\overline{\f}$, by item \eqref{item: property local closure}.

\begin{remark}
It is necessary to construct $\widehat{\f}^\ell$ because $\f^\ell$ may not be orbit-like, since the foliations $\f_x^\ell$ need not to be closed, although they are always homogeneous (recall Proposition \ref{linearized infinitesimal is homogeneous}). Notice, hence, that the technical step of constructing $\widehat{\f}^\ell$ is not needed when $\f$ is infinitesimally closed.
\end{remark}

\subsection{Strong Molino conjecture}

Combining the existence of the sheaf $\mathscr{C}_{\f}$ with the fact that Molino's conjecture is true we can state the following structural theorem, which bears great resemblance with the regular case (recall Theorem \ref{Theo: molino structural}).

\begin{theorem}[Structural theorem for singular Riemannian foliations {\cite[Theorem 6.2]{molino},\cite[Theorem]{alex4}}]
Let $\f$ be a complete singular Riemannian foliation of $M$ and let $\mathrm{g}$ be an adapted metric. Then
\begin{enumerate}[(i)]
\item The closure $\overline{\f}$ is a singular Riemannian foliation with adapted metric $\mathrm{g}$,
\item There exists a locally constant sheaf of Lie algebras $\mathscr{C}_{\f}$ on $M$ which induces $\mathscr{C}_{\mathrm{reg}}$ on $\Sigma_{\mathrm{reg}}$ and whose restriction to a singular stratum $\Sigma_r$ admits $\mathscr{C}_r$ as a quotient sheaf.
\end{enumerate}
\end{theorem}

Unlike the regular case, however, it is not possible to conclude from what we have seen so far that $\mathscr{C}_{\f}$ is a sheaf of Lie algebras of germs of transverse Killing vector fields, since the extension of $\mathscr{C}_{\mathrm{reg}}$ to $\mathscr{C}_\f$ is continuous: we do not know whether the extensions of sections admit smooth representatives. In fact, this would imply Molino's conjecture, since in this case for a small neighborhood $U$ the (smooth) fields in $\mathscr{C}_{\f}(U)$ would be transitive on the closures of the leaves. For this reason, we state the following:

\begin{conjecture}[Strong Molino conjecture]
The sheaf $\mathscr{C}_{\f}$ is a sheaf of Lie algebras of germs of transverse Killing vector fields.
\end{conjecture}

Molino proposes this conjecture in \cite[p. 215]{molino}. It does not follow directly from the already mentioned proof of Molino's conjecture that appears in \cite{alex4}, but from the results in \cite{alex9} that we saw in Section \ref{section: the proof of Molinos conjecture} one can conclude that it is true for orbit-like foliations.

\section{Singular Killing foliations}\label{section: singular Killing}

In this section we propose a definition for singular Killing foliations. On the one hand, as discussed in Section \ref{section: molino conjecture and its proof}, we do not know whether in general the Molino sheaf $\mathscr{C}_\f$ is a sheaf of germs of transverse Killing vector fields -- the strong Molino conjecture. On the other hand, the fact that $\mathscr{C}_\f$ is indeed a globally trivial sheaf of germs of transverse Killing fields when $\f$ is a regular Killing foliation is of fundamental relevance for this class of foliations. So we will assume this \textit{a priori} in our generalization of Killing foliations to the singular setting.

\begin{definition}[Singular Killing foliation]
A complete singular Riemannian foliation $(M,\f)$ is a \textit{singular Killing foliation} if it's Molino sheaf $\mathscr{C}_\f$ is a globally constant sheaf of Lie algebras of germs of transverse Killing vector fields.
\end{definition}

Notice that the structural algebra of a singular Killing foliation $\f$ is Abelian, since it is the structural algebra of the (regular) Killing foliation $\f_{\mathrm{reg}}$. We will therefore follow our notation of the regular case and denote it by $\mathfrak{a}$. We have an isomorphism $\mathfrak{a}=\mathscr{C}_\f(M)$. Also in analogy with the regular case, if $M$ is simply connected then $\mathscr{C}_\f$ is automatically globally constant. Thus a singular Riemannian foliation of a simply connected manifold is a singular Killing foliation, provided is satisfies the strong Molino conjecture. We also have the following:

\begin{example}[Homogeneous singular Riemannian foliations are Killing]
As we saw in Example \ref{exe: molino sheaf of homogeneous singular foliation}, if $\f$ is homogeneous, given by the orbits of $H<\mathrm{Iso}(M)$, then $\mathscr{C}_{\f}$ is induced by the fundamental Killing vector fields of the action of $\overline{H}<\mathrm{Iso}(M)$, hence $\f$ is a singular Killing foliation.
\end{example}

It seems therefore relevant to study this class of foliations and investigate to what extent the results concerning regular Killing foliations that we saw in Sections \ref{section: transverse topology of killing} and \ref{section: transverse geometry of Killing} generalize to the singular setting. This is intended to a forthcoming paper. For now we point out, for instance, the following.

\begin{proposition}
Let $(M,\f)$ be a singular Killing foliation and $\mathfrak{a}$ its structural algebra. Then we have a transverse infinitesimal action of $\mathfrak{a}$ on $\f$ which turns $(\Omega(\f),d)$ into an $\mathfrak{a}^\star$-algebra.
\end{proposition}

\begin{proof}
The proof is analogous to that of the regular case \cite[Proposition 3.12]{goertsches}. The transverse infinitesimal action of $\mathfrak{a}$ is given by the isomorphism $\mathfrak{a}\cong\mathscr{C}_\f(M)<\mathfrak{l}(\f)$, so we can identify $\mathfrak{a}\equiv\mathscr{C}_\f(M)$. For each $X\in\mathfrak{a}$ we define the derivations $\mathcal{L}_X:=\mathcal{L}_{\tilde{X}}$ and $\iota_X:=\iota_{\tilde{X}}$, where $\tilde{X}\in\mathfrak{L}(\f)$ represents $X$. Notice that these operators are well defined, since we are restricted to forms on $\Omega(\f)$, and inherit the needed $\mathfrak{a}^\star$-algebra relations from $\Omega(M)$.

It thus only remains to show that $\Omega(\f)$ is closed with respect to each $\iota_X$ and $\mathcal{L}_X$. In fact, if $Y\in\mathfrak{X}(\f)$, then $\iota_Y\iota_{\tilde{X}}\omega=-\iota_{\tilde{X}}\iota_Y\omega=0$, since $\iota_Y\omega=0$, and $\mathcal{L}_Y\iota_{\tilde{X}}\omega=\iota_{\tilde{X}}\mathcal{L}_Y\omega+\iota_{[Y,\tilde{X}]}\omega=0$, since $\mathcal{L}_Y\omega=0$ and $[Y,{\tilde{X}}]\in\mathfrak{X}(\f)$. Hence $\iota_X\omega\in\Omega(\f)$. Similarly, $\iota_Y\mathcal{L}_{\tilde{X}}\omega=\mathcal{L}_{\tilde{X}}\iota_Y\omega-\iota_{[\tilde{X},Y]}\omega=0$ and $\mathcal{L}_Y\mathcal{L}_{\tilde{X}}\omega=\mathcal{L}_{\tilde{X}}\mathcal{L}_Y\omega-\mathcal{L}_{[\tilde{X},Y]}\omega=0$, so we conclude that $\mathcal{L}_X\omega\in\Omega(\f)$.
\end{proof}

Therefore $\f$ possesses a natural equivariant basic cohomology, its \textit{$\mathfrak{a}$-equivariant basic cohomology}, defined as the equivariant cohomology of the associated Cartan complex $C_\mathfrak{a}(\Omega(\f))$ (see Section \ref{section: equivariant basic cohomology}):
$$H_\mathfrak{a}(\f):=H_\mathfrak{a}(\Omega(\f))=H(C_\mathfrak{a}(\Omega(\f),d_\mathfrak{a})).$$

\section*{Acknowledgements}
We thank Porf. Dirk Töben for the insightful discussions. The second author is grateful to the Department of Mathematics of the University of São Paulo (IME-USP) for the welcoming environment where part of this work was developed, and to the São Paulo Research Foundation (FAPESP) for the research funding.

\end{document}

%% file: haefligercocycle.pdf_tex
\begingroup%
  \makeatletter%
  \providecommand\color[2][]{%
    \errmessage{(Inkscape) Color is used for the text in Inkscape, but the package 'color.sty' is not loaded}%
    \renewcommand\color[2][]{}%
  }%
  \providecommand\transparent[1]{%
    \errmessage{(Inkscape) Transparency is used (non-zero) for the text in Inkscape, but the package 'transparent.sty' is not loaded}%
    \renewcommand\transparent[1]{}%
  }%
  \providecommand\rotatebox[2]{#2}%
  \newcommand*\fsize{\dimexpr\f@size pt\relax}%
  \newcommand*\lineheight[1]{\fontsize{\fsize}{#1\fsize}\selectfont}%
  \ifx\svgwidth\undefined%
    \setlength{\unitlength}{266.36408768bp}%
    \ifx\svgscale\undefined%
      \relax%
    \else%
      \setlength{\unitlength}{\unitlength * \real{\svgscale}}%
    \fi%
  \else%
    \setlength{\unitlength}{\svgwidth}%
  \fi%
  \global\let\svgwidth\undefined%
  \global\let\svgscale\undefined%
  \makeatother%
  \begin{picture}(1,0.64834085)%
    \lineheight{1}%
    \setlength\tabcolsep{0pt}%
    \put(0,0){\includegraphics[width=\unitlength,page=1]{haefligercocycle.pdf}}%
    \put(0.09483851,0.45451913){\color[rgb]{0,0,0}\makebox(0,0)[lt]{\lineheight{0}\smash{\begin{tabular}[t]{l}$U_i$\end{tabular}}}}%
    \put(0.78134059,0.41511865){\color[rgb]{0,0,0}\makebox(0,0)[lt]{\lineheight{0}\smash{\begin{tabular}[t]{l}$U_j$\end{tabular}}}}%
    \put(0.07223962,0.32512704){\color[rgb]{0,0,0}\makebox(0,0)[lt]{\lineheight{0}\smash{\begin{tabular}[t]{l}$\pi_i$\end{tabular}}}}%
    \put(0.98339127,0.27560766){\color[rgb]{0,0,0}\makebox(0,0)[lt]{\lineheight{0}\smash{\begin{tabular}[t]{l}$\pi_j$\end{tabular}}}}%
    \put(0.26115369,-2.45325821){\color[rgb]{0,0,0}\makebox(0,0)[lt]{\lineheight{0}\smash{\begin{tabular}[t]{l} \end{tabular}}}}%
    \put(0.57628614,0.5960639){\color[rgb]{0,0,0}\makebox(0,0)[lt]{\lineheight{0}\smash{\begin{tabular}[t]{l}$\mathcal{F}$\end{tabular}}}}%
    \put(0.46221346,0.05564139){\color[rgb]{0,0,0}\makebox(0,0)[lt]{\lineheight{0}\smash{\begin{tabular}[t]{l}$\gamma_{ji}$\end{tabular}}}}%
    \put(0.00627913,0.12835974){\color[rgb]{0,0,0}\makebox(0,0)[lt]{\lineheight{0}\smash{\begin{tabular}[t]{l}$S_i$\end{tabular}}}}%
    \put(0.95965987,0.06847681){\color[rgb]{0,0,0}\makebox(0,0)[lt]{\lineheight{0}\smash{\begin{tabular}[t]{l}$S_j$\end{tabular}}}}%
    \put(0,0){\includegraphics[width=\unitlength,page=2]{haefligercocycle.pdf}}%
    \put(-0.15903571,-2.30202881){\color[rgb]{0,0,0}\makebox(0,0)[lt]{\begin{minipage}{0\unitlength}\raggedright \end{minipage}}}%
  \end{picture}%
\endgroup%

%% file: holonomy.pdf_tex
\begingroup%
  \makeatletter%
  \providecommand\color[2][]{%
    \errmessage{(Inkscape) Color is used for the text in Inkscape, but the package 'color.sty' is not loaded}%
    \renewcommand\color[2][]{}%
  }%
  \providecommand\transparent[1]{%
    \errmessage{(Inkscape) Transparency is used (non-zero) for the text in Inkscape, but the package 'transparent.sty' is not loaded}%
    \renewcommand\transparent[1]{}%
  }%
  \providecommand\rotatebox[2]{#2}%
  \ifx\svgwidth\undefined%
    \setlength{\unitlength}{323.44035468bp}%
    \ifx\svgscale\undefined%
      \relax%
    \else%
      \setlength{\unitlength}{\unitlength * \real{\svgscale}}%
    \fi%
  \else%
    \setlength{\unitlength}{\svgwidth}%
  \fi%
  \global\let\svgwidth\undefined%
  \global\let\svgscale\undefined%
  \makeatother%
  \begin{picture}(1,0.34190739)%
    \put(0,0){\includegraphics[width=\unitlength,page=1]{holonomy.pdf}}%
    \put(0.91198339,0.1604226){\color[rgb]{0,0,0}\makebox(0,0)[lb]{\smash{$y=\gamma_{i_3i_1}(x)$}}}%
    \put(0,0){\includegraphics[width=\unitlength,page=2]{holonomy.pdf}}%
    \put(0.06950445,0.1723977){\color[rgb]{0,0,0}\makebox(0,0)[lb]{\smash{$x$}}}%
    \put(0.09007956,0.33662414){\color[rgb]{0,0,0}\makebox(0,0)[lb]{\smash{$T_{i_1}$}}}%
    \put(0.47339781,0.01897227){\color[rgb]{0,0,0}\makebox(0,0)[lb]{\smash{$T_{i_2}$}}}%
    \put(0,0){\includegraphics[width=\unitlength,page=3]{holonomy.pdf}}%
    \put(0.87576013,0.2558673){\color[rgb]{0,0,0}\makebox(0,0)[lb]{\smash{$T_{i_3}$}}}%
    \put(0,0){\includegraphics[width=\unitlength,page=4]{holonomy.pdf}}%
  \end{picture}%
\endgroup%

%% file: hopf.pdf_tex
\begingroup%
  \makeatletter%
  \providecommand\color[2][]{%
    \errmessage{(Inkscape) Color is used for the text in Inkscape, but the package 'color.sty' is not loaded}%
    \renewcommand\color[2][]{}%
  }%
  \providecommand\transparent[1]{%
    \errmessage{(Inkscape) Transparency is used (non-zero) for the text in Inkscape, but the package 'transparent.sty' is not loaded}%
    \renewcommand\transparent[1]{}%
  }%
  \providecommand\rotatebox[2]{#2}%
  \ifx\svgwidth\undefined%
    \setlength{\unitlength}{236.87206836bp}%
    \ifx\svgscale\undefined%
      \relax%
    \else%
      \setlength{\unitlength}{\unitlength * \real{\svgscale}}%
    \fi%
  \else%
    \setlength{\unitlength}{\svgwidth}%
  \fi%
  \global\let\svgwidth\undefined%
  \global\let\svgscale\undefined%
  \makeatother%
  \begin{picture}(1,0.99375118)%
    \put(0,0){\includegraphics[width=\unitlength,page=1]{hopf.pdf}}%
  \end{picture}%
\endgroup%

%% file: molinosheaf.pdf_tex
\begingroup%
  \makeatletter%
  \providecommand\color[2][]{%
    \errmessage{(Inkscape) Color is used for the text in Inkscape, but the package 'color.sty' is not loaded}%
    \renewcommand\color[2][]{}%
  }%
  \providecommand\transparent[1]{%
    \errmessage{(Inkscape) Transparency is used (non-zero) for the text in Inkscape, but the package 'transparent.sty' is not loaded}%
    \renewcommand\transparent[1]{}%
  }%
  \providecommand\rotatebox[2]{#2}%
  \ifx\svgwidth\undefined%
    \setlength{\unitlength}{180.39946545bp}%
    \ifx\svgscale\undefined%
      \relax%
    \else%
      \setlength{\unitlength}{\unitlength * \real{\svgscale}}%
    \fi%
  \else%
    \setlength{\unitlength}{\svgwidth}%
  \fi%
  \global\let\svgwidth\undefined%
  \global\let\svgscale\undefined%
  \makeatother%
  \begin{picture}(1,0.88138531)%
    \put(0,0){\includegraphics[width=\unitlength,page=1]{molinosheaf.pdf}}%
    \put(0.72663098,0.12460075){\color[rgb]{0,0,0}\makebox(0,0)[lb]{\smash{$\overline{L}$}}}%
    \put(0,0){\includegraphics[width=\unitlength,page=2]{molinosheaf.pdf}}%
    \put(0.29411212,0.27522825){\color[rgb]{0,0,0}\makebox(0,0)[lb]{\smash{$U$}}}%
    \put(0,0){\includegraphics[width=\unitlength,page=3]{molinosheaf.pdf}}%
  \end{picture}%
\endgroup%

%% file: liftedfoliation.pdf_tex
\begingroup%
  \makeatletter%
  \providecommand\color[2][]{%
    \errmessage{(Inkscape) Color is used for the text in Inkscape, but the package 'color.sty' is not loaded}%
    \renewcommand\color[2][]{}%
  }%
  \providecommand\transparent[1]{%
    \errmessage{(Inkscape) Transparency is used (non-zero) for the text in Inkscape, but the package 'transparent.sty' is not loaded}%
    \renewcommand\transparent[1]{}%
  }%
  \providecommand\rotatebox[2]{#2}%
  \newcommand*\fsize{\dimexpr\f@size pt\relax}%
  \newcommand*\lineheight[1]{\fontsize{\fsize}{#1\fsize}\selectfont}%
  \ifx\svgwidth\undefined%
    \setlength{\unitlength}{376.62093035bp}%
    \ifx\svgscale\undefined%
      \relax%
    \else%
      \setlength{\unitlength}{\unitlength * \real{\svgscale}}%
    \fi%
  \else%
    \setlength{\unitlength}{\svgwidth}%
  \fi%
  \global\let\svgwidth\undefined%
  \global\let\svgscale\undefined%
  \makeatother%
  \begin{picture}(1,0.51016521)%
    \lineheight{1}%
    \setlength\tabcolsep{0pt}%
    \put(0,0){\includegraphics[width=\unitlength,page=1]{liftedfoliation.pdf}}%
    \put(0.15456338,0.13450719){\color[rgb]{0,0,0}\makebox(0,0)[lt]{\lineheight{0}\smash{\begin{tabular}[t]{l}$x^\Yup$\end{tabular}}}}%
    \put(0,0){\includegraphics[width=\unitlength,page=2]{liftedfoliation.pdf}}%
    \put(0.20946093,0.30978846){\color[rgb]{0,0,0}\makebox(0,0)[lt]{\lineheight{0}\smash{\begin{tabular}[t]{l}$x^\Yup$\end{tabular}}}}%
    \put(0,0){\includegraphics[width=\unitlength,page=3]{liftedfoliation.pdf}}%
    \put(0.30788407,0.25797075){\color[rgb]{0,0,0}\makebox(0,0)[lt]{\smash{\begin{tabular}[t]{l}$\mathcal{H}_{x^\Yup}$\end{tabular}}}}%
    \put(0.24171514,0.47602622){\color[rgb]{0.69803922,0.69803922,0.69803922}\makebox(0,0)[lt]{\smash{\begin{tabular}[t]{l}$(\pi^\Yup)^{-1}(x)\cong \mathrm{O}(q)$\end{tabular}}}}%
    \put(0.19953373,0.05208135){\color[rgb]{0,0,0}\makebox(0,0)[lt]{\lineheight{0}\smash{\begin{tabular}[t]{l}$x$\end{tabular}}}}%
    \put(0.35468527,0.08935569){\color[rgb]{0,0,0}\makebox(0,0)[lt]{\smash{\begin{tabular}[t]{l}$U$\end{tabular}}}}%
    \put(0.38105652,0.37697831){\color[rgb]{0,0,0}\makebox(0,0)[lt]{\smash{\begin{tabular}[t]{l}$U^\Yup$\end{tabular}}}}%
    \put(0,0){\includegraphics[width=\unitlength,page=4]{liftedfoliation.pdf}}%
    \put(0.7427317,0.37697831){\color[rgb]{0,0,0}\makebox(0,0)[lt]{\smash{\begin{tabular}[t]{l}$S^\Yup$\end{tabular}}}}%
    \put(0.7427317,0.08935569){\color[rgb]{0,0,0}\makebox(0,0)[lt]{\smash{\begin{tabular}[t]{l}$S$\end{tabular}}}}%
    \put(0,0){\includegraphics[width=\unitlength,page=5]{liftedfoliation.pdf}}%
    \put(0.55219405,0.14003982){\color[rgb]{0,0,0}\makebox(0,0)[lt]{\smash{\begin{tabular}[t]{l}$\rho$\end{tabular}}}}%
    \put(0.55502798,0.32593851){\color[rgb]{0,0,0}\makebox(0,0)[lt]{\smash{\begin{tabular}[t]{l}$\rho^\Yup$\end{tabular}}}}%
    \put(0,0){\includegraphics[width=\unitlength,page=6]{liftedfoliation.pdf}}%
    \put(0.41728991,0.23339982){\color[rgb]{0,0,0}\makebox(0,0)[lt]{\lineheight{0}\smash{\begin{tabular}[t]{l}$\pi^\Yup$\end{tabular}}}}%
    \put(0.67786851,0.23339988){\color[rgb]{0,0,0}\makebox(0,0)[lt]{\lineheight{0}\smash{\begin{tabular}[t]{l}$\pi_S^\Yup$\end{tabular}}}}%
    \put(0,0){\includegraphics[width=\unitlength,page=7]{liftedfoliation.pdf}}%
    \put(0.80424887,0.05389769){\color[rgb]{0,0,0}\makebox(0,0)[lt]{\smash{\begin{tabular}[t]{l}$\bar{x}$\end{tabular}}}}%
    \put(0.93325035,0.25805631){\color[rgb]{0,0,0}\makebox(0,0)[lt]{\smash{\begin{tabular}[t]{l}$\ker(\omega_S)_{\bar{x}^\Yup}$\end{tabular}}}}%
    \put(0,0){\includegraphics[width=\unitlength,page=8]{liftedfoliation.pdf}}%
    \put(0.91375507,0.1226391){\color[rgb]{0,0,0}\makebox(0,0)[lt]{\smash{\begin{tabular}[t]{l}$\bar{x}^\Yup$\end{tabular}}}}%
    \put(0,0){\includegraphics[width=\unitlength,page=9]{liftedfoliation.pdf}}%
    \put(0.85533334,0.31203559){\color[rgb]{0,0,0}\makebox(0,0)[lt]{\smash{\begin{tabular}[t]{l}$\bar{x}^\Yup$\end{tabular}}}}%
    \put(0.07218798,0.29223678){\color[rgb]{0,0,0}\makebox(0,0)[lt]{\smash{\begin{tabular}[t]{l}$L_{x^\Yup}$\end{tabular}}}}%
    \put(0,0){\includegraphics[width=\unitlength,page=10]{liftedfoliation.pdf}}%
  \end{picture}%
\endgroup%

%% file: molinobundle.pdf_tex
\begingroup%
  \makeatletter%
  \providecommand\color[2][]{%
    \errmessage{(Inkscape) Color is used for the text in Inkscape, but the package 'color.sty' is not loaded}%
    \renewcommand\color[2][]{}%
  }%
  \providecommand\transparent[1]{%
    \errmessage{(Inkscape) Transparency is used (non-zero) for the text in Inkscape, but the package 'transparent.sty' is not loaded}%
    \renewcommand\transparent[1]{}%
  }%
  \providecommand\rotatebox[2]{#2}%
  \newcommand*\fsize{\dimexpr\f@size pt\relax}%
  \newcommand*\lineheight[1]{\fontsize{\fsize}{#1\fsize}\selectfont}%
  \ifx\svgwidth\undefined%
    \setlength{\unitlength}{337.19718392bp}%
    \ifx\svgscale\undefined%
      \relax%
    \else%
      \setlength{\unitlength}{\unitlength * \real{\svgscale}}%
    \fi%
  \else%
    \setlength{\unitlength}{\svgwidth}%
  \fi%
  \global\let\svgwidth\undefined%
  \global\let\svgscale\undefined%
  \makeatother%
  \begin{picture}(1,0.75313579)%
    \lineheight{1}%
    \setlength\tabcolsep{0pt}%
    \put(0,0){\includegraphics[width=\unitlength,page=1]{molinobundle.pdf}}%
    \put(0.57718597,0.17548023){\color[rgb]{0,0,0}\makebox(0,0)[lt]{\lineheight{0}\smash{\begin{tabular}[t]{l}leaf closures\end{tabular}}}}%
    \put(0,0){\includegraphics[width=\unitlength,page=2]{molinobundle.pdf}}%
    \put(0.46687497,0.30235228){\color[rgb]{0,0,0}\makebox(0,0)[lt]{\lineheight{0}\smash{\begin{tabular}[t]{l}$\pi^\Yup$\end{tabular}}}}%
    \put(0.66231103,0.4608358){\color[rgb]{0,0,0}\makebox(0,0)[lt]{\lineheight{0}\smash{\begin{tabular}[t]{l}$b$\end{tabular}}}}%
    \put(0.88067543,0.30116603){\color[rgb]{0,0,0}\makebox(0,0)[lt]{\lineheight{0}\smash{\begin{tabular}[t]{l}$\mathrm{O}(q)$-orbits\end{tabular}}}}%
    \put(0.93579984,0.50286465){\color[rgb]{0,0,0}\makebox(0,0)[lt]{\lineheight{0}\smash{\begin{tabular}[t]{l}$\mathrm{O}(q)b(x^\Yup)$\end{tabular}}}}%
    \put(0,0){\includegraphics[width=\unitlength,page=3]{molinobundle.pdf}}%
    \put(0.89349278,0.60408516){\color[rgb]{0.50196078,0.50196078,0.50196078}\makebox(0,0)[lt]{\lineheight{0}\smash{\begin{tabular}[t]{l}$\mathrm{O}(q)b(y^\Yup)$\end{tabular}}}}%
    \put(0,0){\includegraphics[width=\unitlength,page=4]{molinobundle.pdf}}%
    \put(0.46848833,0.17253082){\color[rgb]{0,0,0}\makebox(0,0)[lt]{\lineheight{0}\smash{\begin{tabular}[t]{l}$x^\Yup$\end{tabular}}}}%
    \put(0,0){\includegraphics[width=\unitlength,page=5]{molinobundle.pdf}}%
    \put(-0.00919062,0.12004087){\color[rgb]{0,0,0}\makebox(0,0)[lt]{\lineheight{0}\smash{\begin{tabular}[t]{l}$L_x=\overline{L_x}$\end{tabular}}}}%
    \put(0.11694544,0.02801647){\color[rgb]{0,0,0}\makebox(0,0)[lt]{\lineheight{0}\smash{\begin{tabular}[t]{l}$\overline{L_y}$\end{tabular}}}}%
    \put(0.06879488,0.2223312){\color[rgb]{0.50196078,0.50196078,0.50196078}\makebox(0,0)[lt]{\lineheight{0}\smash{\begin{tabular}[t]{l}$L_y$\end{tabular}}}}%
    \put(0,0){\includegraphics[width=\unitlength,page=6]{molinobundle.pdf}}%
    \put(0.30977597,0.49184022){\color[rgb]{0,0,0}\makebox(0,0)[lt]{\lineheight{0}\smash{\begin{tabular}[t]{l}$x^\Yup$\end{tabular}}}}%
    \put(0.00496012,0.55880315){\color[rgb]{0,0,0}\makebox(0,0)[lt]{\lineheight{0}\smash{\begin{tabular}[t]{l}$\overline{L_x^\Yup}$\end{tabular}}}}%
    \put(0.45481539,0.67658939){\color[rgb]{0.50196078,0.50196078,0.50196078}\makebox(0,0)[lt]{\lineheight{0}\smash{\begin{tabular}[t]{l}$(\pi^\Yup)^{-1}(\overline{L_y})$\end{tabular}}}}%
    \put(0,0){\includegraphics[width=\unitlength,page=7]{molinobundle.pdf}}%
  \end{picture}%
\endgroup%

%% file: liftedkronecker.pdf_tex
\begingroup%
  \makeatletter%
  \providecommand\color[2][]{%
    \errmessage{(Inkscape) Color is used for the text in Inkscape, but the package 'color.sty' is not loaded}%
    \renewcommand\color[2][]{}%
  }%
  \providecommand\transparent[1]{%
    \errmessage{(Inkscape) Transparency is used (non-zero) for the text in Inkscape, but the package 'transparent.sty' is not loaded}%
    \renewcommand\transparent[1]{}%
  }%
  \providecommand\rotatebox[2]{#2}%
  \newcommand*\fsize{\dimexpr\f@size pt\relax}%
  \newcommand*\lineheight[1]{\fontsize{\fsize}{#1\fsize}\selectfont}%
  \ifx\svgwidth\undefined%
    \setlength{\unitlength}{147.71465271bp}%
    \ifx\svgscale\undefined%
      \relax%
    \else%
      \setlength{\unitlength}{\unitlength * \real{\svgscale}}%
    \fi%
  \else%
    \setlength{\unitlength}{\svgwidth}%
  \fi%
  \global\let\svgwidth\undefined%
  \global\let\svgscale\undefined%
  \makeatother%
  \begin{picture}(1,0.82384617)%
    \lineheight{1}%
    \setlength\tabcolsep{0pt}%
    \put(24.67564251,22.91297618){\color[rgb]{0,0,0}\makebox(0,0)[lt]{\lineheight{0}\smash{\begin{tabular}[t]{l} \end{tabular}}}}%
    \put(23.91794237,23.1856771){\color[rgb]{0,0,0}\makebox(0,0)[lt]{\begin{minipage}{0\unitlength}\raggedright \end{minipage}}}%
    \put(14.68611501,5.64812669){\color[rgb]{0,0,0}\makebox(0,0)[lt]{\begin{minipage}{0\unitlength}\raggedright \end{minipage}}}%
    \put(0,0){\includegraphics[width=\unitlength,page=1]{liftedkronecker.pdf}}%
    \put(0.23901517,0.22307876){\color[rgb]{0,0,0}\makebox(0,0)[lt]{\lineheight{0}\smash{\begin{tabular}[t]{l}$0$\end{tabular}}}}%
    \put(0.03956393,0.40256857){\color[rgb]{0,0,0}\makebox(0,0)[lt]{\lineheight{0}\smash{\begin{tabular}[t]{l}$a_2$\end{tabular}}}}%
    \put(0.44690802,0.03448799){\color[rgb]{0,0,0}\makebox(0,0)[lt]{\lineheight{0}\smash{\begin{tabular}[t]{l}$a_1$\end{tabular}}}}%
    \put(0.78524688,0.60113441){\color[rgb]{0,0,0}\makebox(0,0)[lt]{\smash{\begin{tabular}[t]{l}$\widetilde{\f}(\lambda)$\end{tabular}}}}%
  \end{picture}%
\endgroup%

%% file: homothetictransformation.pdf_tex
\begingroup%
  \makeatletter%
  \providecommand\color[2][]{%
    \errmessage{(Inkscape) Color is used for the text in Inkscape, but the package 'color.sty' is not loaded}%
    \renewcommand\color[2][]{}%
  }%
  \providecommand\transparent[1]{%
    \errmessage{(Inkscape) Transparency is used (non-zero) for the text in Inkscape, but the package 'transparent.sty' is not loaded}%
    \renewcommand\transparent[1]{}%
  }%
  \providecommand\rotatebox[2]{#2}%
  \newcommand*\fsize{\dimexpr\f@size pt\relax}%
  \newcommand*\lineheight[1]{\fontsize{\fsize}{#1\fsize}\selectfont}%
  \ifx\svgwidth\undefined%
    \setlength{\unitlength}{174.96313621bp}%
    \ifx\svgscale\undefined%
      \relax%
    \else%
      \setlength{\unitlength}{\unitlength * \real{\svgscale}}%
    \fi%
  \else%
    \setlength{\unitlength}{\svgwidth}%
  \fi%
  \global\let\svgwidth\undefined%
  \global\let\svgscale\undefined%
  \makeatother%
  \begin{picture}(1,0.47196091)%
    \lineheight{1}%
    \setlength\tabcolsep{0pt}%
    \put(0,0){\includegraphics[width=\unitlength,page=1]{homothetictransformation.pdf}}%
    \put(0.74202859,0.09562447){\color[rgb]{0,0,0}\makebox(0,0)[lt]{\lineheight{0}\smash{\begin{tabular}[t]{l}$P$\end{tabular}}}}%
    \put(0.889954,0.26657984){\color[rgb]{0,0,0}\makebox(0,0)[lt]{\lineheight{0}\smash{\begin{tabular}[t]{l}$h_\lambda$\end{tabular}}}}%
    \put(0,0){\includegraphics[width=\unitlength,page=2]{homothetictransformation.pdf}}%
  \end{picture}%
\endgroup%

%% file: holonomyfolation.pdf_tex
\begingroup%
  \makeatletter%
  \providecommand\color[2][]{%
    \errmessage{(Inkscape) Color is used for the text in Inkscape, but the package 'color.sty' is not loaded}%
    \renewcommand\color[2][]{}%
  }%
  \providecommand\transparent[1]{%
    \errmessage{(Inkscape) Transparency is used (non-zero) for the text in Inkscape, but the package 'transparent.sty' is not loaded}%
    \renewcommand\transparent[1]{}%
  }%
  \providecommand\rotatebox[2]{#2}%
  \newcommand*\fsize{\dimexpr\f@size pt\relax}%
  \newcommand*\lineheight[1]{\fontsize{\fsize}{#1\fsize}\selectfont}%
  \ifx\svgwidth\undefined%
    \setlength{\unitlength}{227.1515253bp}%
    \ifx\svgscale\undefined%
      \relax%
    \else%
      \setlength{\unitlength}{\unitlength * \real{\svgscale}}%
    \fi%
  \else%
    \setlength{\unitlength}{\svgwidth}%
  \fi%
  \global\let\svgwidth\undefined%
  \global\let\svgscale\undefined%
  \makeatother%
  \begin{picture}(1,0.40052717)%
    \lineheight{1}%
    \setlength\tabcolsep{0pt}%
    \put(0,0){\includegraphics[width=\unitlength,page=1]{holonomyfolation.pdf}}%
    \put(0.5582973,0.35309875){\color[rgb]{0,0,0}\makebox(0,0)[lt]{\lineheight{0}\smash{\begin{tabular}[t]{l}$E_x$\end{tabular}}}}%
    \put(0,0){\includegraphics[width=\unitlength,page=2]{holonomyfolation.pdf}}%
    \put(0.53662835,0.10416089){\color[rgb]{0,0,0}\makebox(0,0)[lt]{\lineheight{0}\smash{\begin{tabular}[t]{l}$\gamma$\end{tabular}}}}%
    \put(0,0){\includegraphics[width=\unitlength,page=3]{holonomyfolation.pdf}}%
    \put(0.89154063,0.05290431){\color[rgb]{0,0,0}\makebox(0,0)[lt]{\lineheight{0}\smash{\begin{tabular}[t]{l}$L$\end{tabular}}}}%
    \put(0.58156641,0.17613319){\color[rgb]{0.40392157,0.40392157,0.40392157}\makebox(0,0)[lt]{\smash{\begin{tabular}[t]{l}$X(t)$\end{tabular}}}}%
  \end{picture}%
\endgroup%